\documentclass[12pt]{article}
\usepackage{amsmath,amsthm,verbatim,amssymb,amsfonts,amscd,diagrams, graphicx, mathrsfs, hyperref,mathtools,tipa}
\usepackage[usenames,dvipsnames]{color}
\usepackage{amstext}
\usepackage{color}

\theoremstyle{plain}
\newtheorem{theorem}{Theorem}[section]
\newtheorem{corollary}[theorem]{Corollary}
\newtheorem{lemma}[theorem]{Lemma}
\newtheorem{proposition}[theorem]{Proposition}

\theoremstyle{definition}
\newtheorem{definition}[theorem]{Definition}

\theoremstyle{remark}
\newtheorem{remark}[theorem]{Remark}
\newtheorem{example}[theorem]{Example}
\newcommand{\codim}{\text{codim}}
\newarrow{ul}---->
\newarrow{Backwards}<----   

\newcommand{\primeset}[1]{#1}
\newcommand{\id}{\textup{id}}
\newcommand{\onto}{\twoheadrightarrow}

\newcommand{\td}[1]{\tilde{#1}}
\newcommand{\into}{\hookrightarrow}
\newcommand{\Z}{\mathbb{Z}}
\newcommand{\Q}{\mathbb{Q}}
\newcommand{\R}{\mathbb{R}}

\newcommand{\D}{\mathbb{D}}

\newcommand{\bd}{\partial}
\newcommand{\pf}{\pitchfork}

\renewcommand{\H}{\mathbb H}

\newcommand{\mc}[1]{\mathcal{#1}}
\newcommand{\ms}[1]{\mathscr{#1}}

\newcommand{\dlim}{\varinjlim}

\newcommand{\Hom}{\text{Hom}}

\newcommand{\Ext}{\text{Ext}}
\newcommand{\mf}{\mathfrak}

\newcommand{\im}{\text{im}}
\newcommand{\cok}{\text{cok}}

\renewcommand{\P}{\mathbb P}
\newcommand{\q}{\mathfrak q}

\newarrow{Onto}----{>>}
\newarrow{Equals}=====

\newcommand{\ttau}{\text{\texthtt}}

\newcommand{\xr}{\xrightarrow}

\begin{document}

\title{Generalizations of intersection homology and perverse sheaves with duality over the integers}
\author{Greg Friedman\thanks{This work was partially funded by the National Science Foundation under Grant Number (DMS-1308306) to Greg Friedman
and by a grant from the Simons Foundation (\#209127 to Greg Friedman).} }

\date{March 8, 2019}

\maketitle
\tableofcontents

\bigskip



\begin{abstract}
We provide a generalization of the Deligne sheaf construction of intersection homology theory, and a corresponding generalization of Poincar\'e  duality on pseudomanifolds, such that the Goresky-MacPherson, Goresky-Siegel, and Cappell-Shaneson duality theorems all arise as special cases. Unlike classical intersection homology theory, our duality theorem holds with ground coefficients in an arbitrary PID and with no local cohomology conditions on the underlying space. Self-duality does require local conditions, but our perspective leads to a new class of spaces more general than the Goresky-Siegel IP spaces on which upper-middle perversity intersection homology is self dual. We also examine torsion-sensitive t-structures and categories of perverse sheaves that contain our torsion-sensitive Deligne sheaves as intermediate extensions. 
\end{abstract}

\section{Introduction}

In this paper we formulate a modified version of the ``Deligne sheaf'' construction, which  was introduced by Goresky and MacPherson \cite{GM2} as a sheaf-theoretic approach to intersection homology on stratified pseudomanifolds. The perversity parameters in our theory assign to each stratum not only a truncation degree but also a set of primes in a fixed ground PID that are utilized in a variant ``torsion-tipped truncation.'' The resulting ``torsion-sensitive Deligne sheaves'' admit a generalized Poincar\'e duality theorem on stratified pseudomanifolds of which the Goresky-MacPherson \cite{GM2}, Goresky-Siegel \cite{GS83}, and Cappell-Shaneson \cite{CS91} duality theorems for intersection homology all occur as special cases. Our duality theorem holds with no local cohomology conditions such as the Witt condition or the Goresky-Siegel locally torsion free condition that are typically required on the underlying space. We will see that the existence of self-dual Deligne sheaves does require local conditions, though one consequence of our perspective is the discovery of a new class of spaces, more general than the Goresky-Siegel IP spaces, on which the standard upper-middle perversity intersection homology is self-dual. We also study torsion-sensitive t-structures and categories of perverse sheaves in which our torsion-sensitive Deligne sheaves arise as intermediate extensions of the appropriate analogues of local systems on the regular strata. 

In order to further explain these result and their context, we begin by recalling some historical background.

\paragraph{Background.}
In \cite{GM1}, Goresky and MacPherson introduced intersection homology for a closed oriented $n$-dimensional PL stratified pseudomanifold $X$. These homology groups, denoted $I^{\bar p}H_*(X)$, depend on a \emph{perversity parameter}\footnote{In early work on intersection homology, e.g. \cite{GM1, GM2,Bo,GS83}, perversities were only considered that took the same value on all strata of the same codimension. We employ a slightly revisionist history in this introduction by stating the theorems in a form more consonant with  more general notions of perversity; see \cite{GBF35, GBF23, GBF26}.}  $$\bar p:\{\text{singular strata of $X$}\}\to \Z.$$ They showed that if $D\bar p$ is the complementary perversity, i.e.\  $\bar p(Z)+D\bar p(Z)=\codim(Z)-2$ for all singular strata $Z$, then there is an intersection pairing $$I^{\bar p}H_i(X)\otimes I^{D\bar p}H_{n-i}(X)\to \Z$$ induced at the level of PL chains that becomes nonsingular after tensoring with $\Q$. This provides an important generalization of Poincar\'e duality to non-manifold spaces. 

In \cite{GM2}, and in the broader context  of \emph{topological} stratified pseudomanifolds, Goresky and MacPherson further refined this intersection homology version of Poincar\'e duality  into the statement that there is a quasi-isomorphism of  sheaf complexes\footnote{We will generally omit the generic indexing decoration for complexes except where needed for clarity.} over $X$:
\begin{equation}\label{E: duality}
\mc P_{\bar p}\cong \mc D\mc P_{D\bar p}[-n].
\end{equation}
Here $\mc P_{\bar p}$ denotes the ``Deligne sheaf complex'' with perversity $\bar p$, which is an iteratively-constructed sheaf complex characterized by nice axioms and  whose hypercohomology groups give intersection homology via $\H^i_c\left(X;\mc P_{\bar p}\right)\cong I^{\bar p}H_{n-i}(X)$. The symbol $\mc D$ here denotes the Verdier dualizing functor, $[m]$ denotes a shift so that $(\mc S[m])^i\cong \mc S^{i+m}$ for a sheaf complex $\mc S$, and $\cong$ denotes quasi-isomorphism, i.e\ isomorphism in the derived category; we use the convention throughout that $\mc D\mc S[m]$ means $(\mc D\mc S)[m]$ and not $\mc D(\mc S[m])=(\mc D\mc S)[-m]$.
The stratified pseudomanifold $X$ is no longer required to be compact, but for duality to hold the ground ring of coefficients is required in \cite{GM2} to be a field. 

In \cite{GS83}, Goresky and Siegel explored the duality properties of Deligne sheaves with coefficients in a principal ideal domain. This requires the consideration of torsion information. They demonstrated that in this setting one cannot hope for a version of \eqref{E: duality} in complete generality. The obstruction occurs in the form of torsion in certain local intersection homology groups at the singular points of $X$. This led to the definition of a locally $\bar p$-torsion-free space. More precisely, a stratified pseudomanifold is locally $\bar p$-torsion-free (with respect to the PID $R$) if for each $x$ in each singular stratum $Z$ of codimension $k$ the torsion subgroup of $I^{\bar p}H_{k-2-\bar p(Z)}(L;R)$ vanishes, where $L$ is the link of $x$. If $X$ is such a space, then  \eqref{E: duality} holds with coefficients in $R$, leading to certain other nice ``integral'' properties of duality and homology, such as nonsingular linking pairings and a universal coefficient theorem. 

In \cite{CS91}, Cappell and Shaneson proved a ``superduality'' theorem, which holds in a situation that can be considered somewhat the opposite of that of Goresky and Siegel. Cappell and Shaneson showed that if the stratified pseudomanifold $X$ possesses the property that all local intersection homology groups are torsion, then 
\eqref{E: duality} holds between ``superdual'' perversities $\bar p$ and $\bar q$, meaning $\bar p(Z)+\bar q(Z)=\codim(Z)-1$ for all singular strata $Z$ of $X$. While this statement seems more drastic than that of Goresky-Siegel in terms of the number of degrees for which there is a local intersection homology condition, it follows from the proof that one could  impose this ``torsion only'' condition in just one degree per link\footnote{It is also worth noting that Cappell and Shaneson use local coefficient systems on the complement of the singular locus throughout \cite{CS91} so that their local intersection homology groups are akin to Alexander modules of knots. This explains how it is possible for each local intersection homology group to be torsion, even in degree zero. }.

Deligne sheaves with PID coefficients can also be considered from the broader perspective of the perverse sheaves of Be{\u\i}linson, Bernstein, and Deligne (BBD) \cite{BBD}. While most of \cite{BBD} concerns perverse sheaves with ground ring a field, \cite[Complement 3.3]{BBD} contains the definition of the following t-structure on the derived category $D(X,\Z)$ of sheaves of abelian groups on a space $X$ by taking torsion into account:
\begin{align*}
{}^{n^+}D^{\leq 0}&=\{K\in D(X,\Z)\mid \forall x\in X, H^i(K_x)=0\text{ for $i>1$ and $H^1(K_x)\otimes \Q=0$}\}\\
{}^{n^+}D^{\geq 0}&=\{K\in D(X,\Z)\mid \forall x\in X, H^i(K_x)=0\text{ for $i<0$ and $H^0(K_x)$ is torsion free}\}.
\end{align*}
If $X$ is stratified and equipped with a perversity $\bar p$, then one can glue shifts of such t-structures over strata to obtain a t-structure $({}^{\bar p^+}D^{\leq 0}(X,\Z),{}^{\bar p^+}D^{\geq 0}(X,\Z))$, generalizing the t-structures of \cite[Definition 2.1.2]{BBD}. Verdier duality interchanges this t-structure with the standard perverse t-structure $({}^{D\bar p}D^{\leq 0}(X,\Z),{}^{D\bar p}D^{\geq 0}(X,\Z))$. Torsion in t-structures is also considered abstractly in \cite{Ju09}.

\paragraph{Results.}
Our first principal goal in this paper is to introduce generalized Deligne sheaves for which  a version of \eqref{E: duality} holds over a PID for any topological stratified pseudomanifold. 
The construction will incorporate certain local torsion information in a manner analogous to the above BBD t-structure, but rather than asking our Deligne sheaves to be either ``all torsion'' or ``no torsion'' in certain degrees, we allow mixed situations by  taking as part of our perversity information a set of primes on each stratum. Then, just like the classical Deligne construction, our ``torsion-sensitive Deligne sheaves'' will involve certain cohomology truncations, but for us the perversity information will determine both the truncation degree and the types of torsion that can occur in the cohomology at that degree. We will see that Verdier duality then interchanges the set of primes on a particular stratum with the complementary set of primes. This leads to some interesting duality results, even for relatively simple spaces. 

More specifically, in order to implement our construction, we generalize the notion of perversity from that of a function
$$\bar p:\{\text{singular strata of X}\}\to \Z$$
to that of a function
$$\vec p=(\vec p_1,\vec p_2):\{\text{singular strata of X}\}\to \Z\times \P(\primeset{P}(R)) ,$$
where $\primeset{P}(R)$ is the set of primes (up to unit) of the PID $R$ and $ \P(\primeset{P}(R))$ is its power set (the set of all subsets); note that $\vec p_1$ is itself a perversity in the standard sense. 
We refer to such functions $\vec p$ as ``torsion-sensitive perversities'' or ``ts-perversities,'' and we construct a ``torsion-sensitive Deligne sheaf complex,'' or ``ts-Deligne sheaf,''   $\ms P_{\vec p}$ by a modification of the standard iterated ``pushforward and truncate'' Deligne construction.  
In the case that $\vec p_2(Z)=\emptyset$ for all singular strata $Z$, then $\ms P_{\vec p}$ is quasi-isomorphic to the classical Deligne sheaf $\mc P_{\vec p_1}$. We will see in Section \ref{S: axiomatics} that the ts-Deligne sheaves are characterized by a generalization of the Goresky-MacPherson axioms.  

The complementary ts-perversity $D\vec p$ to a ts-perversity $\vec p$ is defined by letting $(D\vec p)_1$ be the dual perversity to $\vec p_1$ in the standard sense, i.e.\ $(D\vec p)_1(Z)=\codim(Z)-2-\vec p_1(Z)$, while $(D\vec p)_2(Z)$ is defined to be the complement of $\vec p_2(Z)$ in $\primeset{P}(R)$. Our generalized duality theorem, whose precise technical statement can be found in Theorem \ref{T: duality}, then has the form 

\begin{equation}\label{E: ts-duality}
\ms P_{\vec p}\cong \mc D\ms P_{D\vec p}[-n],
\end{equation}
with no local cohomology conditions on the underlying stratified pseudomanifold $X$.
The duality theorems of Goresky-MacPherson, Goresky-Siegel, and Cappell-Shaneson all occur as special cases:
\begin{enumerate}

\item With coefficients over a field, $\ms P_{\vec p}\cong \mc P_{\vec p_1}$, the perversity $\vec p_1$ Deligne sheaf. Then \eqref{E: ts-duality} reduces to the Goresky-MacPherson version of \eqref{E: duality}.

\item When $X$ is locally $\vec p_1$-torsion-free over the ground ring $R$, again $\ms P_{\vec p}\cong \mc P_{\vec p_1}$ and  \eqref{E: ts-duality} reduces to the Goresky-Siegel version of \eqref{E: duality}.

\item If the local intersection homology of $X$ is all torsion and $\vec p_2(Z)=\primeset{P}(R)$ for all singular strata $Z$, then  $\ms P_{\vec p}\cong \mc P_{\vec p_1+1}$, where $\vec p_1+1$ is the perversity defined by $(\vec p_1+1)(Z)=\vec p_1(Z)+1$. Also, since this forces $D\vec p_2(Z)=\emptyset$ for all $Z$,  we have $\ms P_{D\vec p}\cong \mc P_{D\vec p_1}$ and \eqref{E: ts-duality} reduces to the Cappell-Shaneson version of \eqref{E: duality}.

\end{enumerate}

After demonstrating these general duality results, we turn in Section \ref{S: self dual} to the important question of when our ts-Deligne sheaves might be self dual, i.e.\ when are $\ms P_{\vec p}$ and $\mc D\ms P_{\vec p}$ quasi-isomorphic up to degree shifts? Such situations lead to further key invariants such as signatures. In order to achieve such self-duality, local conditions on the cohomology of links come back into play, and we recover such torsion and torsion-free conditions as were observed in \cite{GS83,Pa90,CS91}. However, we also make what we believe to be a new observation: that the Goresky-Siegel torsion-free conditions on odd-codimension strata and the Cappell-Shaneson torsion conditions on even-codimension strata can be fused to define a class of spaces more general than the well-known IP spaces \cite{GS83,Pa90} on which the classical upper-middle perversity intersection homology $I^{\bar n}H_*(X;\Z)$ admits self-duality. We also explore a class of spaces on which the lower-middle perversity intersection homology groups $I^{\bar m}H_*(X;\mc E)$ are self dual assuming that $\mc E$ is a coefficient system of torsion modules, generalizing another observation from \cite{CS91}. 

Following our study of ts-Deligne sheaves, we consider the more general context for such sheaf complexes by introducing t-structures on the derived category of sheaf complexes on $X$ whose associated ts-perverse sheaves (i.e.\ the objects living in the heart of the t-structure) similarly depend upon varying truncation degrees and torsion information on each stratum. Analogously to classical results about Deligne sheaves, our ts-Deligne sheaves turn out to be the intermediate extensions in these t-structures (Proposition \ref{P: deligne is ie}), and the general machinery of t-structures then leads to an alternative proof of Theorem \ref{T: duality}, our main duality theorem --- see Remark \ref{R: other dual}. We also show in Example \ref{E: point noeth} that in general the hearts of these t-structures are neither Noetherian nor Artinian categories, in contrast to the case of perverse sheaves with complex coefficients on complex varieties which are well known to be both \cite[Theorem 4.3.1]{BBD}. We do show, however, in Theorem \ref{T: noether} that the hearts are Artinian when we allow all torsion and Noetherian when we allow none. 

The final section of the paper explores in some detail the case of a pseudomanifold with just one isolated singularity and relates the abstract duality results of the preceding sections with more hands-on computations involving the homology groups of the manifold with boundary obtained by removing a neighborhood of the singularity. We see in this case that some of our intersection homology results correspond to well-known facts from manifold theory but that others are not so obvious from a more classical point of view.

\medskip

\paragraph{The main technical idea.}
Before concluding the introduction, let us attempt to provide some brief, but more technical, motivation for why the Goresky-Siegel or Cappell-Shaneson conditions are necessary for duality over a PID in the classical Deligne-sheaf formulation of intersection homology and how those conditions motivate our definition of ts-Deligne sheaves $\ms P_{\vec p}$. To simplify this discussion, we work over $\Z$ and suppose perversity values depend only on codimension as in \cite{GM1}.  We also will not attempt to get too deeply into technical details here; we will limit ourselves to presenting the basic idea. 

Recall that the Deligne sheaf is defined by a process of consecutive pushforwards and truncations. In the original Goresky-MacPherson formulation \cite{GM2}, if $X=X^n\supset X^{n-2}\supset \cdots$ is a stratified pseudomanifold and  $\mc P_{\bar p}(k)$ is the Deligne sheaf defined over $X-X^{n-k}$ (or, if $k=2$, $\mc P_{\bar p}(2)$ is a locally constant sheaf of coefficients over $\Z$), then one extends  $\mc P_{\bar p}(k)$ to  $\mc P_{\bar p}(k+1)$ on $X-X^{n-k-1}$ as $$\mc P_{\bar p}(k+1)=\tau_{\leq \bar p(k)}Ri_{k*}\mc P_{\bar p}(k),$$
where $i_k$ is the inclusion $X-X^{n-k}\into X-X^{n-k-1}$, $\tau$ is the sheaf complex truncation functor, and $\bar p(k)$ is the common value of $\bar p$ on all strata of codimension $k$. In particular, it follows that at a point $x\in X^{n-k}$, $k\geq 2$, with link $L$,  we have $H^i((\mc P_{\bar p})_x)=0$ for $i>\bar p(k)$, while for $i\leq \bar p(k)$, we have $H^i((\mc P_{\bar p})_x)=\H^{i}(L;\mc P_{\bar p})$, the hypercohomology of the link. 

On the other hand, letting $\bar q=D\bar p$ and using the properties of Verdier duality \cite[Section V.7.7]{Bo}, one obtains a universal coefficient-flavored calculation for the cohomology of the dual that looks like this\footnote{See Section \ref{S: duality} for more details.}:
$$H^i((\mc D\mc P_{\bar q}[-n])_x)\cong  \Hom(H^{n-i}(f^!_x\mc P_{\bar q}),\Z)\oplus \Ext(H^{n-i+1}(f^!_x\mc P_{\bar q}),\Z),$$
where $f_x:x\to X$ is the inclusion. If we were working instead with  coefficients in a field  $F$, the $\Ext$ term would vanish, and so $H^i((\mc D\mc P_{\bar q}[-n])_x)\cong  \Hom(H^{n-i}(f^!_x\mc P_{\bar q}),F)$. One of the steps in proving the Goresky-MacPherson duality isomorphism \eqref{E: duality} then involves showing\footnote{For the technicalities, see \cite{Bo}, in particular Step $(b)$ of the proof of Theorem V.9.8 and the $(2b)$ implies $(1'b)$ part of the proof of Proposition V.4.9.}  that $H^{n-i}(f^!_x\mc P_{\bar q})=0$ for $i>\bar p(k)$, which is compatible with our computation for $H^i((\mc P_{\bar p})_x)$. 
With a bit more work, one then shows that the sheaf complexes  $\mc D\mc P_{\bar q}[-n]$ and $\mc P_{\bar p}$ are in fact quasi-isomorphic.

However, with $\Z$ coefficients  we have the following problem: From the truncations in the definition of $\mc P_{\bar p}$, we must have $H^{\bar p(k)+1}((\mc P_{\bar p})_x)=0$. Meanwhile, from the duality computation, we have 
$$H^{\bar p(k)+1}((\mc D\mc P_{\bar q}[-n])_x)\cong  \Hom(H^{n-(\bar p(k)+1)}(f^!_x\mc P_{\bar q}),\Z)\oplus \Ext(H^{n-(\bar p(k)+1)+1}(f^!_x\mc P_{\bar q}),\Z).$$ The observation of the last paragraph that  $H^{n-i}(f^!_x\mc P_{\bar q})=0$ for $i>\bar p(k)$ holds for any PID coefficients and implies that $H^{n-(\bar p(k)+1)}(f^!_x\mc P_{\bar q})=0$. However, it will not generally be true that $H^{n-\bar p(k)}(f^!_x\mc P_{\bar q})=0$, and so 
$$H^{\bar p(k)+1}((\mc D\mc P_{\bar q}[-n])_x)\cong  \Ext(H^{n-\bar p(k)}(f^!_x\mc P_{\bar q}),\Z)$$ 
might not be zero, in which case we could not have $H^{\bar p(k)+1}((\mc D\mc P_{\bar q}[-n])_x)\cong H^{\bar p(k)+1}((\mc P_{\bar p})_x)$. However, $H^{n-\bar p(k)}(f^!_x\mc P_{\bar q})$  will be finitely generated and 
if it  were also torsion-free, then $H^{\bar p(k)+1}((\mc D\mc P_{\bar q}[-n])_x)$ would indeed vanish! It turns out that one could then continue on to complete the argument that $\mc P_{\bar p}$ and $\mc D\mc P_{\bar q}[-n]$ are quasi-isomorphic. This is the source of the Goresky-Siegel condition, which, with a bit more computation, implies that $H^{n-\bar p(k)}(f^!_x\mc P_{\bar q})$ is indeed torsion free. See \cite{GS83} for details (and be mindful of the different indexing convention!).

The Cappell-Shaneson computation is similar but ``from the other side.'' If we extend our perversity from $\bar p$ to $\bar p+1$ (but keep $\bar q$ the same), then it is acceptable to have $H^{\bar p(k)+1}((\mc D\mc P_{\bar q}[-n])_x)$ not vanish, but as we have seen, it must be isomorphic to the torsion group  $\Ext(H^{n-\bar p(k)}(f^!_x\mc P_{\bar q}),\Z)$, as $\Hom(H^{n-(\bar p(k)+1)}(f^!_x\mc P_{\bar q}),\Z)$ still vanishes. This is problematic if $H^{\bar p(k)+1}((\mc D\mc P_{\bar q}[-n])_x)$ is not all torsion, but if we assume it is torsion, then again this turns out to be enough to dodge catastrophe and allow the quasi-isomorphism argument to go through. The Cappell-Shaneson torsion condition ensures this.

The preceding arguments  lead one to the thought that it might be possible to make the duality quasi-isomorphism arguments ``come out alright'' provided one is able to exercise sufficient control on when (and what kind of) torsion is allowed to crop up in local intersection homology groups and when it is not. Indeed, such control at the level of spaces is precisely the idea behind the Goresky-Siegel and Cappell-Shaneson conditions. We will pursue an alternative route by allowing the pseudomanifold to be arbitrary while instead building such torsion control into the definition of the Deligne sheaf complex. This is precisely what the second component $\vec p_2$ of our torsion sensitive perversities will do: it is a switch indicating what kind of torsion the strata are permitted to have in their local intersection homology groups at the cut-off dimension. This information is assimilated into the ts-Deligne sheaf via a modified ``torsion-tipped'' version of the truncation functor that, rather than simply cutting off all stalk cohomology of a sheaf complex at a given dimension, permits a certain torsion subgroup of the stalk cohomology to continue to exist for one dimension above the cutoff, analogously to the Be{\u\i}linson-Bernstein-Deligne construction. The ts-Deligne sheaf then incorporates this torsion-tipped truncation according to the instructions given by $\vec p_2$. 

This is the main idea of the paper. The rest is details!

\paragraph{A conjecture about topological invariance.}
In addition to duality, another key feature of classical intersection homology is topological invariance, i.e.\ independence of pseudomanifold stratification. This property does not hold in general but rather requires spaces with no codimension one strata and  perversities that depend only on codimension and satisfy $\bar p(2)=0$ and $\bar p(k)\leq \bar p(k+1)\leq \bar p(k)+1$ for  $k\geq 2$; see \cite[Section 4]{GM2}, \cite[Section V.4]{Bo}, \cite[Section 5.5]{GBF35}. In our torsion-sensitive context, even this condition in the first component $\vec p_1$ of a ts-perversity $\vec p$ is not sufficient, as one can see by considering toy examples on spaces of the form $\R\times cL\cong c(SL)$, with $S$ denoting suspension and letting these homeomorphic spaces have the two different implied stratifications. However, based on such examples, we conjecture that topological invariance does hold for ts-perversities $\vec p$ such that $\vec p_1$ satisfies the above conditions, $\vec p_2(k+1)\supset \vec p_2(k)$ when $\vec p_1(k+1)=\vec p_1(k)$, and $\vec p_2(k+1)\subset \vec p_2(k)$ when $\vec p_1(k+1)=\vec p_1(k)+1$.

\paragraph{Outline of the paper.}
Section \ref{S: algebra} contains some algebraic preliminaries. 
In Section \ref{S: TTT}, we introduce the torsion-tipped truncation functor. Then in Section \ref{S: TSD}, we construct the torsion-sensitive Deligne sheaf, demonstrate that it satisfies a set of characterizing axioms generalizing the Deligne sheaf axioms of Goresky and MacPherson \cite{GM2}, and prove our duality theorem, Theorem \ref{T: duality}. Self-duality is treated in Section \ref{S: self dual}. Section \ref{S: perverse} contains our study of torsion-sensitive t-structures and ts-perverse sheaves.
 Finally, in Section \ref{S: man}, we conclude with an example, computing the hypercohomology groups of ts-Deligne sheaves for pseudomanifolds with isolated singularities and showing in some cases how their duality relates to classical Poincar\'e-Lefschetz duality for manifolds with boundary. We obtain some results concerning manifold theory that are not so obvious from more direct approaches.

\paragraph{Prerequisites and assumptions.} 

We assume the reader has some background in intersection homology and the derived category of sheaf complexes, which we work in throughout, along the lines of Goresky-MacPherson \cite{GM2}, Borel \cite{Bo}, or Banagl \cite{BaIH}. We will also freely utilize arbitrary perversity functions, for which background can be found in \cite{GBF35, GBF23, GBF26}. Accordingly, we also allow stratified pseudomanifolds to possess codimension one strata. Some knowledge of t-structures and perverse sheaves may be useful, but not critical, in Section \ref{S: perverse}; further references are listed there.

\paragraph{Acknowledgments.} 
The author is indebted to the referees for suggestions that substantially improved this paper. The original manuscript contained only the material from the Deligne sheaf point of view, and it was one of the referees who suggested a detailed exploration from the perspective of perverse sheaves. This referee also recommended that it should be clarified what the self-dual torsion-sensitive perversities are, leading to Section \ref{S: self dual} and the results contained therein. A second referee made several further suggestions that greatly improved the exposition.

In addition, the author thanks Jim McClure for conversations and questions that led to the initial ideas that evolved into this paper. Thanks are also due to Sylvain Cappell for pointing out that the sets of primes don't have to be all or nothing and for introducing the author to intersection homology in the first place. Thanks to Scott Nollet and Loren Spice for many helpful conversations.

\section{Algebraic preliminaries}\label{S: algebra}

Throughout the paper, $R$ will denote a principal ideal domain (PID). We let $\primeset{P}(R)$ be the set of equivalence classes of primes of $R$, where two primes $p,q\in R$ are equivalent if $p=uq$ for some unit $u\in R$. In practice, we fix a representative of each equivalence class and identify the class with its representative prime, i.e. we think of $\primeset{P}(R)$ as a set of specific primes, one from each equivalence class. 

Let $\wp\subset \primeset{P}(R)$ be a set of primes of $R$. Define the \emph{span of $\wp$}, $S(\wp)$, to be the set of products of powers of primes in $\wp$, i.e.\ $S(\wp)=\left\{n\in R\middle| n=\prod_{i=1}^s p_i^{m_i}, \text{ where }p_i\in\wp\text{ and } m_i,s\in \Z_{\geq 0}\right\}.$  We allow $s=0$ (which is necessary when $\wp=\emptyset$), and in this case we interpret the product to be $1$. Then $S(\emptyset)=\{1\}$, and $\{1\}\subset S(\wp)$ for all $\wp$, so, in particular, $S(\wp)$ is never empty. Also, notice that $0\notin S(\wp)$ for any $\wp$. 

If $M$ is an $R$-module, we define $T^{\wp}M$ to be the submodule of elements of $M$ annihilated by elements of $S(\wp)$, i.e.\ $T^{\wp}M=\{x\in M\mid \exists n\in S(\wp)\text{ such that }nx=0\}.$

\begin{definition}\label{D: torsion}
We will refer to $T^{\wp}M$ as the \emph{$\wp$-torsion submodule} of $M$. 
If $T^{\wp}M=M$ we will say that $M$ is \emph{$\wp$-torsion}, and if $T^{\wp}M=0$ we will say that $M$ is \emph{$\wp$-torsion free}. 
\end{definition}

Note that if $\wp=\emptyset$, then  $T^{\wp}M=0$ for any $M$. So an $\emptyset$-torsion module is trivial, and every module is $\emptyset$-torsion free.

Recall that any finitely-generated $R$-module $M$ for a PID $R$ can be written as a direct sum $M\cong R^{r_M}\oplus \bigoplus_p M(p)$, where $r_M$ is the rank of $M$, the primes $p$ range over $\primeset{P}(R)$ (though there will only be a finite number of non-trivial summands), and each $M(p)$ is isomorphic to a direct sum $R/(p^{\nu_1})\oplus\cdots\oplus R/(p^{\nu_{s_p}})$ with each $\nu_i$ a positive integer \cite[Section III.7]{LANG}. Note that $T^{\wp}M\cong \bigoplus_{p\in \wp} M(p)$ and $M/T^{\wp}M\cong R^{r_M}\oplus \bigoplus_{p\notin\wp} M(p)$. Furthermore, by elementary computations, $\Hom(M,R)\cong R^{r_M}$ and $\Ext(M,R)\cong T^{P(R)}M$.

\section{Torsion-tipped truncation}\label{S: TTT}

In this section we define our torsion-tipped truncation functors. 
Given an integer $k$ and a set $\wp$ of primes of the PID $R$, we first define a global endofunctor $\ttau^\wp_{\leq k}$ in the category of cohomologically indexed complexes of sheaves of $R$-modules on a space $X$. Then in Subsection \ref{S: local trunc} we will define a localized version of this truncation functor that can truncate in different degrees and with different primes on different subsets of $X$.

We begin by defining $\ttau^{\wp}_{\leq k}$ as an endofunctor of presheaves using the following notion of weak $\wp$-boundaries. 

\begin{definition}
Let $A$ be a presheaf complex on $X$ with boundary map $d$. Let $W^{\wp}A^j$ be the \emph{presheaf of weak $\wp$-boundaries in degree $j$}, defined by
 $$W^{\wp}A^j(U)=\{s\in A^j(U)\mid ds=0\text{ and } \exists n\in S(\wp)\text{ such that } ns\in \im(d: A^{j-1}(U)\to A^j(U))\}.$$ 
In particular, if  $\wp=\emptyset$ then $W^{\emptyset}A^j(U)=\im(d: A^{j-1}(U)\to A^j(U))$. Furthermore, if  $\wp\subset \wp'$ then $W^\wp A^j(U)\subset W^{\wp'} A^j(U) $. 
\end{definition}

The following lemma is elementary:

\begin{lemma}\label{L: pres}
The assignment $U\to W^{\wp}A^j(U)$ for open sets $U\subset X$ determines a presheaf. 
\end{lemma}

\begin{definition}\label{D: presheaf trunc}
We define the \emph{$\wp$-torsion-tipped truncation functor} on a presheaf complex $A$ by   
\begin{equation*}
(\ttau^{\wp}_{\leq k}A)^i=
\begin{cases}
0, & i>k+1,\\
W^{\wp}A^{k+1}, &i=k+1,\\
A^i, &i\leq k.
\end{cases}
\end{equation*}
If $f:A\to B$ is a map of presheaf complexes then $\ttau^{\wp}_{\leq k}f$ is given by restriction of $f$.
\end{definition}

\begin{lemma}\label{L: pretrunc}
$\ttau^{\wp}_{\leq k}$ is an endofunctor of presheaf complexes over $X$. Furthermore, if $k\leq k'$ and $\wp\subset \wp'$ there are monomorphisms $\ttau^{\wp}_{\leq k}A\into \ttau^{\wp'}_{\leq k'}A \into A$, natural in $A$. 
\end{lemma}
\begin{proof}
If $A$ is a presheaf complex then $\ttau^{\wp}_{\leq k}A$ is a presheaf complex: We have a presheaf in each degree by Lemma \ref{L: pres}. Furthermore, as already observed, $\im(d:A^{k}\to A^{k+1})=W^\emptyset A^{k+1} \subset W^{\wp}A^{k+1}$, so one readily checks that restriction commutes with boundaries. 

Now suppose $f:A\to B$ is a chain map of presheaf complexes. Then if $du=ns$ for some $u\in A^k$, $s\in A^{k+1}$ with $ds=0$ and $n\in S(\wp)$, we see that $df(s)=f(ds)=0$ and $df(u)=f(du)=f(ns)=nf(s)$, so $s\in W^{\wp}A^{k+1}$ implies $f(s)\in W^{\wp}B^{k+1}$. So $f$ induces in the obvious way a map $\ttau^{\wp}_{\leq k}(f)$, and $\ttau^{\wp}_{\leq k}$ is a functor. 

For the second statement, the inclusions are clear and the naturality follows.
\end{proof}

\begin{lemma}\label{L: presheaf ttau}
For any open $U\subset X$,
\begin{equation*}
H^i\left(\ttau^{\wp}_{\leq k} A(U)\right)\cong
\begin{cases}
0, & i>k+1,\\
T^{\wp}H^{k+1}(A(U)), & i=k+1,\\
H^i(A(U)),& i\leq k.
\end{cases}
\end{equation*}
Furthermore, the homology isomorphisms or torsion submodule isomorphisms, in the respective degrees, are induced by the inclusion $\ttau^{\wp}_{\leq k}A \into A$. 
\end{lemma}
\begin{proof}
This is trivial in all degrees save $i=k+1$. The chain inclusion $\ttau^{\wp}_{\leq k} A(U)\to A(U)$, induces a map $f: H^{k+1}(\ttau^{\wp}_{\leq k} A(U))\to H^{k+1}(A(U))$. 
If $s\in W^{\wp}A^{k+1}(U)$, then for some $n\in S(\wp), u\in A^k$, we have $ns=du$, so the image of $f$ must lie in $T^{\wp}H^{k+1}(A(U))$. Conversely, given a cycle $s$ representing an element of $T^{\wp}H^{k+1}(A(U))$, by the definition of  $T^\wp H^{k+1}(A(U))$ there must be some $n\in S(\wp)$ and $u\in A^k(U)$  such that
$ns=du$. Thus $f$ is surjective. Now suppose $s\in W^{\wp}A^{k+1}(U)$ and $f(s)=0$ in $H^{k+1}(A(U))$. Then there is a $u\in A^{k}(U)$ such that $du=s$. But then this relation also holds in $\ttau^{\wp}_{\leq k} A(U)$ and $s$ represents $0$ in $H^{k+1}(\ttau^{\wp}_{\leq k} A(U))$. Thus $f$ is an isomorphism $H^{k+1}(\ttau^{\wp}_{\leq k} A(U))\to
T^{\wp}H^{k+1}(A(U))$.
\end{proof}

\begin{remark}\label{R: empty2}
If $\wp=\emptyset$, then the lemma demonstrates that $\ttau^{\wp}_{\leq k} A(U)$ has the cohomology we obtain from the standard truncation functor $\tau_{\leq k} A(U)$ with
\begin{equation*}
(\tau_{\leq k} A(U))^i=\begin{cases}
0,&i>k\\
\ker(d),&i=k\\
 A^i(U), &i<k.
\end{cases}
\end{equation*}
In fact, it is not difficult to see that this cohomology isomorphism is induced by an inclusion $\tau_{\leq k} A(U)\into \ttau^{\wp}_{\leq k} A(U)$.
\end{remark}

We can now extend $\ttau^{\wp}_{\leq k}$ to a functor of sheaves by sheafification:

\begin{definition} 
If $\ms S$ is a sheaf complex on $X$, define the \emph{$\wp$-torsion-tipped truncation} $\ttau^{\wp}_{\leq k} \ms S$ as the sheafification of the presheaf $U\to \ttau^{\wp}_{\leq k} (\ms S(U))$.  If $f:\ms S\to \ms T$ is a map of sheaf complexes, then $\ttau^{\wp}_{\leq k}f: \ttau^{\wp}_{\leq k}\ms S\to \ttau^{\wp}_{\leq k}\ms T$ is the map induced by sheafification of the presheaf map $\ttau^{\wp}_{\leq k}f$ of Definition \ref{D: presheaf trunc}.  
\end{definition}

\begin{remark}\label{R: mono}
Due to Lemma \ref{L: pretrunc} and the exactness of sheafification there are natural monomorphisms  $\ttau^{\wp}_{\leq k}\ms S \into \ttau^{\wp'}_{\leq k'}\ms S \into\ms S$ whenever $k\leq k'$ and $\wp\subset \wp'$. 
\end{remark}

\begin{lemma}\label{L: sheaf ttau}
Suppose $\ms S$ is a sheaf complex on $X$ and $x\in X$. Then, 
\begin{equation*}
H^i\left(\left(\ttau^{\wp}_{\leq k} \ms S\right)_x\right)\cong
\begin{cases}
0, & i>k+1,\\
T^{\wp}H^{k+1}(\ms S_x), & i=k+1,\\
H^i(\ms S_x),& i\leq k.
\end{cases}
\end{equation*}
Furthermore, the homology isomorphisms or torsion submodule isomorphisms, in the respective degrees, are induced by the inclusion $\ttau^{\wp}_{\leq k}\ms S \into \ms S$. 
\end{lemma}
\begin{proof}
It is an exercise to show that if $M_i$ is a direct system of $R$-modules then $T^\wp \dlim M_i\cong \dlim T^\wp M_i$. Now apply Lemma \ref{L: presheaf ttau} and basic properties of stalk cohomology.
\end{proof}

\begin{remark}\label{R: empty3}
As in Remark \ref{R: empty2}, if $\wp=\emptyset$, then $\tau_{\leq k}\ms S\into \ttau^{\emptyset}_{\leq k} \ms S$ is a quasi-isomorphism. 
\end{remark}

\subsection{Localized truncation}\label{S: local trunc}

The functor $\ttau^{\wp}_{\leq k}$ performs the same truncation over each point of the base space $X$. We next consider a modification that can truncate with respect to different $\wp$ and different $k$ depending on the point $x\in X$. Versions of such ``localized truncations'' that did not account for torsion information were constructed in \cite{GBF23}.

Let $\mc A$ be a sheaf complex on $X$, and let $\mf F$ be a locally-finite collection of disjoint closed subsets of $X$. Let $|\mf F|=\cup_{F\in \mf F}F$. Let $q$ be a function $q: \mf F\to \Z\times \P(\primeset{P}(R))$. We write $q(F)=(q_1(F),q_2(F))$.
In our construction below of the ts-Deligne sheaves, $\mf F$ will be the set of strata of $X$ of a given dimension and $q$ will be the restriction of a ts-perversity to  $\mf F$.

We will define a sheaf $\mf t_{\leq q}^{\mf F}\mc A$ as the sheafification of a presheaf $\mf T^{\mf F}_{\leq q }\mc A$. To define these, if $U\subset X$ is open and $U\cap |\mf F|\neq \emptyset$, let $\inf q_1(U)=\inf\{q_1(F)\mid F\in \mf F, F\cap U\neq \emptyset\}$, which may take the value $-\infty$, and let $\inf q_2(U)=\bigcap_{F\in \mf F, F\cap U\neq \emptyset} q_2(F)$. Now let 
\begin{equation*}
\mf T^{\mf F}_{\leq q}\mc A(U)=
\begin{cases}
\Gamma(U;\mc A),& U\cap |\mf F|=\emptyset,\\
\Gamma\left(U;\ttau_{\leq \inf q_1(U)}^{\inf q_2(U)}\mc A\right),& U\cap |\mf F| \neq\emptyset.
\end{cases} 
\end{equation*}
If $U\cap |\mf F| \neq\emptyset$ and $\inf q_1(U)=-\infty$, then we let $\mf T^{\mf F}_{\leq q}\mc A(U)=0$.
So, roughly speaking, $\mf T^{\mf F}_{\leq q}\mc A(U)$ is the truncation determined by the smallest truncation degree and smallest set of primes coming from $q(F)$ as $F$ ranges over sets of $\mf F$ that intersect $U$.

For the restriction maps of $\mf T^{\mf F}_{\leq q }\mc A$, if $W\subset U$ then we have $\inf P_1(U)\leq \inf P_1(W)$ and  $\inf P_2(U)\subset \inf P_2(W)$.
So using Remark \ref{R: mono}, if $W\subset U$ and $W\cap |\mf F| \neq\emptyset$, we have the composition of restriction and inclusion maps
\begin{equation*}
\Gamma\left(U;\ttau_{\leq \inf q_1(U)}^{\inf q_2(U)}\mc A\right)\to\Gamma\left(W;\ttau_{\leq \inf q_1(U)}^{\inf q_2(U)}\mc A\right)\into\Gamma\left(W;\ttau_{\leq \inf q_1(W)}^{\inf q_2(W)}\mc A\right).
\end{equation*}
If $U\cap |\mf F|\neq \emptyset$ but $W\cap |\mf F| =\emptyset$ we have a similar composition whose target module is $\Gamma\left(W;\mc A\right)$. Together with the standard restriction of $\mc A$ when $U\cap|\mf F|=\emptyset$, these determine the restriction homomorphism $\mf T^{\mf F}_{\leq q}\mc A(U)\to \mf T^{\mf F}_{\leq q}\mc A(W)$. Using that $\ttau^{\wp}_{\leq k}$ is a functor of sheaf complexes and the naturality of the monomorphisms in Remark \ref{R: mono}, we see that $\mf T^{\mf F}_{\leq q}$ is a functor from sheaf complexes to presheaf complexes over $X$.

\begin{definition}
For a sheaf complex $\mc A$ over $X$, let the \emph{locally torsion-tipped truncation} $\mf t_{\leq q}^{\mf F}\mc A$ be the sheafification of $\mf T^{\mf F}_{\leq q}\mc A$.
For a map $f:\mc A\to \mc B$ of sheaf complexes over $X$, we obtain $\mf t_{\leq q}^{\mf F}f$  by sheafifying the map $\mf T^{\mf F}_{\leq q}f$. 
\end{definition}

The following lemma contains the key facts we will need about $\mf t_{\leq q}^{\mf F}$. 

\begin{lemma}\label{L: properties}\hfill
\begin{enumerate}

\item $\mf t_{\leq q}^{\mf F}$ is an endofunctor of sheaf complexes on $X$.

\item There is a natural inclusion of sheaf complexes $\mf t^{\mf F}_{\leq q}\mc A\into \mc A$.

\item $\left.\left(\mf t_{\leq q}^{\mf F}\mc A\right)\right|_{X-|\mf F|}=\mc A|_{X-|\mf F|}$.

\item For each $F\in \mf F$, we have $\left.\left(\mf t_{\leq q}^{\mf F}\mc A\right)\right|_{F}=\left.\left(\ttau_{\leq q_1(F)}^{q_2(F)}\mc A\right)\right|_{F}$.
\end{enumerate}
\end{lemma}
\begin{proof}
These all follow immediately from the definitions and the properties of $\ttau^{\wp}_{\leq k}$. For the last two items we note that $\mf F$ being a locally finite collection of disjoint closed sets implies that if $x\in F\in \mf F$ then there is a neighborhood $U$ of $x$ that intersects no element of $\mf F$ other than $F$, and if $x\notin |\mf F|$ then there is a neighborhood $U$ of $x$ such that $U\cap |\mf F|=\emptyset$.
\end{proof}

\begin{remark}
$\mf T^{\mf F}_{\leq q}\mc A$ will not necessarily be a sheaf, so the sheafification in the definition is necessary. This is true even when all $q_2(F)=\emptyset$ so that all truncation functors are the classical ones; see \cite[Remark 3.5]{GBF23} for an example. 
\end{remark}

\begin{example}
It follows from the last statement of Lemma \ref{L: properties} that if $\mf F=\{X\}$, then  $\mf t_{\leq q}^{\mf F}\mc A=\ttau_{\leq q_1(X)}^{q_2(X)}\mc A$, which is a $\wp$-torsion-tipped truncation in the sense of our original definition. 
\end{example}

\begin{example}\label{E: whole thing}
Suppose that $q_1(F)=m$ for all $F\in\mf F$, that $q_2(F)=\emptyset$ for all $F\in \mf F$, and that $H^i(\mc A_x)=0$ for all $x\in X-|\mf F|$ and $i>m$. Then the inclusion $\tau_{\leq m}\mc A\into \mf t_{\leq q}^{\mf F}\mc A$ is a quasi-isomorphism, as we can see by looking at the behavior on open sets and consequently on stalk cohomology. 
\end{example}

\section{Torsion-sensitive Deligne sheaves}\label{S: TSD}

In this section, we define and study our torsion-sensitive Deligne sheaves. Section \ref{S: def2} contains the basic construction. In Section \ref{S: axiomatics}, we investigate the axiomatic and constructibility properties. 
Vanishing properties are proven in Section \ref{S: vanishing}, and
the duality theorem then follows in Section \ref{S: duality}. Finally, we treat self-duality in Section \ref{S: self dual}.

\textbf{Notation.} Throughout, we fix a ground PID $R$, and we let $X$ be a \emph{topological stratified $n$-pseudomanifold} \cite{GM2,Bo, GBF35}. Recall \cite[Definition 2.4.13]{GBF35} that a $0$-dimensional stratified pseudomanifold is a discrete set of points and that an $n$-dimensional stratified pseudomanifold has a filtration by closed subsets $X=X^n\supset X^{n-1}\supset \cdots\supset X^0\supset X^{-1}=\emptyset$ such that $X-X^{n-1}$ is dense and such that if $x\in X^{n-k}-X^{n-k-1}$ then there is a \emph{distinguished neighborhood} $U$ of $x$ and a compact $k-1$ dimensional stratified pseudomanifold $L$ (possibly empty) such that $U\cong \R^{n-k}\times cL$, where $cL$ is the open cone on $L$ and the homeomorphism takes  $U\cap X^{j}$ onto $\R^{n-k}\times cL^{j-n+k-1}$ for all $j$. Following \cite[Remark V.2.1]{Bo}, if we omit the density condition we call the resulting stratifications \emph{unrestricted}; unrestricted stratifications will be useful in Section \ref{S: perverse}. These spaces are all Hausdorff, finite dimensional, locally compact, and locally completely paracompact\footnote{A space is locally completely paracompact if every point has a neighborhood all of whose open subsets are paracompact \cite[Section V.1.17]{Bo}. 
Every stratified pseudomanifold $X$ is locally completely paracompact: By \cite[page 82]{GM2}, a compact (unrestricted) stratified pseudomanifold $L$ can be embedded in some Euclidean space. By adding Euclidean dimensions we can thus form $\R^{n-k}\times cL$ as a subspace of Euclidean space. Hence stratified pseudomanifolds are locally metrizable, which is sufficient. That compact pseudomanifolds can be embedded in Euclidean space isn't proven in \cite{GM2} but can be verified inductively over dimension. In fact, the above argument shows how to embed distinguished neighborhoods in Euclidean space once we know the links can be so embedded (which is obvious for $0$-dimensional links), and we can then use these embeddings and a partition of unity over a finite covering by distinguished neighborhoods to embed $X$ into Euclidean space as one does for compact manifolds \cite[Theorem 36.2]{MK2}.}. 

Let $U_k=X-X^{n-k}$, let $X_{n-k}=X^{n-k}-X^{n-k-1}=U_{k+1}-U_k$, and let $i_k:U_k\into U_{k+1}$ be the inclusion. The connected components of $X_{n-k}$ for $k>0$ are called \emph{singular strata}; the components of $X_n=U_1$ are \emph{regular strata}. We typically use $Z$ to denote a stratum.

 Throughout, $\H^*$ denotes hypercohomology and $\mc H^*$ a cohomology sheaf.

\subsection{The definition}\label{S: def2}

Let us recall the original Deligne sheaf construction of \cite{GM2}. 
The original perversity functions of Goresky and MacPherson were functions $\bar p:\Z_{\geq 1}\to \Z$, with additional restrictions including being nonnegative and nondecreasing. Given such a perversity and a local system $\mc E$ on $U_1$ thought of as a complex concentrated in degree $0$, the classical Deligne sheaf is constructed as 
$$\mc P_{X,\bar p,\mc E}= \tau_{\leq \bar p(n)}Ri_{n*} \cdots \tau_{\leq \bar p(1)}Ri_{1*}\mc E,$$
using the standard truncation $\tau$. We wish to modify this construction to account for torsion. 

One limitation of the original construction is that if we allow the possibility that $\bar p(k)>\bar p(k')$ for some $k<k'$ then a truncation at a later stage of the iterated construction will ``lop off'' some of the higher degree local cohomology established at an earlier stage. More dramatically, if we ever allow $\bar p(k)<0$ for any $k$ then $\mc P=0$. So to allow for more general perversities of the form $\bar p: \{\text{singular strata of $X$}\}\to \Z$ and with no further restrictions, which have become useful in recent years, we would also like our truncations to be local in the sense that the truncation at each stage of the iterated construction affects the sheaf only at the points of the strata just added. 
Such Deligne sheaves were constructed in \cite{GBF23}. 

To account both for torsion and for general perversities, we will use our locally torsion-tipped truncation functors. First we need to define our perversities and  coefficient systems.

\paragraph{Perversities.} We first define perversities that also track sets of primes. Let $\primeset{P}(R)$ be the set of primes of $R$ (up to unit), and let $\P(\primeset{P}(R))$ be its power set (so elements of $\P(\primeset{P}(R))$ are sets of primes of $R$).

\begin{definition}\label{D: ts-perv}
Let a \emph{torsion-sensitive perversity} (or simply \emph{ts-perversity}) be a function $\vec p: \{\text{singular strata of $X$}\}\to \Z\times \P(\primeset{P}(R))$. We denote the components of $\vec p(Z)$ by $(\vec p_1(Z),\vec p_2(Z))$. 
\end{definition}

\paragraph{Coefficients.}
Next we need to consider the coefficient systems we will use.
For a stratified pseudomanifold $X$, the classical Deligne sheaf construction assumes given a local system (locally constant sheaf) $\mc E$ of finitely generated $R$-modules on $U_1$ or, equivalently in the derived category, a sheaf complex $\mc E$ with $\mc H^0(\mc E)$ a local system of finitely generated $R$-modules and with  $\mc H^i(\mc E)=0$ for $i\neq 0$. To emulate certain versions of singular intersection homology Habegger and Saper work with much more general coefficients \cite{HS91} (see also \cite{GBF26}). Motivated by the perverse sheaf context we will explore below in Section \ref{S: perverse}, the following seems to be an appropriate definition for coefficients in the torsion-sensitive setting:

\begin{definition}\label{D: ts-coeff}
Let $\wp \subset \primeset{P}(R)$ be a set of primes of the PID $R$. We will call a complex of sheaves  $\mc E$ on a manifold $M$ a \emph{$\wp$-coefficient system} if 
\begin{enumerate}
\item  $\mc H^1(\mc E)$ is a locally constant sheaf of finitely generated $\wp$-torsion modules,
\item $\mc H^0(\mc E)$ is a locally constant sheaf of finitely generated $\wp$-torsion-free modules, and
\item $\mc H^i(\mc E)=0$ for $i\neq 0,1$.
\end{enumerate}
We will typically write ``a $\wp$-coefficient system'' to mean ``a $\wp$-coefficient system for some $\wp\subset \primeset{P}(R)$.'' We use $\wp(\mc E)$ for the primes with respect to which  $\mc E$ is a $\wp$-coefficient system. 

More generally, if $M=\amalg M_j$ is a disjoint union and $\mc E$ restricts on each $M_j$ to a $\wp$-coefficient system for some $\wp$, which may vary by component, we call $\mc E$ a \emph{ts-coefficient system} and write $\wp(M_j,\mc E)$ for $\wp(\mc E|_{M_j})$.  
\end{definition}

\begin{remark}\label{R: empty p}
If $\wp(\mc E)=\emptyset$ then, up to isomorphism in the derived category, a $\wp$-coefficient system is simply a local system of finitely generated $R$-modules in degree $0$. 
\end{remark}

\paragraph{Deligne sheaves.}
We define our ts-Deligne sheaves.
Let $\vec p$ be a ts-perversity on the stratified pseudomanifold $X$. Our construction will use the truncation functor $\mf t_{\leq \vec p}^{X_k}$, where we abuse our notation by allowing $X_k$ to also stand for the set of connected components of $X_k=X^k-X^{k-1}$ and letting $\vec p$ also refer to its restriction to these components.  

\begin{definition}Given a ts-perversity $\vec p: \{\text{singular strata of $X$}\}\to \Z\times \P(\primeset{P}(R))$ and ts-coefficient system $\mc E$ on $X-X^{n-1}$, let the \emph{torsion-sensitive Deligne sheaf} (or \emph{ts-Deligne sheaf}) be defined by
$$\ms P_{X,\vec p,\mc E}=  \mf t_{\leq \vec p}^{X_0}Ri_{n*}\ldots\mf t_{\leq \vec p}^{X_{n-1}}Ri_{1*}\mc E.$$
\end{definition}

This construction generalizes that of Goresky-MacPherson in \cite{GM2} and the construction of the Deligne sheaf for general perversities in \cite{GBF23}. 
If $X$, $\vec p$, or $\mc E$ is fixed, we sometimes drop them from the notation.

\begin{example}\label{E: GM}
Suppose that $\wp(\mc E)=\emptyset$, that $\vec p_2(Z)=\emptyset$ for all singular strata $Z$, and that $\vec p_1$ is a nonnegative and nondecreasing function only of codimension so that we write $\vec p_1(k)$ rather than $\vec p_1(Z)$ when $\codim(Z)=k$. Then $\mc E$ is just a standard local system concentrated in degree $0$ by Remark \ref{R: empty p}, and $\mf t_{\leq \vec p}^{X_{n-k}}$ reduces to $\tau_{\leq \vec p_1(k)}$ by Example \ref{E: whole thing}, using that the nondecreasing assumption on $\vec p_1$ implies that the stalk cohomology over points of $X-X^{n-k}$ will already be trivial in degrees $>\vec p_1(k)$ via the inductive construction. 
Therefore in this case $\ms P$ is quasi-isomorphic to the standard Deligne sheaf $\mc P$ as defined in \cite{GM2}. 

More generally, continuing to assume that $\vec p_1$ is a nonnegative and nondecreasing function only of codimension but letting now $\vec p_2(Z)$ be arbitrary, then $\ms P$ is the standard Deligne sheaf if $\wp(\mc E)=\emptyset$ and if $T^{\vec p_2(Z)}H^{\vec p_1(k)+1}((Ri_{k*}(\ms P|_{U_k}))_x)=0$ for each $k$ and each $x\in Z\subset X_{n-k}$. If $X$ is locally $\bar p$-torsion-free in the sense of Goresky and Siegel \cite{GS83}, this will be the case for any $\vec p$ such that $\vec p_1=\bar p$. To see this, we use that 
$H^{*}((Ri_{k*}(\ms P|_{U_k}))_x)\cong \dlim_{x\in U}\H^*(U;Ri_{k*}(\ms P|_{U_k}))$, while the latter system is constant over distinguished neighborhoods of $x$ by \cite[Lemma V.3.9.b and Proposition V.3.10]{Bo} and Theorem \ref{T: cc} concerning constructibility, which we will demonstrate below. Then also $\H^*(U;Ri_{k*}(\ms P|_{U_k}))\cong \H^*(L;\ms P|_L)$, where $L=L^{k-1}$ is the link of $x$, by \cite[Lemma V.3.9.a]{Bo}. So the condition in \cite[Definition 4.1]{GS83} that 
$I^{\bar p}H_{k-\vec p_1(k)-2}(L;\mc E)=\H^{\bar p(k)+1}(L;\ms P|_L)$ be torsion free is equivalent to the condition that 
 $H^{\vec p_1(k)+1}((Ri_{k*}(\ms P|_{U_k}))_x)$ be torsion free. 
\end{example}

\begin{example}\label{E: CS}
Similarly, suppose that $\wp(\mc E)=\emptyset$, that $\vec p_1$ is a nonnegative and nondecreasing function only of codimension, that $H^{\vec p_1(k)+1}((Ri_{k*}(\ms P|_{U_k}))_x)$ is always a torsion $R$-module, and that $\vec p_2(Z)=\primeset{P}(R)$ for all singular strata $Z$. Then the additional torsion that we get from the torsion-tipped truncation one degree above the standard truncation degree is all the cohomology in that degree, so this is the same as performing the standard truncation one degree higher. Thus the complex $\ms P_{\vec p}$ is the same as the Deligne sheaf $\mc P_{\vec p_1+1}$, where $\vec p_1+1$ is the perversity whose value on $k$ is $\vec p_1(k)+1$. Such Deligne sheaves arise in the Cappell-Shaneson superduality theorem \cite{CS91}.
\end{example}

It would be interesting to have a geometric formulation of the hypercohomology groups $\H^*(X;\ms P_{\vec p})$ in terms of simplicial or singular chains with certain restrictions, as is the case for intersection homology theory and the classical Deligne sheaf $\mc P_{\bar p,\mc E}$.

\subsection{Axiomatics and constructibility} \label{S: axiomatics}

We define a set of axioms analogous to the Goresky-MacPherson axioms Ax1 and show that they characterize $\ms P$. Our treatment parallels the work  of \cite{GM2} and the exposition of \cite[Section V.2]{Bo}.

\begin{definition}
Let $X$ be an $n$-dimensional stratified pseudomanifold, and let $\mc E$ be a ts-coefficient system on $U_1$ over a principal ideal domain $R$. For a sheaf complex $\ms S$ on $X$, let $\ms S_k=\ms S|_{U_k}$. 
 We say $\ms S$ satisfies the \emph{Axioms TAx1$(X,\vec p, \mc E)$ (or simply TAx1)} if 

\begin{enumerate}
\item\label{A: bounded} $\ms S$ is quasi-isomorphic to a complex that is bounded and that is $0$ in negative degrees;
\item \label{A: coeffs} $\ms S|_{U_1}\cong\mc E$;
\item \label{A: truncate} if $x\in Z\subset X_{n-k}$, where $Z$ is a singular stratum, then
 $H^i(\ms S_x)=0$ for $i>\vec p_1(Z)+1$ and $H^{\vec p_1(Z)+1}(\ms S_x)$ is $\vec p_2(Z)$-torsion;
\item \label{A: attach} if $x\in Z\subset X_{n-k}$, where $Z$ is a singular stratum, then the attachment map $\alpha_k:\ms S_{k+1}\to Ri_{k*}\ms S_k$ induces stalkwise cohomology isomorphisms at $x$  in degrees $\leq \vec p_1(Z)$ and it induces stalkwise cohomology isomorphisms $H^{\vec p_1(Z)+1}(\ms S_{k+1,x})\to T^{\vec p_2(Z)}H^{\vec p_1(Z)+1}( (Ri_{k*}\ms S_k)_x)$. 
\end{enumerate} 
\end{definition}

\begin{theorem}\label{T: axioms}
The sheaf complex $\ms P_{X,\vec p,\mc E}$ satisfies the axioms  TAx1$(X,\vec p, \mc E)$, and any sheaf complex satisfying  TAx1$(X,\vec p, \mc E)$ is quasi-isomorphic to $\ms P_{X,\vec p,\mc E}$.
\end{theorem}

The theorem relies on the following lemma.

\begin{lemma}\label{L: axioms}
Suppose $\ms S$ satisfies the axioms  TAx1$(X,\vec p, \mc E)$. Then, for $k>0$, we have $\ms S_{k+1}\cong \mf t_{\leq\vec p}^{X_{n-k}}Ri_{k*}\ms S_k$.
\end{lemma}
\begin{proof}
By the functoriality of the truncation functors and their inclusion properties, we have a commutative diagram
\begin{diagram}
\ms S_{k+1}&\rTo^{\alpha_k}&Ri_{k*}\ms S_{k}\\
\uTo^\beta&&\uTo_\gamma\\
\mf t_{\leq \vec p}^{X_{n-k}}\ms S_{k+1}&\rTo^{\mf t_{\leq \vec p}^{X_{n-k}}\alpha_k}&\mf t_{\leq \vec p}^{X_{n-k}}Ri_{k*}\ms S_{k}.
\end{diagram}
The map $\beta$ is a quasi-isomorphism by axiom \eqref{A: truncate} and Lemmas \ref{L: properties} and \ref{L: sheaf ttau}.

At $x\in Z\subset X_{n-k}$, the map $\mf t_{\leq \vec p}^{X_{n-k}}\alpha_k$ is evidently an isomorphism in degrees $i>\vec p_1(Z)+1$. In degrees $i\leq \vec p_1(Z)$, $\alpha_k$ is a quasi-isomorphism by axiom \eqref{A: attach} and $\gamma$ is a quasi-isomorphism by Lemmas \ref{L: properties} and \ref{L: sheaf ttau}; thus $\mf t_{\leq \vec p}^{X_{n-k}}\alpha_k$ is a quasi-isomorphism in this range, as well. Finally, consider the diagram
\begin{diagram}
H^{\vec p_1(Z)+1}\left(\ms S_{k+1,x}\right)&\rTo&T^{\vec p_2(Z)}H^{\vec p_1(Z)+1}\left((Ri_{k*}\ms S_{k})_x\right)\\
\uTo^\beta&&\uTo\\
H^{\vec p_1(Z)+1}\left(\left(\mf t_{\leq \vec p}^{X_{n-k}}\ms S_{k+1}\right)_x\right)&\rTo^{\mf t_{\leq \vec p}^{X_{n-k}}\alpha_k}&H^{\vec p_1(Z)+1}\left(\left(\mf t_{\leq \vec p}^{X_{n-k}}Ri_{k*}\ms S_{k}\right)_x\right).
\end{diagram}
By Lemma \ref{L: sheaf ttau}, the righthand map is an isomorphism induced by the sheaf inclusion.
The top map is induced by $\alpha$ and is an isomorphism by axiom \eqref{A: attach}. We have already seen that $\beta$ induces an isomorphism. Thus the bottom map must be an isomorphism, and so $\mf t_{\leq \vec p}^{X_{n-k}}\alpha_k$ is a quasi-isomorphism of sheaves. 

Together, $\beta$ and $\mf t_{\leq \vec p}^{X_{n-k}}\alpha_k$ provide the desired quasi-isomorphism of the lemma.
\end{proof}

\begin{proof}[Proof of Theorem \ref{T: axioms}]
It is direct from the construction of $\ms P=\ms P_{X,\vec p,\mc E}$ that it satisfies the axioms. Conversely, suppose $\ms S$ satisfies the axioms and that $\ms S_k\cong\ms P_k$ for some $k$. This is true for $\ms S_1$ by axiom \eqref{A: coeffs}. By the preceding lemma, $\ms S_{k+1}\cong\mf t_{\leq \vec p}^{X_{n-k}}Ri_{k*}\ms S_k$. But by the induction hypothesis, this is quasi-isomorphic to $\mf t_{\leq \vec p}^{X_{n-k}}Ri_{k*}\ms P_k$, which is $\ms P_{k+1}$. The theorem follows by induction.
\end{proof}

Let $\mf X$ denote the stratification of the stratified pseudomanifold $X$. We recall the following definitions; see \cite[Section V.3.3]{Bo}. We say that the sheaf complex $\mc S$ is $\mf X$-cohomologically locally constant ($\mf X$-clc) if each sheaf $\mc H^i(\mc S)$ is locally constant on each stratum. We say $\mc S$ is $\mf X$-cohomologically constructible ($\mf X$-cc) if it is $\mf X$-clc and each stalk $\mc H^i(\mc S)_x$ is finitely generated. We will also use the notion of $\mc S$ being cohomologically constructible (cc); we refer to \cite{Bo} for the full definition but note that by \cite[Remark V.3.4.b]{Bo} if $\mc S$ is already known to be $\mf X$-cc then it is also cc if for all $x\in X$ and $i\in \Z$ the inverse system $\H^i_c(U;\mc S)$ over open neighborhoods of $x$ is essentially constant with finitely generated inverse limit.

\begin{theorem}\label{T: cc}
The sheaf complex $\ms P_{X,\vec p,\mc E}$ is $\mf X$-clc, $\mf X$-cc, and cc.
\end{theorem}
\begin{proof}
This theorem is completely analogous to \cite[Proposition V.3.12]{Bo} and follows from the machinery of \cite[Section V.3]{Bo}.  The only additional observations needed are that $\mc E$ is $\mf X$-cc by definition and that $\mf t_{\leq \vec p}^{X_{n-k}}$ preserves the properties of being $\mf X$-cc.
\end{proof}

As in \cite{GM2, Bo}, it will be useful to reformulate the axioms. First we show that axiom \eqref{A': attach} can be replaced by an equivalent condition. Then we will formulate the axioms TAx1' and show they are equivalent to TAx1.

\begin{lemma}\label{A''}
Suppose $\mc S$ satisfies axiom TAx1\eqref{A: truncate}.  Then TAx1\eqref{A: attach} is equivalent to the following condition:  Suppose  $x\in Z\subset X_{n-k}$, $k>0$, and let $j:X_{n-k}\into X$ be the inclusion; then
\begin{enumerate}
\item  $H^i((j^!\mc S)_x)=0$ for $i\leq \vec p_1(Z)+1$,
\item  $H^{\vec p_1(Z)+2}((j^!\mc S)_x)$ is $\vec p_2(Z)$-torsion free.
\end{enumerate}
\end{lemma}
\begin{proof}
First, let $j:X_{n-k}\into X$, $j_k:X_{n-k}\into U_{k+1}$, and $w: U_{k+1}\into X$. So $j=wj_k$, and we have $j^!=(wj_k)^!\cong  j_k^!w^!\cong j_k^!w^*$ because $w$ is an open inclusion. So $j^!\mc S\cong j_k^!\mc S_{k+1}$, letting $\mc S_k=\mc S|_{U_k}$. 

For $x\in Z$, there is a long exact sequence (see \cite[V.1.8(7)]{Bo})
\begin{diagram}
&\rTo & H^i( (j_k^!\mc S_{k+1})_x)&\rTo &H^i( \mc S_{k+1,x})&\rTo^{\alpha} &H^i( (Ri_{k*}\mc S_{k})_x)&\rTo &.
\end{diagram} 
We have just seen that we must have $H^i( (j_k^!\mc S_{k+1})_x)\cong H^i((j^!\mc S)_x)$, and of course $\mc S_{k+1,x}=\mc S_x$.

Suppose $\mc S$ satisfies TAx1\eqref{A: attach}. Then we have $H^i( (j^!\mc S)_x)=0$ for $i\leq p_1(Z)+1$, noting that $\alpha$ remains injective in degree $p_1(Z)+1$. Around degree $p_1(Z)+2$ and using TAx\eqref{A: truncate}, the sequence specializes to 
\begin{diagram}
0&\rTo & H^{\vec p_1(Z)+1}( \mc S_x)&\rTo^{\alpha} &H^{\vec p_1(Z)+1}( (Ri_{k*}\mc S_k)_x)&\rTo &H^{\vec p_1(Z)+2}( (j^!\mc S)_x)&\rTo &0,
\end{diagram}
and since $\alpha$ is an isomorphism onto $T^{\vec p_2(Z)}H^{\vec p_1(Z)+1}( (Ri_{k*}\mc S_k)_x)$, it follows that $H^{\vec p_1(Z)+2}( (j^!\mc S)_x)\cong H^{\vec p_1(Z)+1}( (Ri_{k*}\mc S_k)_x)/T^{\vec p_2(Z)}H^{\vec p_1(Z)+1}( (Ri_{k*}\mc S_k)_x)$ is $\vec p_2(Z)$-torsion free.

Conversely, if $j^!\mc S$ satisfies the conditions stated in the lemma, then certainly $\alpha$ is an isomorphism on cohomology for $i\leq \vec p_1(Z)$. Around $H^{\vec p_1(Z)+2}( (j^!\mc S)_x)$, we have the same specialized sequence as above.
As $H^{\vec p_1(Z)+1}( \mc S_x)$ is $\vec p_2(Z)$-torsion by assumption, it must map injectively to the $\vec p_2(Z)$-torsion subgroup of $H^{\vec p_1(Z)+1}( (Ri_{k*}\mc S_k)_x)$. But we assume $H^{\vec p_1(Z)+2}( (j^!\mc S)_x)$ is $\vec p_2(Z)$-torsion free, so $\alpha$ must take $H^{\vec p_1(Z)+1}( \mc S_x)$ onto $T^{\vec p_2(Z)}H^{\vec p_1(Z)+1}( (Ri_{k*}\mc S_k)_x)$. 
\end{proof}

\begin{definition}\label{T: Ax1'}
Next, we say $\ms S$ satisfies the \emph{Axioms TAx1'$(X,\vec p, \mc E)$ (or simply TAx1')} if it is $\mf X$-clc and

\begin{enumerate}
\item\label{A': bounded} $\ms S$ is quasi-isomorphic to a complex that is bounded and that is $0$ in negative degrees;
\item \label{A': coeffs} $\ms S|_{U_1}\cong\mc E$;
\item \label{A': truncate} if $x\in Z\subset X_{n-k}$, where $Z$ is a singular stratum, then
$H^i(\ms S_x)=0$ for $i>\vec p_1(Z)+1$ and $H^{\vec p_1(Z)+1}(\ms S_x)$ is $\vec p_2(Z)$-torsion;

\item \label{A': attach}  if $x\in Z\subset X_{n-k}$, where $Z$ is a singular stratum, and $f_x:x\into X$ is the inclusion, then

\begin{enumerate}
\item  $H^i(f_x^!\mc S)=0$ for $i\leq \vec p_1(Z)+n-k+1$
\item $H^{p_1(Z)+n-k+2}(f_x^!\mc S)$ is $\vec p_2(Z)$-torsion free.

\end{enumerate}
\end{enumerate}
\end{definition}

\begin{theorem}\label{T: ax equiv}
TAx1' is equivalent to TAx1.
\end{theorem}
\begin{proof}
If $x\in Z\subset X_{n-k}$ and  $\ell_x:x\into X_{n-k}$, $j:X_{n-k}\into X$, and $f_x:x\into X$ are the inclusions, then $f_x=j\circ \ell_x$, so $f_x^!=\ell_x^!j^!$. So $H^i(f_x^!\mc S)=H^i(\ell_x^!j^!\mc S)$, which, since $X_{n-k}$ is an $n-k$ dimensional manifold, is isomorphic to $H^{i-n+k}( (j^!\mc S)_x)$ by \cite[Proposition V.3.7.b]{Bo}. This last isomorphism uses that $j^!\mc S$ is $\mf X$-clc, which follows from $\mc S$ being $\mf X$-clc by \cite[Proposition V.3.10]{Bo}; that $\mc S$ is $\mf X$-clc holds by assumption if $\mc S$ satisfies TAx1' and by Theorem \ref{T: cc} if $\mc S$ satisfies TAx1. 
Thus the theorem follows from Lemma \ref{A''}.
\end{proof}

\subsection{Vanishing results}\label{S: vanishing}

In this section we prove some vanishing results that are both interesting in their own right and useful in our proof of duality. 
For this we first need a torsion-sensitive version of \cite[Lemma V.9.5]{Bo}, though we simplify a bit by assuming that $\mc S$ is $\mf X$-cc, which is all we will need. It will also be sufficient for our later needs to fix a collection of primes $\wp$ and not vary it by stratum. When $\wp=\emptyset$, the following lemma is a special case of  \cite[Lemma V.9.5]{Bo}.

\begin{lemma}\label{L: vanishing}
Let $X$ be a stratified pseudomanifold, $\ell\in \Z$, and $\mc S$ a bounded-below $\mf X$-cc complex of sheaves on $X$. Suppose for each $x\in X$ that if $x\in X_k$ then $H^i(\mc S_x)=0$ for $i>\ell-k+1$ and $H^{\ell-k+1}(\mc S_x)$ is $\wp$-torsion. Then $\H_c^i(X;\mc S)=0$ for $i>\ell+1$ and  
$\H_c^{\ell+1}(X;\mc S)$ is $\wp$-torsion.
\end{lemma}
\begin{proof}
We first suppose $X$ is an $n$-manifold, trivially filtered ($X^k=\emptyset$ for $k<n$). Then our hypothesis is $H^i(\mc S_x)=0$ for $i>\ell-n+1$ and $H^{\ell-n+1}(\mc S_x)$ is $\wp$-torsion.
The module $\H^i(X;\mc S)$ is the abutment of a spectral sequence with $E_2^{p,q}=H^p_c(X;\mc H^q(\mc S))$ \cite[Section V.1.4]{Bo}. By \cite[Definitions II.16.3 and II.16.6 and Corollary II.16.28]{Br}, since $X$ is an $n$-manifold we have $H^p_c(X;\mc A)=0$ for $p>n$ and any sheaf $\mc A$ of $R$-modules. If  $p\leq n$ and $p+q>\ell+1$, then $q>\ell-n+1$; so each $E_2^{p,q}$ is $0$ for $p+q>\ell+1$ and $\H^i_c(X;\mc S)=0$ for $i>\ell+1$. 
Similarly, if $p+q=\ell+1$ then the only possible nonzero $E_2^{p,q}$ term is $H^n_c(X;\mc H^{\ell-n+1}(\mc S))$. Raising either index results in a $0$ module, so this is also the only $E^{p,q}_\infty$ term for $p+q=\ell+1$, by which $\H_c^{\ell+1}(X;\mc S)\cong H^n_c(X;\mc H^{\ell-n+1}(\mc S))$. By our assumptions, $\mc H^{\ell-n+1}(\mc S)$ is a locally-constant sheaf of finitely-generated $\wp$-torsion modules. By \cite[Theorem III.1.1]{Br}, $H^n_c(X;\mc H^{\ell-n+1}(\mc S))$ is isomorphic to the classical singular compactly supported cohomology with coefficients in $\mc H^{\ell-n+1}(\mc S)$. It is then evident from the definition \cite[page 26]{Bo} that $H^n_c(X;\mc H^{\ell-n+1}(\mc S))$ is $\wp$-torsion; in fact each singular cochain is $\wp$-torsion. 

We now can now proceed to more general $X^n$ by induction on dimension. If $n=0$ then we are done by the manifold case. Suppose now that the lemma is proven through dimension $n-1$, and let $\mc S$ satisfy the hypotheses on $X=X^n$. 
We have a long exact sequence \cite[Remark 2.4.5]{DI04}
$$\cdots \to \H^i_c(X-X^{n-1};\mc S)\to \H^i_c(X;\mc S)\to \H^i_c(X^{n-1};\mc S)\to \cdots.$$
The restriction of $\mc S$ to $X^{n-1}$ and $X-X^{n-1}$ continues to satisfy the hypotheses on each of these subspaces, so by the induction assumption and the manifold case we have 
$\H^i_c(X-X^{n-1};\mc S)=\H^i_c(X^{n-1};\mc S)=0$ for $i>\ell+1$, and so $\H^i_c(X;\mc S)=0$ for $i>\ell+1$. Similarly, by induction and the manifold case $\H^{\ell+1}_c(X-X^{n-1};\mc S)$ and $\H^{\ell+1}_c(X^{n-1};\mc S)$ are each $\wp$-torsion. It follows that $\H^{\ell+1}_c(X-X^{n-1};\mc S)$ is $\wp$-torsion.
\end{proof}

\begin{theorem}\label{T: vanishing}
Suppose $\mc S$ satisfies the axioms TAx1$(X,\vec p, \mc E)$  on the $n$-dimensional stratified pseudomanifold $X$ for some $ts$-perversity $\vec p$ and $ts$-coefficient system $\mc E$.
If $\mc H^1(\mc E)$ is a local system of $\wp$-torsion modules then:

\begin{enumerate}
\item For each open set $U\subset X$ we have $\H^i_c(U;\mc S)=0$ for $i>n+1$ and $\H^{n+1}_c(U;\mc S)$ is $\wp$-torsion.

\item\label{I: local} If $x\in X_{n-k}$ for $k>0$ then $H^i(\mc S_x)=0$ for $i>k$ and $H^{k}(\mc S_x)$ is $\wp$-torsion.

\end{enumerate}

\end{theorem}

\begin{remark}\label{R: vanishing}
Note that if $H^1(\mc E_x)\neq 0$ then the second property fails for $k=0$, making the restriction $k>0$ necessary.
On the other hand, if $H^1(\mc E_x)= 0$, we can suppose $\wp=\emptyset$ and conclude that each $\H^i_c(U;\mc S)=0$ for $i>n$.
\end{remark}

\begin{proof}[Proof of Theorem \ref{T: vanishing}]
We will perform an induction argument over the depth of $X$, utilizing Lemma \ref{L: vanishing} and taking $\ell=n$.

First, suppose $X$ has depth $0$ so that $X=U_1$ is a manifold. Since $\mc S|_{U_1}\cong \mc E$, we have for $x\in U_1$ that $H^i(\mc S_x)=H^i(\mc E_x)=0$ for $i>1$ and $H^1(\mc S_x)=H^1(\mc E_x)$ is $\wp$-torsion. The results about $\H^i_c(U;\mc S)$ follow from Lemma \ref{L: vanishing} taking $\ell=n$.

Now, assume as induction hypothesis that we have shown the theorem for any stratified pseudomanifold of depth $K$ for $0\leq K<m$, and let $X$ have depth $m$. In particular, the proposition holds then for  $U_m=X-X^{n-m}$ and $\mc S|_{U_m}$. 
We will extend the result to open subsets of $X=U_m\cup X_{n-m}$ and points $x\in X_{n-m}$.

Let $x\in Z\subset X_{n-m}$. From the axioms and Lemmas \ref{L: axioms} and \ref{L: sheaf ttau}, we know that  $H^i(\mc S_x)$ is a submodule of $H^i((Ri_{m*}\mc S|_{U_m})_x)\cong \dlim_{x\in U}\H^i(U-Z;\mc S)$. 
 Since $x$ has a cofinal system of distinguished neighborhoods, we can suppose $U-Z\cong \R^{n-m+1}\times L$, where $L$ is the $m-1$ dimensional link of $Z$. Then $\H^i(U-Z;\mc S)\cong \H^i(L;\mc S|_L)$ by \cite[Lemma V.3.8.b]{Bo}, as we have shown in Theorem \ref{T: cc} that sheaves satisfying the axioms must be $\mf X$-cc.
Recall that (fixing a specific embedding of $L$) we have $L^{m-1-j}=X^{n-j}\cap L$ for all $j$. So if  $y\in L_{m-1-j}$, $j>0$, then $y\in X_{n-j}$ and so by induction hypothesis 
 $H^i(\mc S_y)=0$ for $i>j$ and $H^{j}(\mc S_y)$ is $\wp$-torsion. If $j=0$, then we have $L_{m-1}=X_{n}\cap L$, and we know for such points that $H^i(\mc S_y)=H^i(\mc E_y)=0$ for $i>1$ while $H^1(\mc S_y)=H^1(\mc E_y)$ is $\wp$-torsion. So in the full range $0\leq j<m$, the hypotheses of Lemma \ref{L: vanishing} hold on $L$ for $\mc S|_L$ with $\ell=m-1$ (in fact, the hypotheses are possibly sharp only on the top strata $L_{m-1}$). 
So, as $L$ is compact, we have $\H^i(L;\mc S|_L)=\H_c^i(L;\mc S|_L)=0$ for $i>m$ and $\H^m(L;\mc S|_L)$ is $\wp$-torsion, implying the same for $H^i(\mc S_x)$.

Finally, we can employ Lemma \ref{L: vanishing} again on any open $U\subset X$ with $\ell=n$ to conclude $\H^i_c(U;\mc S)=0$ for $i>n+1$ and $\H^{n+1}_c(U;\mc S)$ is $\wp$-torsion\footnote{Once again the hypotheses needed for Lemma \ref{L: vanishing} are only sharp on the top strata $X_{n}$ and only if $H^1(\mc E)\neq 0$. The reader might therefore wonder if we could prove a stronger vanishing result in the case $\mc H^1(\mc E)=0$. But we have already seen in Remark \ref{R: vanishing} that the theorem as stated is enough in this case to tell us $\mc H^{i}_c(U)=0$ for all $i>n$, and the conclusion $H^i(\mc S_x)=0$ for $i>m-1$ for $x\in X_{n-m}$, $m>0$, follows similarly from the argument on the links. So this stronger result is already a consequence of the current one. If we assume further that $\mc H^1(\mc E)=0$ and $\mc H^0(\mc E)$ is $\wp$-torsion then we could strengthen our application of Lemma \ref{L: vanishing}
 to conclude that $\H^n_c(U;\mc S)$ is $\wp$-torsion, but in fact in this case it's not hard to modify the argument of Lemma \ref{L: vanishing} to see directly that \emph{all} of the  $\H^i_c(U;\mc S)$ must be $\wp$-torsion; cf.\ Lemma \ref{L: all torsion}.}.  This completes the induction.  
\end{proof}

\begin{remark}
The proposition demonstrates that we could limit our perversities $\vec p$ so that $\vec p_1(Z)\leq \codim(Z)$, as stalk cohomology of ts-Deligne sheaves always vanishes in higher degrees so that truncating in higher degrees gives nothing new. 
\end{remark}

\subsection{Duality}\label{S: duality}

In this section we prove our duality theorem for ts-Deligne sheaves. Throughout we let 
 $\mc D_X$ denote the Verdier dualizing functor on $X$, omitting the space when clear from context.

\begin{definition}\label{D: dual perv}
Given a ts-perversity $\vec p$, we define the \emph{dual ts-perversity} by $D\vec p=(D\vec p_1,D\vec p_2)$ with $D\vec p_1(Z)=\codim(Z)-2-\vec p_1(Z)$ and $D\vec p_2(Z)=\primeset{P}(R)-\vec p_2(Z)$, the complement of $\vec p_2(Z)$ in the set of primes (up to unit) of $R$. Notice that $D\vec p_1$ is the perversity that is complementary to the perversity $\vec p_1$ in the standard sense \cite[Definition 3.1.7]{GBF35}. 
\end{definition}

\begin{theorem}\label{T: duality}
Let $X$ be an $n$-dimensional stratified pseudomanifold, $\vec p$ a ts-perversity on $X$, and $\mc E$ a ts-coefficient system on $U_1$ over a principal ideal domain $R$. Then $\mc D \ms P_{\vec p,\mc E}[-n]$ is quasi-isomorphic to $\ms P_{D\vec p,\mc D\mc E[-n]}$ by a quasi-isomorphism that extends the identity morphism of $\mc D\mc E[-n]$ on $U_1$.
\end{theorem}

Before providing the proof, we make some observations, present some corollaries, and show that $\mc D\mc E[-n]$ is indeed a ts-coefficient system. In fact, we will see in Proposition \ref{P: dual coeff} that $\mc E$ and $\mc D\mc E[-n]$ are ts-coefficient systems with respect to complementary sets of primes. 

\begin{remark}
If our base ring is a field, then $\mc E$ is a locally-constant system of finitely-generated vector spaces and by Example \ref{E: GM} each $\ms P_{\vec p,\mc E}$ is the standard Deligne sheaf $\mc P_{\vec p_1,\mc E}$, where $\vec p_1$ is the first component of $\vec p$. In this case, Theorem \ref{T: duality} reduces to the Goresky-MacPherson duality theorem \cite{GM2} if $\vec p_1$ is a Goresky-MacPherson perversity. If $\vec p_1$ is a general perversity,  Theorem \ref{T: duality} with field coefficients reduces to the duality theorem  of \cite{GBF23}. 

Suppose $R$ is a PID, $\bar p$ is a general perversity, $\wp(\mc E)=\emptyset$, and $X$ is locally $(\bar p,\mc E)$-torsion-free in the sense of \cite{GS83} (see also \cite{GBF35}), i.e.  for each singular stratum $Z$ and each $x\in Z$, the $R$-module $I^{\bar p}H_{\codim(Z)-2-\bar p(Z)}(L_x;\mc E)$ is $R$-torsion-free, where $L_x$ is the link of $x$ in $X$. In this case again by Example \ref{E: GM} we have $\ms P_{\vec p,\mc E}=\mc P_{\bar p,\mc E}$ for any $\vec p$ such that $\vec p_1=\bar p$, and Theorem \ref{T: duality} reduces to the duality theorem of Goresky and Siegel \cite{GS83} if $\bar p$ is a Goresky-MacPherson perversity or the duality theorem  proven in \cite{GBF23} for more general perversities.

Finally, suppose that $\bar p$ is a Goresky-MacPherson perversity, that $\wp(\mc E)=\emptyset$, and that for each singular stratum $Z$ and each $x\in Z$, $I^{\bar p}H_{*}(L_x;\mc E)$ is $R$-torsion. Suppose further that $\vec p$ is a ts-perversity with $\vec p_2(Z)=\primeset{P}(R)$ for all singular strata $Z$. Then by Example \ref{E: CS} $\ms P_{\vec p,\mc E}=\mc P_{\vec p_1+1,\mc E}$, where $\vec p_1+1$ is the $\Z$-valued perversity such that $(\vec p_1+1)(Z)=\vec p_1(Z)+1$ for all singular $Z$. Also $\ms P_{D\vec p,\mc D\mc E[-n]}=\mc P_{\bar q,\mc D\mc E[-n]}$, where $\bar q$ is the $\Z$-valued perversity such that $(\vec p_1+1)(Z)+\bar q(Z)=\vec p_1(Z)+1+\bar q(Z)=\codim(Z)-1$. 
With these assumptions Theorem \ref{T: duality} reduces to the Superduality Theorem of Cappell and Shaneson \cite{CS91}.  Note that in order to have $\ms P_{\vec p,\mc E}=\mc P_{\vec p_1+1,\mc E}$ it is in fact sufficient to require only $I^{\bar p}H_{k-2-\bar p(Z)}(L_x;\mc E)$ to be torsion. 
\end{remark}

\begin{corollary}\label{T: PD}
Let $X$ be a n-dimensional stratified pseudomanifold, and let $\mc E$ be a ts-coefficient system on $U_1$ over a principal ideal domain $R$. Let $T\H^*$ and $F\H^*$ denote, respectively, the $R$-torsion submodule and $R$-torsion-free quotient module of $\H^*$, and let $Q(R)$ denote the field of fractions of $R$.

Suppose  $\text{\emph{Ext}}\left( \H_c^{n-i+1}\left(X;\ms P_{\vec p,\mc E}\right),R\right)$ is a torsion $R$-module (for example, if $\H_c^{n-i+1}\left(X;\ms P_{\vec p,\mc E}\right)$ is finitely generated).
Then 
\begin{align*}
	F\H^{i}\left(X;\ms P_{D\vec p,\mc D\mc E[-n]}\right)&\cong \text{\emph{Hom}}\left(\H_c^{n-i}\left(X;\ms P_{\vec p,\mc E}\right),R\right)\cong \text{\emph{Hom}}\left(F\H_c^{n-i}\left(X;\ms P_{\vec p,\mc E}\right),R\right)\\
	T\H^{i}\left(X;\ms P_{D\vec p,\mc D\mc E[-n]}\right)& \cong \text{\emph{Ext}}\left(\H_c^{n-i+1}\left(X;\ms P_{\vec p,\mc E}\right),R\right)\cong \text{\emph{Hom}}\left(T\H_c^{n-i+1}\left(X;\ms P_{\vec p,\mc E}\right),Q\left(R\right)/R\right).
\end{align*}

In particular, if $X$ is compact and orientable and $\mc E=R_{U_1}$ is the constant sheaf with stalk $R$ in degree $0$ then 
\begin{align*}
F\H^{i}\left(X;\ms P_{D\vec p,R_{U_1}}\right)&\cong \text{\emph{Hom}}\left(F\H^{n-i}\left(X;\ms P_{\vec p,R_{U_1}}\right),R\right)\\
 T\H^{i}\left(X;\ms P_{D\vec p,R_{U_1}}\right)&\cong \text{\emph{Hom}}\left(T\H^{n-i+1}\left(X;\ms P_{\vec p,R_{U_1}}\right),Q\left(R\right)/R\right).
 \end{align*} 
\end{corollary}

\begin{proof}
These statements follow from Theorem \ref{T: duality}, using the universal coefficient sequence for Verdier duality \cite[Theorem 3.4.4]{BaIH} and basic homological algebra \cite[Section 8.4.1]{GBF35}. 
\end{proof}

Next we show that $\mc D\mc E[-n]$ is also a ts-coefficient system. 
The following basic fact about $\mc D$ will be useful both here and in several places below: If $f_x:x\into X$ is the inclusion and $\mc S$ is in the bounded constructible derived category $D_{\mf X}^b(X)$ then
\begin{align}
H^i((\mc D\mc S[-n])_x)&\cong H^{i-n}(f_x^*\mc D\mc S)
\cong H^{i-n}(\mc D(f^!_x\mc S))\notag\\
&\cong \Hom(H^{n-i}(f^!_x\mc S),R)\oplus \Ext(H^{n-i+1}(f^!_x\mc S),R).\label{E: dual stalk}
\end{align}
The second isomorphism is due to \cite[Theorem V.10.17]{Bo}, and for the last we use the universal coefficient theorem for Verdier duality \cite[Theorem 3.4.4]{BaIH} and that $\Hom(H^{n-i}(f^!_x\mc S),R)$ is free, as $H^{n-i}(f^!_x\mc S)$ is finitely generated because $\mc S$ is constructible by assumption (see \cite[Section V.3.3.iii]{Bo}).

\begin{proposition}\label{P: dual coeff}
Suppose $\mc E$ is a $\wp$-coefficient system on an $n$-manifold $M$, and let $D\wp=\primeset{P}(R)-\wp$, the complementary set of primes. Then $\mc D\mc E[-n]$ is a $D\wp$-coefficient system. Furthermore, 
 \begin{equation*}
 \mc H^i(\mc D\mc E[-n])_x\cong
 \begin{cases}
TH^0(\mc E_x),&i=1,\\
R^{\text{rank}(H^0(\mc E_x))}\oplus TH^1(\mc E_x), &i=0.
\end{cases}
\end{equation*}
\end{proposition}

\begin{proof}
As $\mc E$ is $\mf X$-cc by definition,  the complex $\mc D\mc E$ is also $\mf X-cc$ by \cite[Corollary V.8.7]{Bo} and so each $\mc H^i(\mc D\mc E)$ is finitely generated and locally constant. Now let $\ell_x:x\into U_1$. Then $ \mc H^i(\mc D\mc E[-n])_x\cong \Hom(H^{-i}(\mc E_x),R)\oplus \Ext(H^{-i+1}(\mc E_x),R)$
by \eqref{E: dual stalk} and using that $f_x^!\cong f_x^*[-n]$ on a manifold \cite[Proposition V.3.7.b]{Bo}. 
 The result follows by our observations about $\Hom$ and $\Ext$ in Section \ref{S: algebra}.
\end{proof}

As a ts-coefficient system on $X$ is simply a sheaf complex on $M$ that restricts to a $\wp$-coefficient system on each connected component of $M$ (for some $\wp$ that might vary by stratum), the following corollary is immediate:

\begin{corollary}\label{C: dual coeff}
If $\mc E$ is a ts-coefficient system then  $\mc D\mc E[-n]$ is a ts-coefficient system.
\end{corollary}

\begin{proof}[Proof of Theorem \ref{T: duality}]
The preceding corollary shows that if $\mc E$ is a ts-coefficient system then $\mc D\mc E[-n]$ is a ts-coefficient system.
Thus, as in \cite{GM2,Bo}, it suffices to verify that $\mc D \ms P_{\vec p,\mc E}[-n]$ satisfies the axioms for $\ms P_{D\vec p,\mc D\mc E[-n]}$. However, we do not have available the reformulation into a version of the Goresky-MacPherson axioms Ax2, so our proof will have to proceed a bit differently from those in \cite{GM2, Bo}; instead we emulate the proof of  \cite[Theorem 3.2]{CS91} and utilize the axioms TAx1'.

\emph{Constructibility.} By \cite[Corollary V.8.7]{Bo}, $\mc D \ms P_{\vec p,\mc E}$ is $\mf X$-clc and $\mf X$-cc because $\ms P_{\vec p,\mc E}$ is by Theorem \ref{T: cc}.

\emph{Axiom TAx1'\eqref{A: coeffs}.}
Let $i:U_1\into X$ be the inclusion. Since $U_1$ is open in $X$, $i^!=i^*$, and thus if $\mf D_X$ is the Verdier dualizing sheaf on $X$, $i^*\mf D_X=i^!\mf D_X=\mf D_{U_1}$. Now for any sheaf complex $\mc S$ we have $\mc D\mc S\cong R\text{\emph{Hom}}(\mc S,\mf D)\cong \text{\emph{Hom}}(\mc S,\mf D)$, since $\mf D$ is injective in Borel's construction \cite[Corollary V.7.6]{Bo}. Furthermore, it is clear from the construction of the sheaf functor \emph{Hom} that $\text{\emph{Hom}}(\mc S,\mf D)|_{U_1}\cong \text{\emph{Hom}}(\mc S|_{U_1},\mf D_{U_1})\cong \mc D(\mc S|_{U_1})$. Thus since $\ms P_{\vec p,\mc E}|_{U_1}\cong \mc E$, it follows that $(\mc D \ms P_{\vec p,\mc E}[-n])|_{U_1}\cong\mc D\mc E[-n]$. This demonstrates axiom TAx1'\eqref{A: coeffs}.

\emph{Axiom TAx1'\eqref{A': truncate}.}
Next, let $x\in Z\subset X_{n-k}$, $k>0$. Let $f_x:x\into X$  be the inclusion map, and let us abbreviate 
$\ms P_{\vec p,\mc E}$ as simply $\ms P$. 
Then
$H^i((\mc D\ms P[-n])_x)\cong \Hom(H^{n-i}(f^!_x\ms P),R)\oplus \Ext(H^{n-i+1}(f^!_x\ms P),R)$  by \eqref{E: dual stalk}.
Since $\ms P$ satisfies TAx1'$(X,\vec p,\mc E)$,
we know
$H^i(f_x^!\ms P)=0$ for $i\leq p_1(Z)+n-k+1$ and $H^{p_1(Z)+n-k+2}(f_x^!\ms P)$ is $\vec p_2(Z)$-torsion free. 
Thus  $H^i((\mc D\ms P[-n])_x)=0$ for $n-i+1\leq p_1(Z)+n-k+1$, i.e. for $i\geq k-p_1(Z)=D\vec p_1(Z)+2$. Furthermore, 
\begin{align*}
H^{D\vec p_1(Z)+1}((\mc D\ms P[-n])_x)&\cong  \Hom(H^{n-D\vec p_1(Z)-1}(f^!_x\ms P),R)\oplus \Ext(H^{n-\vec Dp_1(Z)}(f^!_x\ms P),R)\\
&=\Hom(H^{\vec p_1(Z)+n-k+1}(f^!_x\ms P),R)\oplus \Ext(H^{\vec p_1(Z)+n-k+2}(f^!_x\ms P),R)\\
&=\Ext(H^{\vec p_1(Z)+n-k+2}(f^!_x\ms P),R)
\end{align*}
Since $H^{\vec p_1(Z)+n-k+2}(f^!_x\ms P)$ is finitely generated, again by the constructibility of $\ms P$, and since it has no $\vec p_2(Z)$-torsion,
$H^{D\vec p_1(Z)+1}((\mc D\ms P[-n])_x)$ must then consist entirely of $D\vec p_2(Z)$-torsion. 

 This demonstrates TAx1'\eqref{A': truncate}.

\emph{Axiom TAx1'\eqref{A': attach}.}
Next, by \cite[Proposition V.8.2]{Bo} and \cite[Theorem 3.4.4]{BaIH} we have 
\begin{equation*}
H^i(f_x^!\mc D\ms P[-n])\cong H^{i-n}(f_x^!\mc D\ms P)
\cong H^{i-n}(\mc D \ms P_x) \cong \Hom(H^{n-i}(\ms P_x),R)\oplus \Ext(H^{n-i+1}(\ms P_x),R).
\end{equation*}
Since $\ms P$ satisfies TAx1'$(X,\vec p,\mc E)$, we know that 
$H^i(\ms P_x)=0$ for $i>\vec p_1(Z)+1$ and $H^{\vec p_1(Z)+1}(\ms P_x)$ is $\vec p_2(Z)$-torsion. 
This immediately implies  $H^i(f_x^!\mc D\ms P[-n])=0$ if $n-i>p_1(Z)+1$, i.e. if $i\leq n-\vec p_1(Z)-2=D\vec p_1(Z)+n-k$. Furthermore, if $i=D\vec p_1(Z)+n-k+1$, then $n-i=\vec p_1(Z)+1$, and we still have $n-i+1>p_1(Z)+1$, so  $H^{D\vec p_1(Z)+n-k+1}(f_x^!\mc D\ms P[-n])\cong \Hom(H^{\vec p_1(Z)+1}(\ms P_x),R)$. But $H^{\vec p_1(Z)+1}(\ms P_x)$ is torsion by the axioms, so $H^{D\vec p_1(Z)+n-k+1}(f_x^!\mc D\ms P[-n])$ also vanishes.

It remains to show that 
$H^{D\vec p_1(Z)+n-k+2}(f_x^!\mc D\ms P[-n])$ is $D\vec p_2(Z)$-torsion free. 
From our formula above,
$$H^{D\vec p_1(Z)+n-k+2}(f_x^!\mc D\ms P[-n])\cong \Hom(H^{-D\vec p_1(Z)+k-2}(\ms P_x),R)\oplus \Ext(H^{-D\vec p_1(Z)+k-1}(\ms P_x),R).$$
As all modules are finitely generated from the constructibility assumptions, the torsion subgroup will be the $\Ext$ summand. 
The module $H^{-D\vec p_1(Z)+k-1}(\ms P_x)=H^{\vec p_1(Z)+1}(\ms P_x)$ is $\vec p_2(Z)$-torsion by the axioms for $\ms P_x$, so all the torsion of $H^{D\vec p_1(Z)+n-k+2}(f_x^!\mc D\ms P[-n])$ is $\vec p_2(Z)$-torsion. As $\vec p_2(Z)$ and $D\vec p_2(Z)$ are complementary sets of prime, this shows that $H^{D\vec p_1(Z)+n-k+2}(f_x^!\mc D\ms P[-n])$ is $D\vec p_2(Z)$-torsion free.

This verifies Axiom TAx1'\eqref{A': attach}.

\emph{Axiom TAx1'\eqref{A: bounded}.}
It follows from $\mc D\mc E[-n]$ being a $ts$-coefficient system and from TAx1'\eqref{A': truncate}, which we have already proven, that $\mc D\ms P[-n]$ is bounded above (up to quasi-isomorphism). 
We need to demonstrate 
that $H^i((\mc D\ms P[-n])_x)=0$ for $i<0$, and hence complete axiom TAx1'\eqref{A: bounded}. 
We have 
$$
H^i((\mc D\ms P[-n])_x)\cong \dlim_{x\in U} \H^{i-n}(U;\mc D\ms P)
\cong  \dlim_{x\in U}  \Hom(\H^{n-i}_c(U;\ms P),R)\oplus \Ext(\H^{n-i+1}_c(U;\ms P),R).
$$
So  it suffices to show that for any neighborhood $U$ of $x$ we have $\H^{j}_c(U;\ms P)=0$ for $j>n+1$ and $\H^{n+1}_c(U;\ms P)=0$ is torsion. But this follows from Theorem \ref{T: vanishing}, taking $\wp$ there to be the set of all primes, as $\ms P$ satisfies the axioms TAx1. 
\end{proof}

\subsection{Self-duality}\label{S: self dual}

By Theorem \ref{T: duality}, we know that $\ms P_{\vec p,\mc E}$ is dual to $\ms P_{D\vec p,\mc (\mc D\mc E)[-n]}$, i.e.\ that $\mc D\ms P_{\vec p,\mc E}[-n]\cong \ms P_{D\vec p,(\mc D\mc E)[-n]}$ in the derived category $D^b(X)$. The next natural question is when we have self-duality, i.e.\ that $\mc D\ms P_{\vec p,\mc E}[-n]\cong\ms P_{\vec p,\mc E}$, possibly up to further degree shifts. Such situations lead to further invariants such as signatures. Applying Theorem \ref{T: duality}, such self-duality occurs when $\ms P_{\vec p,\mc E}$ is quasi-isomorphic to $\ms P_{D\vec p,\mc (\mc D\mc E)[-n]}$, up to shifts. 

For the classical Deligne sheaves $\mc P_{\bar p,\mc E}$, with $\bar p$ a perversity in the standard sense and $\mc E$ a local system, it is well known that such self-duality can always be achieved, say for constant coefficients on orientable pseudomanifolds, by imposing strong enough conditions on the space $X$. For example, if $X$ is a trivially-stratified manifold, then  $\mc P_{\bar p,\mc E}$ is independent of $\bar p$, and so all Deligne sheaves with the same coefficient systems are isomorphic. Hence the usual focus is on finding the minimal conditions on a space that will ensure self-duality for some perversity. In this setting, we have $D\bar p(Z)=\codim(Z)-\bar p(Z)-2$, so $D\bar p=\bar p$ implies $\bar p(Z)=\frac{\codim(Z)-2}{2}$. Of course this is not possible if $X$ has strata of odd codimension, as perversities take integer values, so one looks instead at the next most general case, the dual lower- and upper-middle perversities\footnote{There is no reason we are forced to make a consistent rounding choice at each odd codimension stratum, but it is convenient as there are canonical maps $\mc P_{\bar p,\mc E}\to \mc P_{\bar q,\mc E}$ whenever $\bar p(Z)\leq \bar q(Z)$ for all $Z$.} defined by $\bar m(Z)=\left\lfloor\frac{\codim(Z)-2}{2}\right\rfloor$ and $\bar n(Z)=\left\lceil\frac{\codim(Z)-2}{2}\right\rceil$, and asks for conditions for which  $\mc P_{\bar m,\mc E}\cong \mc P_{\bar n,\mc E}$. So let us see what we can do along these lines.

In the torsion-sensitive setting, we first observe how $\mc D$ behaves on the ts-coefficient systems $\mc E$. Recall that at each $x$ we have by definition that $H^i(\mc E_x)=0$ unless $i=0,1$, and in these cases
$H^1(\mc E_x)$ is $\wp$-torsion for some $\wp\in \primeset{P}(R)$ while $H^0(\mc E_x)$ is $\wp$-torsion free.
By Proposition \ref{P: dual coeff}, taking $\mc E$ to $(\mc D\mc E)[-n]$ results (cohomologically) in interchanging the degrees of the torsion subgroups. In other words, we saw $\mc H^0(\mc D\mc E[-n])_x\cong R^{\text{rank}(H^0(\mc E_x))}\oplus TH^1(\mc E_x)$ and $\mc H^1(\mc D\mc E[-n])_x\cong TH^0(\mc E_x)$,
and we know that $TH^0(\mc E_x)$ and $TH^1(\mc E_x)$ cannot have torsion with respect to the same primes.
Thus it is not possible for $\mc E$ to be isomorphic  to $\mc D\mc E$ (up to shifts) unless either \begin{enumerate}
\item $H^*(\mc E_x)$ is torsion-free for all $x$, or 
\item $H^*(\mc E_x)$ is nontrivial in only one degree (either $0$ or $1$), where it must be a torsion module.
\end{enumerate}
And the latter case requires an additional degree shift.

In what follows, we will consider each of these cases individually. First, however, as we will be writing conditions in terms of links, it will be useful to make the following observation, again generalizing a known result for the usual Deligne sheaves. The following lemma says that the restriction of a ts-Deligne sheaf to a link is a ts-Deligne sheaf of the link. In the statement of the lemma, we let $\vec p$ stand also for its own restriction to $L$. In other words, if $\mc Z$ is a codimension $j$ stratum of $L$ then $\mc Z$ is contained in a codimensions $j$ stratum $Z$ of $X$ and we set $\vec p(\mc Z)=\vec p(Z)$. We also write $\mc E|_L$ rather than the more correct $\mc E|_{L-L^{k-2}}$.

\begin{lemma}\label{L: link sheaf}
Let $X$ be a stratified pseudomanifold, $\vec p$ a ts-perversity, and $\mc E$ a ts-coefficient system. Suppose $x\in X_{n-k}$ and $L=L^{k-1}$ is a link of $X$ at $x$ that we can assume embedded in $X$ via some distinguished neighborhood $V\cong \R^{n-k}\times cL$ of $x$ and given the induced filtration. Then $\ms P_{X,\vec p,\mc E}|_L\cong\ms P_{L,\vec p,\mc E|_{L}}$.
\end{lemma}
\begin{proof}

Note that $\mc E|_{L}$ satisfies the same stalk conditions on $L^{k-2}$ as  $\mc E$ does on $X-X^{n-1}$, and so it is a ts-coefficient system on $L$.

Now let $\ms P=\ms P_{X,\vec p,\mc E}$.
To prove the lemma, it suffices by Theorems \ref{T: axioms} and \ref{T: ax equiv} to show that $\ms P|_L$ satisfies the axioms TAx1' on $L$ (Definition \ref{T: Ax1'}). As stalk cohomology and quasi-isomorphisms commute with restriction, the only condition that is not immediate is the cosupport condition. 

We identify $V-X_{n-k}$ with $\R^{n-k+1}\times L$ and identify $L$ with $\{u\}\times L$ for some $u\in\R^{n-k+1}$; let $r:L\into \R^{n-k+1}\times L$ be the embedding.
We first consider the restriction $\ms P|_{V-X_{n-k}}$. Since all of the axioms TAx1 are local, this is also a ts-Deligne sheaf on $V-X_{n-k}\cong \R^{n-k+1}\times L$ and so satisfies the axioms. We can thus work with $\ms P|_{V-X_{n-k}}$, which we relabel to $\ms P$ for convenience. Then we have
$\ms P|_L\cong r^*\ms P\cong r^!\ms P[n-k+1]$; see Lemma \ref{L: nns} below. 

Now let $z$ be a point in a stratum of $L$ of codimension $\ell$. The point $z$ also lives in a stratum of codimension $\ell$ of $V-X_{n-k}$. Let $f_z$ and $g_z$ be the inclusions of $z$ into $V-X_{n-k}$ and $L$, respectively, so that $rg_z=f_z$. Then
$$
H^i(g_z^!(\ms P|_L))\cong H^i(g_z^!r^*\ms P)
\cong H^i(g_z^!r^!\ms P[n-k+1])
\cong H^{i+n-k+1}(f_z^!\ms P).
$$
As $\ms P$ satisfies the axioms TAx1', we have $H^j(f_z^!\ms P)=0$ for $j\leq \vec p_1(Z)+n-\ell+1$ and is $\vec p_2(Z)$-torsion free for $j= \vec p_1(Z)+n-\ell+2$. So $H^i(g_z^!(\ms P|_L))=0$ for $i+n-k+1\leq \vec p_1(Z)+n-\ell+1$, i.e.\ for $i\leq \vec p_1(Z)+k-\ell=\vec p_1(Z)+(k-1)-\ell+1$, and it is $\vec p_2(Z)$-torsion free for $i= \vec p_1(Z)+(k-1)-\ell+2$. As $\dim(L)=k-1$, and $\vec p(\mc Z)=\vec p(Z)$ by definition, $\ms P|_L$ satisfies the cosupport property on $L$. 
\end{proof}

The property that $r^*$ and $r^!$ agree up to shifts for an embedding $r:L\into L\times \R^m$ seems to be well known, but the author could not find a proper citation. So here is an argument based on formulas in \cite{KS}:
\begin{lemma}\label{L: nns}
Let $X$ be a stratified pseudomanifold, and let  $E=X\times \R^m$  with filtration $\mf E$ given by the product filtration $E^j=X^{j-m}\times \R^m$. Let $\pi:E\to X$ be the projection and $r:X\into E$ the inclusion of the zero section. Let $\mc S\in D^+(E)$ be $\mf E$-clc on $E$. Then $r^!\mc S\cong r^*\mc S[-m]\in D^+(X)$. 
\end{lemma}
\begin{proof}
The map $\pi$ is a topological submersion \cite[Definition 3.31]{KS} with fiber dimension $m$, and $\pi r=\id$ is a topological submersion \cite[Definition 3.31]{KS} with fiber dimension $0$. So by \cite[Proposition 3.3.4.iii]{KS}, taking the input sheaf to be $R\pi_*\mc S$, we have 
$$r^!\pi^*R\pi_*\mc S\cong r^*\pi^*R\pi_*\mc S\otimes or_{X/X}\otimes r^*or_{E/X}[-m],$$
where $or$ is the relative orientation sheaf \cite[Definition 3.3.3]{KS}. 
By \cite[Lemma V.10.14.i]{Bo}, we have $\pi^*R\pi_*\mc S\cong \mc S$. If $R_X$ is the constant sheaf on $X$ with stalk $R$, then by \cite[Equation 3.3.2]{KS} and \cite[Definition 3.1.16.i]{KS} we have  $or_{X/X}\cong \id^! R_X\cong \id^*R_X=R_X$ and $or_{E/X}[-m]\cong \pi^!R_X[-2m]$. 

Now, let $p:E\to R^m$ be the projection. Taking $F=R_{\R^m}$ and $G=R_X$ in \cite[Proposition 3.4.4]{KS}, and simplifying by using $R\operatorname{\emph{Hom}}(R_Y,\mc S)\cong \mc S$ on any space $Y$, gives us (cf.\ \cite[Lemma 1.13.11]{Sch00}) \[\pi^!R_X\cong \mf D_{\R^m}\overset{L}{\boxtimes} R_X=p^*\mf D_X\overset{L}{\otimes}\pi^*R_X,\]
where $\mf D$ denotes the dualizing complex (written $\omega$ in \cite{KS}).
 But $\pi^*R_X\cong R_E$, while  $\mf D_{\R^m}\cong or_{\R^m}[m]\cong  R_{\R^m}[m]$ by \cite[Equation 3.3.2 and Proposition 3.3.6]{KS}. So $\pi^!R_X[-2m]\cong (R_E[m]\overset{L}{\otimes}R_E)[-2m]\cong R_E[-m]$. Altogether then we get \[r^!\mc S\cong r^*\mc S\otimes R_X\otimes r^*R_E[-m]\cong r^*\mc S\otimes R_X\otimes R_X[-m]\cong r^*\mc S[-m].\qedhere\] 
\end{proof}

\subsubsection{Torsion-free coefficients}\label{S: TFC}

In this section we suppose $H^i(\mc E_x)$ is trivial unless $i=0$, in which case it is free and finitely generated. We can thus assume that $\mc E$ is in fact a local system of finitely-generated free modules concentrated in degree $0$. We also assume that $\mc E\cong (\mc D\mc E)[-n]$, for example if $\mc E$ is constant and $X$ is orientable \cite[Section V.7.10]{Bo}. 

Now suppose that $\vec m$ is some ts-perversity with $\vec m_1=\bar m$ and that $\vec n=D\vec m$. To simplify the notation we will write $\ms P^{\vec m}_k= \ms P_{X,\vec m,\mc E}|_{U_k}$ and $\ms P^{\vec n}_k= \ms P_{X,\vec n,\mc E}|_{U_k}$.
Suppose that $\ms P^{\vec m}_k\cong \ms P^{\vec n}_k$ in $D^+(U_k)$. We will examine what conditions are needed to extend this isomorphism to $U_{k+1}$, i.e.\ over the strata $Z\subset X_{n-k}$. By construction, we know that $\ms P^{\vec m}_{k+1}=\mf t_{\leq\vec m}^{X_{n-k}}Ri_{k*}\ms P^{\vec m}_{k}$, with $i_k:U_k\into U_{k+1}$ the inclusion, and similarly for $\ms P^{\vec n}_{k+1}$. So, given the isomorphism over $U_k$ and Lemma \ref{L: properties}, the issue comes down to when the truncations  $\mf t_{\leq\vec n}^{X_{n-k}}Ri_{k*}\ms P^{\vec n}_{k}$ and $\mf t_{\leq\vec m}^{X_{n-k}}Ri_{k*}\ms P^{\vec m}_{k}\cong \mf t_{\leq\vec m}^{X_{n-k}}Ri_{k*}\ms P^{\vec n}_{k}$ produce the same results on the strata of $X_{n-k}$. This, in turn, comes down to analyzing the behavior of stalks over points $x\in Z\subset X_{n-k}$. If $x$ is such a point with link $L$ and $\vec p$ is any ts-perversity then by \cite[Lemma V.3.9 and Proposition V.3.10.b]{Bo} and Lemma \ref{L: link sheaf} we know that
$$\mc H^i\left(Ri_{k*}\ms P^{\vec p}_{k}\right)_x\cong \dlim_{x\in U} \H^{i}\left(U;Ri_{k*}\ms P^{\vec p}_{k}\right)\cong \H^{i}\left(L;\ms P^{\vec p}_{k}|_L\right)\cong \H^{i}\left(L;\ms P^{\vec p}_{L}\right),$$ where $\ms P^{\vec p}_{L}$ is the ts-Deligne sheaf on $L$ with perversity and coefficients restricted from $X$ (see the discussion prior to Lemma \ref{L: link sheaf}).  
Therefore, 

\begin{equation}
\mc H^i\left(\mf t_{\leq\vec p}^{X_{n-k}}Ri_{k*}\ms P^{\vec p}_{k}\right)_x\cong 
\begin{cases}
0,& i>\vec p\left(Z\right)+1,\\
T^{\vec p_2\left(Z\right)}H^i\left(L;\ms P^{\vec p}_L\right),& i=\vec p\left(Z\right)+1,\\\label{E: link co}
H^i\left(L;\ms P^{\vec p}_L\right),&i\leq \vec p\left(Z\right).
\end{cases}
\end{equation}
We need to see what conditions ensure that these groups agree for the perversities $\vec m$ and $\vec n$.

There are two cases to consider depending on the parity of the codimension of $Z$.
Since we assume that $\ms P^{\vec m}_k\cong \ms P^{\vec n}_k$, these complexes have a shared injective resolution $\mc I$ that we can use for computing 
$\mf t_{\leq\vec n}^{X_{n-k}}Ri_{k*}\ms P^{\vec n}_{k}$ and $\mf t_{\leq\vec m}^{X_{n-k}}Ri_{k*}\ms P^{\vec n}_{k}$
as $\mf t_{\leq\vec n}^{X_{n-k}}i_{k*}\mc I$ and $\mf t_{\leq\vec m}^{X_{n-k}}i_{k*}\mc I$, respectively. To compare the complexes, we make this assumption in what follows.

\textbf{$\codim(Z)$ is even.} In this case, $\bar m(Z)=\bar n(Z)$, so the truncation dimensions agree. Hence using the assumed isomorphism over $U_k$, the two perversities give us isomorphic stalk cohomology at $x\in Z$ if and only if $T^{\vec m_2\left(Z\right)}\H^{\bar n\left(Z\right)+1}\left(L;\ms P^{\vec n}_L\right)\cong T^{\vec n_2\left(Z\right)}\H^{\bar n\left(Z\right)+1}\left(L;\ms P^{\vec n}_L\right)$. Since $\vec n_2=D\vec m_2$, this happens if and only if these modules vanish, i.e.\ if $\H^{\bar n\left(Z\right)+1}\left(L;\ms P^{\vec n}_L\right)$ is torsion free. Furthermore, if $\H^{\bar n\left(Z\right)+1}\left(L;\ms P^{\vec n}_L\right)$ is torsion free then at each $x\in Z$ we have quasi-isomorphisms 
$\ttau^\emptyset_{\leq \bar n(Z)}i_{k*}\mc I\to   \mf t_{\leq\vec n}^{X_{n-k}}i_{k*}\mc I$ and $\ttau^\emptyset_{\leq \bar n(Z)}i_{k*}\mc I\to\mf t_{\leq\vec m}^{X_{n-k}}i_{k*}\mc I$ induced by the inclusion of sheaf complexes as in Lemmas \ref{L: pretrunc} and \ref{L: sheaf ttau} (localizing this argument to points of $Z$ as in Section \ref{S: local trunc}), and so $\mf t_{\leq\vec n}^{X_{n-k}}Ri_{k*}\ms P^{\vec n}_{k}\cong \mf t_{\leq\vec m}^{X_{n-k}}Ri_{k*}\ms P^{\vec n}_{k}$.
Thus we recover the Goresky-Siegel locally torsion free condition and its consequences \cite{GS83}.

\textbf{$\codim(Z)$ is odd.} In this case $\bar n(Z)=\bar m(Z)+1$ and so, as sheaf complexes, $ \mf t_{\leq\vec m}^{X_{n-k}}i_{k*}\mc I\subset \mf t_{\leq\vec n}^{X_{n-k}}i_{k*}\mc I$. In fact, from the definitions, $\mf t_{\leq\vec m}^{X_{n-k}}i_{k*}\mc I=\mf t_{\leq\vec m}^{X_{n-k}}\mf t_{\leq\vec n}^{X_{n-k}}i_{k*}\mc I$, and the inclusion induces cohomology maps at points in $Z$ as in Lemma \ref{L: sheaf ttau} (on $U_k$ the inclusion restricts to the identity). 
Thus we have a quasi-isomorphism  $\mf t_{\leq\vec n}^{X_{n-k}}Ri_{k*}\ms P^{\vec n}_{k}\cong \mf t_{\leq\vec m}^{X_{n-k}}Ri_{k*}\ms P^{\vec n}_{k}$ if and only if 
 $T^{\vec n_2(Z)}\H^{\bar n(Z)+1}\left(L;\ms P^{\vec n}_L\right)=0$ and $\H^{\bar n(Z)}\left(L;\ms P^{\vec n}_L\right)\cong T^{\vec m_2(Z)}\H^{\bar n(Z)}\left(L;\ms P^{\vec n}_L\right)$. In fact, the second condition implies the first as we now show:

We have assumed that $\ms P^{\vec n}_k\cong \ms P^{\vec m}_k \cong \mc D\ms P^{\vec n}_k[-n]$ on $U_k$. Consequently, if we fix a link $L$ in a distinguished neighborhood $V$ of $x\in Z\subset X_{n-k}$, we have that $V-Z\cong L\times \R^{n-k+1}$ and so 
\begin{align}
\H^i\left(L;\ms P^{\vec n}_L\right)&\cong \H^i\left(L;\ms P^{\vec n}_k|_L\right)\cong \H^i\left(V-Z; \ms P^{\vec n}_k\right)\cong \H^{i-n}\left(V-Z; \mc D\ms P^{\vec n}_k\right)\label{E: link condition}\\
&\cong \Hom\left(\H^{n-i}_c\left(V-Z; \ms P^{\vec n}_k\right),R\right)\oplus \Ext\left(\H^{n-i+1}_c\left(V-Z; \ms P^{\vec n}_k\right),R\right)\notag\\
&\cong \Hom\left(\H^{k-i-1}_c\left(L; \ms P^{\vec n}_L\right),R\right)\oplus \Ext\left(\H^{k-i}_c\left(L;\ms P^{\vec n}_L\right),R\right).\notag
\end{align}
The first two isomorphisms are by Lemma \ref{L: link sheaf} and \cite[Remark V.3.4]{Bo}. The third isomorphism is by assumption. The last two are by \cite[Theorem 3.4.4]{BaIH},  \cite[Lemma V.3.8]{Bo}, Lemma \ref{L: link sheaf}, and that the $\H^{\ell}_c\left(V-Z; \ms P^{\vec n}_k\right)\cong \H^{\ell-(n-k+1)}_c\left(L; \ms P^{\vec n}_L\right)$ are finitely generated owing to the compactness of $L$ and the constructibility of $\ms P^{\vec n}_L$ (Theorem \ref{T: cc} and \cite[Remark V.3.4]{Bo}).

So if $\codim(Z)=k=2j+1$, we have $\bar n(Z)=\left\lceil\frac{2j-1}{2}\right\rceil=j$ and therefore
$$T\H^{\bar n\left(Z\right)+1}\left(L;\ms P^{\vec n}_L\right)=T\H^{j+1}\left(L;\ms P^{\vec n}_L\right)\cong \Ext\left(\H^{j}_c\left(L; \ms P^{\vec n}_L\right),R\right)\cong T\H^{j}\left(L; \ms P^{\vec n}_L\right),$$
using that $L$ is compact and again that each $\H^\ell(L;\ms P^{\vec n}_L)$ is finitely generated. So if we assume that $\H^{\bar n(Z)}(L;\ms P^{\vec n}_L)=\H^{j}(L;\ms P^{\vec n}_L)$ is $\vec m_2(Z)$-torsion, then $T\H^{\bar n(Z)+1}\left(L;\ms P^{\vec n}_L\right)$ will also be $\vec m_2(Z)$-torsion and so $T^{\vec n_2(Z)}\H^{\bar n(Z)+1}\left(L;\ms P^{\vec n}_L\right)=0$ as $\vec n_2(Z)=D\vec m_2(Z)$. 

Altogether then, for $\vec m$ and $\vec n$ to give the same modules on the extension to odd codimension strata the condition is that $\H^{\bar n(Z)}\left(L;\ms P^{\vec n}_L\right)\cong T^{\vec m_2(Z)}\H^{\bar n(Z)}\left(L;\ms P^{\vec n}_L\right)$. If we choose the least restrictive possibility with $\vec m_2(Z)=\primeset{P}(R)$ and $\vec n_2(Z)=\emptyset$, then we can obtain a self-dual extension so long as $\H^{\bar n(Z)}\left(L;\ms P^{\vec n}_L\right)$ is a torsion module. This condition is essentially the Cappell-Shaneson torsion condition for superduality \cite{CS91}, but applied only to the odd codimension strata.

\textbf{Conclusion for torsion-free coefficients.} Putting together the preceding paragraphs, we obtain the following conclusion. As in Lemma \ref{L: link sheaf}  we let $\vec n$ and $\mc E|_L$ denote the restrictions of $\vec n$ and $\mc E$ to $L$. The last statement is due to Example \ref{E: GM}, as $\bar n$ is a nonnegative and nondecreasing function of codimension.

\begin{theorem}\label{T: torsion free duality}
Suppose $X$ is an $n$-dimensional stratified pseudomanifold, that $\vec n$ is a ts-perversity satisfying $\vec n_1=\bar n$, and that $\mc E$ is a coefficient system with finitely-generated torsion-free stalks that satisfies $\mc E\cong\mc D\mc E[-n]$.
Then  $\ms P_{\vec n}=\ms P_{X,\vec n,\mc E}$ satisfies $\ms P_{\vec n}\cong \mc D\ms P_{\vec n}[-n]$ if and only if the following conditions hold:

\begin{enumerate}

\item If $L$ is a link of a point in a stratum of codimension $2j$ then $\H^{j}\left(L;\ms P_{L,\vec n,\mc E|_L}\right)$ is torsion-free.

\item If $L$ is a link of a point in a stratum of codimension $2j+1$ then 
$\H^{j}\left(L;\ms P_{L,\vec n,\mc E|_L}\right)$ is $D\vec n_2(Z)$-torsion. 
 
\end{enumerate}

In particular, taking $\vec n=(\bar n, \emptyset)$, the ordinary Deligne-sheaf $\mc P_{X,\bar n,\mc E}$ satisfies $\mc P_{X,\bar n,\mc E}\cong \mc D\mc P_{X,\bar n,\mc E}[-n]$ if and only if $\H^{j}\left(L;\mc P_{L, \bar n,\mc E|_L}\right)\cong I^{\bar n}H_{j-1}\left(L;\mc E|_L\right)$ is torsion-free for each link $L^{2j-1}$
 and  $\H^{j}\left(L;\mc P_{L,\bar n,\mc E|_L}\right)\cong I^{\bar n}H_{j}\left(L;\mc E|_L\right)$ is a torsion module for each link $L^{2j}$. 
\end{theorem}

\begin{remark}
The construction of self-dual spaces in Goresky-Siegel \cite[Section 7]{GS83} (called IP spaces in Pardon \cite{Pa90}) requires that $I^{\bar n}H_{j-1}(L;\mc E|_L)$ be torsion-free for each link $L^{2j-1}$ while 
 $I^{\bar n}H_{j}(L;\mc E|_L)=0$ for each link $L^{2j}$. The spaces in Cappell-Shaneson \cite{CS91} have only even codimension strata and/or link intersection homology modules that are all torsion. The last statement of Theorem \ref{T: torsion free duality} includes both of these classes of spaces, exposing a more general class on which upper-middle perversity intersection homology is self-dual. As far as we know, it has not been observed previously that self-duality extends to such spaces. We will consider this class of spaces further in future work.
\end{remark}

\subsubsection{Torsion coefficients}

Next we consider self-duality when the coefficient system has torsion stalks. Cappell and Shaneson first observed examples of such dualities in \cite[pages 340-341]{CS91}. Our shift degrees will be slightly different from theirs owing to a difference in conventions; see Remark \ref{R: CS convention}.

The following lemma shows that when we work with torsion coefficient systems our link cohomology is always torsion. 

\begin{lemma}\label{L: all torsion}
Suppose $X$ is a compact stratified pseudomanifold and that $\mc E$ is a ts-coefficient system on $X$ such that $H^i\left(\mc E_x\right)$ is a torsion module for all $x,i$. Then $\H^i\left(X;\ms P_{\vec p,\mc E}\right)$ is a torsion module for all $i$. 
\end{lemma}
\begin{proof}
Let $\ms P=\ms P_{\vec p,\mc E}$. We first show that $\H^i(X;\ms P)$ is torsion if all the $H^i\left(\ms P_x\right)$ are torsion. In fact, $\H^i(X;\ms P)$ is the abutment of a spectral sequence with $E_2^{p,q}\cong H^p(X;\mc H^q(\ms P))$ \cite[Section V.1.4]{Bo}. Let $Q(R)$ be the field of fractions of $R$. Then $H^p(X;\mc H^q(\ms P))\otimes Q(R)\cong H^p(X;\mc H^q(\ms P)\otimes Q(R))=0$ by \cite[Theorem II.15.3]{Br} and our hypotheses (we also use here that $X$ is compact). So each $E_\infty^{p,q}$ is also a torsion module. The module $\H^i(X;\ms P)$ is therefore the end result of a finite sequence of extensions of torsion modules by torsion modules, so it follows that each $\H^i(X;\ms P)$ is a torsion module.

We can now proceed by an induction on the depth of $X$. If $X$ has depth $0$ then $\ms P=\mc E$ so the result follows from the preceding paragraph. Now suppose we have shown the lemma for $X$ of depth $<K$ and that $X$ has depth $K$. By the preceding paragraph it suffices to show that $H^i\left(\ms P_x\right)$ is always torsion. This is true over $U_1$ by assumption as $\ms P|_{U_1}\cong \mc E$. If $x\in X_{n-k}$ then we have  $\mc H^i(\ms P)_x\subset \H^i(L;\ms P|_L)$ by \eqref{E: link co} and the construction of $\ms P$. As $\ms P|_L$ is the ts-Deligne sheaf with the restricted ts-perversity and ts-coefficient system by Lemma \ref{L: link sheaf} and as $L$ is compact and has depth less than that of $X$, these modules are $R$-torsion by the induction hypothesis.
\end{proof}

In this section we suppose $H^i(\mc E)_x$ is trivial unless $i=0$, in which case it is a finitely-generated torsion module. We can thus assume that $\mc E$ is in fact a local system of such modules concentrated in degree $0$.
By Proposition \ref{P: dual coeff}, we know that $\mc D\mc E[-n]_x\cong \mc E[-1]_x$. So our global duality assumption for coefficients throughout this section is that $\mc D\mc E[-n]\cong \mc E[-1]$. This will hold, for example, if $\mc E$ is constant and $X$ is orientable. 

By Theorem \ref{T: duality}, we have 
$\mc D\ms P_{\vec p,\mc E}[-n]\cong \ms P_{D\vec p,\mc D\mc E[-n]}\cong \ms P_{D\vec p,\mc E[-1]}.$
But if we start with $\mc E[-1]$ as our coefficient system and form the ts-Deligne sheaf $\ms P_{\vec q,\mc E[-1]}$, then it is not hard to see from the definitions that $\ms P_{\vec q,\mc E[-1]}=\ms P_{\vec q^{\,-},\mc E}[-1]$, where $\vec q^{\,-}(Z)=\left(\vec q_1(Z)-1,\vec q_2(Z)\right)$ on the singular stratum $Z$. 
So we obtain $\mc D\ms P_{\vec p,\mc E}[-n]\cong \ms P_{(D\vec p)^-,\mc E}[-1]$. 
Thus, matching coefficient degrees, we see in this case that self-duality up to shifts means 
\begin{equation}\label{E: sd tor}
\ms P_{\vec p,\mc E}\cong \mc D\ms P_{\vec p,\mc E}[-n+1]\cong \ms P_{(D\vec p)^-,\mc E}.
\end{equation}
So we must ask when $\ms P_{\vec p,\mc E}\cong \ms P_{(D\vec p)^-,\mc E}$.

If $\vec p_1(Z)=(D\vec p_1)^-(Z)$, then $\vec p_1(Z)=\codim(Z)-2-\vec p_1(Z)-1$ and so $\vec p_1(Z)=\frac{\codim(Z)-3}{2}$. If $\codim (Z)=2k+1$, we thus get $\vec p_1(Z)=k-1=\bar m(Z)$. If $\codim(Z)=2k$, then of course $\frac{\codim(Z)-3}{2}=\frac{2k-3}{2}$ is not valid. 
If we round up, we get $\left\lceil\frac{2k-3}{2}\right\rceil=k-1=\bar m(Z)$. If we round down, we get $\left\lfloor\frac{2k-3}{2}\right\rfloor=k-2=\bar m(Z)-1$.  Let us write this as $$\bar \mu(Z)=\left\lfloor\frac{\codim(Z)-3}{2}\right\rfloor.$$ Thus we will consider ts-perversities $\vec \mu$ and $\vec m$ such that  $\vec m_1=\bar m$, $\vec \mu_1=\bar \mu$, and $\vec m_2(Z)=D\vec \mu_2(Z)$ for all singular strata $Z$.
Adopting our notation from the torsion-free coefficient case, we will examine when an isomorphism $\ms P^{\vec \mu}_k\cong \ms P^{\vec m}_k$ in $D^+(U_k)$ can be extended to $U_{k+1}$ by considering the cohomology stalks over points of strata $Z\subset X_{n-k}$. As in that case, we assume that $\ms P^{\vec m}_k, \ms P^{\vec \mu}_k$ have a shared injective resolution $\mc I$ and treat
$\mf t_{\leq\vec \mu}^{X_{n-k}}Ri_{k*}\ms P^{\vec \mu}_{k}$ and $\mf t_{\leq\vec m}^{X_{n-k}}Ri_{k*}\ms P^{\vec m}_{k}$
as $\mf t_{\leq\vec \mu}^{X_{n-k}}i_{k*}\mc I$ and $\mf t_{\leq\vec m}^{X_{n-k}}i_{k*}\mc I$, respectively.

\textbf{$\codim(Z)$ is odd.} In this case $\bar \mu(Z)=\bar m(Z)$, so the truncation dimensions agree.  Thus using the assumed isomorphism over $U_k$,  the stalk cohomologies will agree over $x\in Z$ if and only if $T^{\vec \mu_2\left(Z\right)}\H^{\bar m\left(Z\right)+1}\left(L;\ms P^{\vec m}_{L}\right)\cong T^{\vec m_2\left(Z\right)}\H^{\bar m\left(Z\right)+1}\left(L;\ms P^{\vec m}_{L}\right)$. But $\vec \mu_2(Z)=D\vec m_2(Z)$, so this happens only if these modules vanish. If $\codim(Z)=2j+1$ this means that we must have 
$T^{\vec \mu_2\left(Z\right)}\H^{j}\left(L;\ms P^{\vec m}_{L}\right)=T^{\vec m_2\left(Z\right)}\H^{j}\left(L;\ms P^{\vec m}_{L}\right)=0$. As $\vec \mu_2(Z)$ and $\vec m_2(Z)$ are complementary, this is equivalent to $\H^{j}\left(L;\ms P^{\vec m}_{L}\right)$ being torsion free. But by Lemma \ref{L: all torsion} this is equivalent to  $\H^{j}\left(L;\ms P^{\vec m}_{L}\right)=0$.
This requirement is met vacuously in \cite[pages 340-341]{CS91} due to the assumption there that $X$ have only even-codimension strata. Conversely, if this condition holds, then as in the even codimension case in Section \ref{S: TFC}, both $\mf t_{\leq\vec \mu}^{X_{n-k}}i_{k*}\mc I$ and $\mf t_{\leq\vec m}^{X_{n-k}}i_{k*}\mc I$ are isomorphic to the $\ttau_{\leq j-1}^\emptyset$ truncation of $i_{k*}\mc I$ over $X_{n-k}$, and so 
$\mf t_{\leq\vec \mu}^{X_{n-k}}Ri_{k*}\ms P^{\vec \mu}_{k}\cong \mf t_{\leq\vec m}^{X_{n-k}}Ri_{k*}\ms P^{\vec m}_{k}$

\textbf{$\codim(Z)$ is even.} In this case $\bar m(Z)=\bar \mu(Z)+1$. So again as in the analogous case in Section \ref{S: TFC}
 $\mf t_{\leq\vec \mu}^{X_{n-k}}i_{k*}\mc I=\mf t_{\leq\vec \mu}^{X_{n-k}}\mf t_{\leq\vec m}^{X_{n-k}}i_{k*}\mc I$ with the inclusion inducing cohomology maps at points in $Z$ as in Lemma \ref{L: sheaf ttau}.
Thus this inclusion will be a quasi-isomorphism if and only if we have $T^{\vec m_2\left(Z\right)}\H^{\bar m\left(Z\right)+1}\left(L;\ms P^{\vec m}_{L}\right)=0$ and $\H^{\bar m\left(Z\right)}\left(L;\ms P^{\vec m}_{L}\right)\cong T^{\vec \mu_2\left(Z\right)}\H^{\bar m\left(Z\right)}\left(L;\ms P^{\vec m}_{L}\right)$. If $\codim\left(Z\right)=2j$, then $\bar \mu\left(Z\right)=j-2$, $\bar m\left(Z\right)=j-1$, and $\dim\left(L\right)=2j-1$. So the conditions become $T^{\vec m_2\left(Z\right)}\H^{j}\left(L;\ms P^{\vec m}_{L}\right)=0$ and $\H^{j-1}\left(L;\ms P^{\vec m}_{L}\right)\cong T^{\vec \mu_2\left(Z\right)}\H^{j-1}\left(L;\ms P^{\vec m}_{L}\right)$.

We use again the computation \eqref{E: link condition}, replacing $-n$ with $-n+1$, taking $\codim(Z)=k=2j$, and recalling that all modules are torsion. This results in
$\H^j\left(L;\ms P^{\vec m}_{L}\right)\cong \Ext\left(\H^{j-1}_c\left(L; \ms P^{\vec m}_{L}\right),R\right)\cong \H^{j-1}\left(L;\ms P^{\vec m}_{L}\right)$. So the conditions are equivalent to these isomorphic modules being $\mu_2(Z)$-torsion.
If $\vec m_2(Z)=\emptyset$ and so $\vec \mu_2(Z)=\primeset{P}(R)$, then this condition is always true.
This is the situation utilized in \cite[pages 340-341]{CS91}. At the other extreme, if $\vec m_2(Z)=\primeset{P}(R)$ and $\vec \mu_2(Z)=\emptyset$, the requirement becomes that $\H^{j}\left(L;\ms P^{\vec m}_{L}\right)=\H^{j-1}\left(L;\ms P^{\vec m}_{L}\right)=0$.  

\textbf{Conclusion for torsion coefficients.}
Putting together the preceding paragraphs, we obtain the following conclusion. As in Lemma \ref{L: link sheaf}  we let $\vec m$ and $\mc E|_L$ denote the restrictions of $\vec m$ and $\mc E$ to $L$. The last statement is due to Example \ref{E: GM}, as $\bar m$ is a nonnegative and nondecreasing function of codimension when $X$ has no codimension one strata.

\begin{theorem}
Suppose $X$ is an $n$-dimensional stratified pseudomanifold, that $\vec m$ is a ts-perversity satisfying $\vec m_1=\bar m$, and that $\mc E$ is a coefficient system with finitely-generated torsion stalks that satisfies $\mc E\cong \mc D\mc E[-n+1]$.
Then  $\ms P_{\vec m}=\ms P_{X,\vec m,\mc E}$ satisfies $\ms P_{\vec m}\cong\mc D\ms P_{\vec m}[-n+1]$ if and only if the following conditions hold:

\begin{enumerate}

\item If $L$ is a link of a point in a stratum of codimension $2j+1$ then $\H^{j}\left(L;\ms P_{L,\vec m,\mc E|_L}\right)=0$.

\item If $L$ is a link of a point in a stratum of codimension $2j$ then 
$\H^{j}\left(L;\ms P_{L,\vec m,\mc E|_L}\right)$ is $D\bar m_2(Z)$-torsion.
 
\end{enumerate}

In particular, taking $\vec m=(\bar m, \emptyset)$ and assuming $X$ has no codimension one strata, the ordinary Deligne-sheaf $\mc P_{X,\bar m,\mc E}$ satisfies $\mc P_{X,\bar m,\mc E}\cong \mc D\mc P_{X,\bar m,\mc E}[-n+1]$ if and only $\H^{j}(L;\mc P_{L, \bar m,\mc E|_L})\cong I^{\bar m}H_{j-1}\left(L;\mc E|_L\right)=0$ for each link $L^{2j-1}$. 
\end{theorem}

\begin{remark}
If $X$ does have codimension one strata, the hypotheses of the theorem cannot be satisfied nontrivially, as in this case the link $L$ will be $0$-dimensional, meaning that $\H^{0}\left(L;\ms P_{L,\vec m,\mc E|_L}\right)\cong \H^{0}(L;\mc E)$ cannot always be $0$ unless $\mc E=0$.  In this case $\ms P=0$. 

Similarly, in Theorem \ref{T: torsion free duality}, if $X$ has codimension one strata then we obtain conditions requiring that $\H^{0}(L;\mc E)$ be $D\vec n_2(Z)$-torsion. As $\mc E$ is assumed to have torsion-free stalks for that theorem, again this only happens if $\ms P=0$. 
\end{remark}

\begin{remark}
If we begin instead with $\mc H^i(\mc E)=0$ for $i\neq 1$, so that we consider the local torsion system $\mc E$ in degree $1$, we obtain an equivalent condition to that studied above. In this case $\mc D\mc E[-n]\cong\mc E[1]$, so we have $\mc D\ms P_{\vec p,\mc E}[-n]\cong \ms P_{D\vec p,\mc D\mc E[-n]}\cong \ms P_{D\vec p,\mc E[1]}.$  
So here for the degrees of the coefficients to agree the self-duality equation must become $\ms P_{\vec p,\mc E}\cong \mc D\ms P_{\vec p,\mc E}[-n-1]\cong \ms P_{D\vec p,\mc E[1]}[-1]$. 
But also $\ms P_{\vec p,\mc E}\cong \ms P_{\vec p^{\,-} ,\mc E[1]} [-1]$, so this becomes 
$\ms P_{\vec p^{\,-} ,\mc E[1]} [-1]\cong \ms P_{D\vec p,\mc E[1]}[-1]$.
Replacing $\vec p$ with $\vec q^{\,+}$, noting that $D(\vec q^{\,+})=(D\vec q)^-$, and shifting  gives 
$\ms P_{ \vec q,\mc E[1]} \cong \ms P_{(D\vec q)^-,\mc E[1]}.$
But now this is precisely the same condition as \eqref{E: sd tor}.
\end{remark}

\begin{remark}\label{R: CS convention}
In \cite{CS91}, the convention is to define the Deligne sheaves with the coefficients in degree $-n$. In this case, if $\mc E$ is a local torsion system in degree $-n$ and $\ms Q_{\vec p,\mc E}$ is the corresponding Deligne sheaf, the duality statement becomes $\mc D\ms Q_{\vec p,\mc E}[n]\cong \ms Q_{D\vec p,\mc D\mc E[n]}$, with $\mc D\mc E[n]\cong \mc E[-1]$ living in degree $-n+1$. So here self-duality becomes
$$\ms Q_{\vec p,\mc E}\cong \mc D\ms Q_{\vec p,\mc E}[n+1]\cong \ms Q_{D\vec p,\mc E[-1]}[1]\cong 
\ms Q_{(D\vec p)^-,\mc E}.$$ 
In particular, the shift necessary from $\mc D\ms Q$ to $\ms Q$ becomes $[n+1]$; cf.\ \cite[pages 340-341]{CS91}.
\end{remark}

\section{Torsion-sensitive t-structures and ts-perverse sheaves}\label{S: perverse}

In this section we consider our ts-Deligne sheaves within the broader abstract setting of perverse sheaves and t-structures.
The primary source for perverse sheaves is \cite{BBD}. Good expository references include \cite{BaIH, KS, DI04, Bhatt15}. Many of the arguments in this section are variants of arguments that can be found in these texts. 

\subsection{The natural torsion-sensitive t-structure}

We begin by building a torsion-sensitive t-structure on the derived category of sheaf complexes on a stratified pseudomanifold. In this section we consider a generalization of the natural t-structure \cite[Example 1.3.2]{BBD}. In the next section we will glue across strata.

\begin{definition}\label{D: natural t}
Let $X$ be a stratified pseudomanifold, $R$ a PID, $\wp$ a set of primes of $R$, and $D(X)$ the derived category of complexes of sheaves of $R$-modules on $X$. We define strictly full subcategories ${}^\wp D^{\leq 0}(X)$ and ${}^\wp D^{\geq 0}(X)$ of $D(X)$ with objects
\begin{align*}
Ob({}^\wp D^{\leq 0}(X))&=\{\mc S\in D(X)\mid \text{$\forall x\in X$, $H^i(\mc S_x)=0$ for $i>1$ and $H^1(\mc S_x)$ is $\wp$-torsion}\}\\
Ob({}^\wp D^{\geq 0}(X))&=\{\mc S\in D(X)\mid\text{$\forall x\in X$, $H^i(\mc S_x)=0$ for $i<0$ and $H^0(\mc S_x)$ is $\wp$-torsion free}\}.
\end{align*}
We call $({}^\wp D^{\leq 0}(X),{}^\wp D^{\geq 0}(X))$ the \emph{natural $\wp$-t-structure} and denote the heart by ${}^\wp D^{\heartsuit}(X)={}^\wp D^{\leq 0}(X)\cap {}^\wp D^{\geq 0}(X)$.

We similarly obtain $t$-structures by restricting to the subcategories $D^+(X)$, $D^{-}(X)$, or $D^{b}(X)$, consisting respectively of sheaves with cohomology bounded below, bounded above, or bounded, or by restricting to the subcategories $D_{\mf X}(X)$, $D_{\mf X}^+(X)$, $D_{\mf X}^{-}(X)$, or $D_{\mf X}^{b}(X)$ consisting of complexes that are additionally $\mf X$-cc. 
\end{definition}

\begin{proposition}\label{P: t}
 $({}^\wp D^{\leq 0}(X),{}^\wp D^{\geq 0}(X))$ is a $t$-structure on $D(X)$. Similarly, the restrictions to the subcategories mentioned in Definition \ref{D: natural t} are t-structures. 
\end{proposition}
\begin{proof}
We must check the three conditions to be a $t$-structure  \cite[Definition 1.3.1]{BBD}:

First let ${}^\wp D^{\leq n}={}^\wp D^{\leq 0}[-n]$ and ${}^\wp D^{\geq n}={}^\wp D^{\geq 0}[-n]$. Then it is immediate from the definitions that ${}^\wp D^{\leq -1}\subset {}^\wp D^{\leq 0}$ and ${}^\wp D^{\geq 1}\subset {}^\wp D^{\geq 0}$.

Next we must show that $\Hom_{D(X)}(\mc S,\mc T)=0$ if $\mc S\in {}^\wp D^{\leq 0}$ and $\mc T\in {}^\wp D^{\geq 1}$. Let $\mc H^*(\mc A)$ denote the cohomology sheaf complex of the sheaf complex $\mc A$. From the definitions, we note that $\mc H^i(\mc S)=0$ for $i>1$ and $\mc H^i(\mc T)=0$ for $i<1$, so by \cite[Proposition 8.1.8]{BaIH} (see also \cite[Lemma V.9.13]{Bo}) we have an isomorphism 
$\Hom_{D(X)}(\mc S,\mc T)\cong  \Hom_{Sh(X)}(\mc H^1(\mc S),\mc H^1(\mc T)),$
where $Sh(X)$ is the category of sheaf complexes on $X$. But at each $x\in X$ we have that $\mc H^1(\mc S)_x\cong H^1(\mc S_x)$ is $\wp$-torsion while 
$\mc H^1(\mc T)_x$ is $\wp$-torsion free. Since any sheaf map would have to take $T^\wp\mc H^1(\mc S)_x=\mc H^1(\mc S)_x$ to  $T^\wp\mc H^1(\mc T)_x=0$, it follows, as desired, that $\Hom_{D(X)}(\mc S,\mc T)=\Hom_{Sh(X)}(\mc H^1(\mc S),\mc H^1(\mc T))=0.$

For the last condition, we must show that to every $\mc S\in D(X)$ we can associate a distinguished triangle $$\mc A\to\mc S\to\mc C\xr{+1}$$ such that $\mc A\in {}^\wp D^{\leq 0}$ and $\mc C\in {}^\wp D^{\geq 1}$. For this, we consider the exact sequence in $Sh(X)$ 
$$0\to \ttau^{\wp}_{\leq 0} \mc S\xr{f} \mc S\xr{g} \mc S/\ttau^{\wp}_{\leq 0} \mc S\to 0,$$
in which $f$ is our standard inclusion and $g$ is the quotient map. Such a short exact sequence determines a distinguished triangle with the same complexes and the same maps $f,g$ \cite[Section 2.4]{BaIH}, and $\ttau^{\wp}_{\leq 0} \mc S\in {}^\wp D^{\leq 0}$ by construction. Taking cohomology and looking at stalks results in isomorphisms $\mc H^i(\ttau^{\wp}_{\leq 0} \mc S)_x\to \mc H^i(\mc S)_x$ for $i\leq 0$, so $\mc H^i(\mc S/\ttau^{\wp}_{\leq 0} \mc S)_x=0$ for $i<0$ and $\mc H^0(\mc S/\ttau^{\wp}_{\leq 0} \mc S)_x\to \mc H^1(\ttau^{\wp}_{\leq 0} \mc S)_x$ is injective.
So near degree $1$ the next portion of the exact sequence looks like
\begin{diagram}
\mc H^0(\mc S/\ttau^{\wp}_{\leq 0} \mc S)_x&\rInto& \mc H^1(\ttau^{\wp}_{\leq 0} \mc S)_x&\rTo^f& \mc H^1(\mc S)_x&\rTo^g& \mc H^1(\mc S/\ttau^{\wp}_{\leq 0} \mc S)_x&\rTo& \mc H^2(\ttau^{\wp}_{\leq 0} \mc S)_x=0
\end{diagram}
We know by Lemma \ref{L: sheaf ttau} that $f$ is an isomorphism onto the $\wp$-torsion submodule of $\mc H^1(\mc S)_x$, so it follows that $\mc H^0(\mc S/\ttau^{\wp}_{\leq 0} \mc S)_x=0$ and $\mc H^1(\mc S/\ttau^{\wp}_{\leq 0} \mc S)_x\cong \mc H^1(\mc S)_x/T^\wp\mc H^1(\mc S)_x$. So $\mc H^1(\mc S/\ttau^{\wp}_{\leq 0} \mc S)_x$ is $\wp$-torsion free.

For the last statement of the proposition, we observe that the preceding arguments are identical restricting to the various subcategories.
\end{proof}

\begin{remark}\label{R: natural heart}
For a fixed $\wp\subset \primeset{P}(R)$, the heart ${}^\wp D^{\heartsuit}(X)$ consists of sheaf complexes $\mc S$ such that each $H^1(\mc S_x)$ is $\wp$-torsion, each $H^0(\mc S_x)$ is $\wp$-torsion free, and all other cohomology is trivial. 
If $\wp=\emptyset$ then ${}^\emptyset D^{\heartsuit}(X)$ consists of those sheaf complexes $\mc S$ on  $X$ such that $\mc H^i(\mc S)=0$ for $i\neq 0$; this is equivalent to the category $Sh(X)$ \cite[Example 10.1.3]{KS}. The intersection ${}^\emptyset D^{\heartsuit}(X)\cap D_{\mf X}(X)$ consists (up to quasi-isomorphisms) of the sheaves of finitely-generated $R$-modules that are locally constant on each stratum. For arbitrary $\wp$ and $X$ stratified trivially ${}^\wp D^{\heartsuit}(X)\cap D_{\mf X}(X)$ consists precisely of the $\wp$-coefficient systems of Definition \ref{D: ts-coeff}.
\end{remark}

\begin{remark}\label{R: components}
It will be convenient to have a generalization of our natural t-structure that behaves differently on different connected components of a disconnected space. This will be used below in gluing arguments to avoid an infinite number of gluings for spaces with an infinite number of strata.

Suppose $X$ is the disjoint union $X=\amalg_{\alpha \in A} Y_\alpha$ for some indexing set $A$ and that we have a 
function $\vec q=(\vec q_1,\vec q_2):A\to \Z\times \P(\primeset{P}(R))$. This is slightly different from a ts-perversity in the sense of previous sections, though below such a $\vec q$ will arise as the restriction of a perversity to all the strata of a single dimension. 
Then we let
\begin{align*}
Ob({}^{\vec q} D^{\leq 0}(X))&=\{\mc S\in D(X)\mid \text{$\forall \alpha,\forall x\in Y_\alpha$, $H^i(\mc S)_x=0$ for $i>\vec q_1(Y_\alpha)+1$}\\&\qquad\text{and $H^{\vec q_1(Y_\alpha)+1}(\mc S)_x$ is $\vec q_2(Y_\alpha)$-torsion}\}\\
Ob({}^{\vec q} D^{\geq 0}(X))&=\{\mc S\in D(X)\mid\text{$\forall \alpha,\forall x\in Y_\alpha$, $H^i(\mc S)_x=0$ for $i<\vec q_1(Y_\alpha)$}\\&\qquad\text{and $H^{\vec q_1(Y_\alpha)}(\mc S)_x$ is $\vec q_2(Y_\alpha)$-torsion free}\}.
\end{align*}
This is also a $t$-structure by the same arguments as for Proposition \ref{P: t}, using a different torsion tipped truncation on each component.
Note that $\mc S\in {}^{\vec q} D^{\leq 0}(X)$ if and only if $\mc S|_{Y_\alpha}\in  {}^{\vec q_2(Y_\alpha)} D^{\leq \vec q_1(Y_\alpha)}(Y_\alpha)$ for all $\alpha$, and similarly reversing the inequalities.
\end{remark}

\paragraph{Convention.}
In what follows we will work only within the derived category  $D_{\mf X}^b(X)$, though we will omit the decorations from the already cluttered notation for the $t$-structures. 

\subsection{Torsion sensitive perverse sheaves}

In this section we build a t-structure that takes stratification into account. Though we are primarily interested in stratified pseudomanifolds, it will be useful to allow spaces slightly more general by dropping the requirement that $X-X^{n-1}$ be dense. Such spaces are said to have \emph{topological stratifications} in \cite[Section 1.1]{GM2}, while they are called \emph{unrestricted stratifications} in the Remark of \cite[Section V.2.1]{Bo}. As noted there by Borel, the constructibility and Verdier duality properties of sheaves on stratified pseudomanifolds extend to spaces with these unrestricted stratifications. 
We formulate the definitions of this section in this greater generality. We call such spaces unrestricted stratified pseudomanifolds, and we maintain the notation $U_k=X-X^{n-k}$ and $X_{n-k}=X^{n-k}-X^{n-k-1}=U_{k+1}-U_k$, the definition of strata, etc.
If $X$ is an unrestricted stratified pseudomanifold then so is each $X^m$ and each $X-X^m$, which is our primary reason for considering such spaces; this is not the case for the usual stratified pseudomanifolds. Of course all stratified pseudomanifolds are also unrestricted stratified pseudomanifolds.

Now that we have chosen our spaces, we need to extend our notion of perversity to include data on all strata. 

\begin{definition}\label{D: extended perv}
Let $X$ be an unrestricted stratified pseudomanifold.
Let an \emph{extended torsion-sensitive perversity} (or simply  \emph{extended ts-perversity}) be a function $\vec p: \{\text{strata of $X$}\}\to \Z\times \P(\primeset{P}(R))$.

Given a ts-perversity (Definition \ref{D: ts-perv}) and a ts-coefficient system $\mc E$ (Definition \ref{D: ts-coeff}) on a stratified pseudomanifold, we let $\vec p_{\mc E}$ denote the extended ts-perversity given by 
\begin{equation*}
\vec p_{\mc E}=
\begin{cases}
(0,\wp(Z,\mc E)),&\text{$Z$ a regular stratum},\\
\vec p(Z),&\text{$Z$ a singular stratum}.
\end{cases}
\end{equation*}
If $Y$ is a union of strata of $X$, we also write $\vec p$ for the restriction of $\vec p$ to the strata of $Y$.
\end{definition}

If $\vec p$ is an extended ts-perversity on $X$, then on each $X_{n-k}$ we have the $t$-structure $({}^{\vec p} D^{\leq 0}(X_{n-k}),{}^{\vec p} D^{\geq 0}(X_{n-k}))$ 
given in Remark \ref{R: components}; on individual strata, these are shifts of the $t$-structures defined in  Definition \ref{D: natural t} \cite[Remark 7.1.2]{BaIH}.
We claim that for each $k\geq 1$ the inclusions $U_k\xhookrightarrow{i} U^{k+1}\xhookleftarrow{j} X_{n-k}$ and the resulting functors $i_!,i^*=i^!,Ri_*,j^*, j_*=j_!,j^!$ among the derived categories\footnote{We here let $D^b_{\mf X}(U_k)$ denote the bounded derived category of complexes that are cohomologically constructible with respect to the stratification of $U_k$ induced from $X$, and similarly for the other subspaces. We also use the notation of \cite{Bo} in letting $j^!$ stand directly for the functor between derived categories.}   $D^b_{\mf X}(U_k)$, $D^b_{\mf X}(U_{k+1})$, and $D^b_{\mf X}(X_{n-k})$ provide gluing data; see \cite[Section 1.4]{BBD} or \cite[Theorem 7.2.2 and Section 7.3]{BaIH}. This will allow us to iteratively glue the $t$-structures over the various strata.

\begin{lemma}\label{L: gluing}
The functors $i_!,i^*=i^!,Ri_*,j^*, j_*=j_!,j^!$ among the derived categories   $D^b_{\mf X}(U_k)$, $D^b_{\mf X}(U_{k+1})$, and $D^b_{\mf X}(X_{n-k})$ provide gluing data.
\end{lemma}

\begin{proof}
The necessary adjunction properties hold already as functors among the $D^+(U_k)$, $D^+(U_{k+1})$, and $D^+(X_{n-k})$; see \cite[Section 7.2.1]{BaIH}. Therefore we need only show that these functors preserve boundedness and constructibility. This is clear for the restrictions and the extensions by $0$. For $Ri_*$ and $j^!$ these properties appear to be well known (e.g. they are implicit in \cite[Section 1.12]{GM2}). Detailed arguments can be made using the results in \cite[Section V.3]{Bo}, particularly \cite[V.3.3 and V.3.9-11]{Bo}, and the spectral sequence for hypercohomology with compact supports \cite[Section V.1.4]{Bo}.
\end{proof}

\begin{definition}
Let $X$ be an $n$-dimensional unrestricted stratified pseudomanifold and $\vec p$ an extended ts-perversity. Let $({}^{\vec p} D^{\leq 0}(X),{}^{\vec p} D^{\geq 0}(X))$ denote the $t$-structure on $D^b_{\mf X}(X)$ obtained by iterative gluing of the $t$-structures  $({}^{\vec p} D^{\leq 0}(X_{n-k}),{}^{\vec p} D^{\geq 0}(X_{n-k}))$ on $D^b_{\mf X}(X_{n-k})$ for each $k\geq 0$.  We call this the \emph{$\vec p$-perverse $t$-structure}. We denote the heart of this $t$-structure by 
$${}^{\vec p} D^{\heartsuit}(X) = {}^{\vec p} D^{\leq 0}(X)\cap{}^{\vec p} D^{\geq 0}(X).$$
If $\mc S\in {}^{\vec p} D^{\heartsuit}(X)$ for some $\vec p$, we call $\mc S$ \emph{ts-perverse}.
\end{definition}

We can describe the objects of $({}^{\vec p} D^{\leq 0}(X),{}^{\vec p} D^{\geq 0}(X))$ explicitly using the definition of the gluing procedure. Recall \cite[Theorem 7.2.2]{BaIH} that in general if we have gluing data  $U\xhookrightarrow{i} X\xhookleftarrow{j} F$ with $U$ open and $F=X-U$ and with t-structures $(D_U^{\leq 0},D_U^{\geq 0})$ and $(D_F^{\leq 0},D_F^{\geq 0})$ on $D(U)$ and $D(F)$, then the glued $t$-structure on $D(X)$ satisfies
\begin{align*}
D^{\leq 0}_X&=\{S\in D(X)\mid i^*S\in D_U^{\leq 0}, j^*S\in D_F^{\leq 0}\}\\
D^{\geq 0}_X&=\{S\in D(X)\mid i^*S\in D_U^{\geq 0}, j^!S\in D_F^{\geq 0}\}.
\end{align*}
Therefore,  recalling that $i^*=i^!$ when $i$ is an open inclusion, we have the following by an easy induction argument and the definitions:

The next proposition now follows directly from the definitions.

\begin{proposition}\label{P: perverse t conditions}
Suppose $\mc S\in D^b_{\mf X}(X)$. Then $\mc S\in {}^{\vec p} D^{\leq 0}(X)$ if and only if for each inclusion $j_k:X_{n-k}\into X$, $k\geq 0$, we have $j_k^*\mc S\in {}^{\vec p} D^{\leq 0}(X_{n-k})$, and 
$\mc S\in {}^{\vec p} D^{\geq 0}(X)$ if and only if for each inclusion $j_k:X_{n-k}\into X$ we have $j_k^!\mc S\in {}^{\vec p} D^{\geq 0}(X_{n-k})$. In particular,

\begin{enumerate}
\item $\mc S\in {}^{\vec p} D^{\leq 0}(X)$ if and only if the following holds for all $x\in X$: if $x$ is contained in the stratum  $Z\subset X_{n-k}$ then $H^i(\mc S_x)=0$ if $i>\vec p_1(Z)+1$ and $H^{\vec p_1(Z)+1}(\mc S_x)$ is $\vec p_2(Z)$-torsion.

\item $\mc S\in {}^{\vec p} D^{\geq 0}(X)$ if and only if the following holds for all $x\in X$: if $x$ is contained in the stratum  $Z\subset X_{n-k}$ and $j_{k}:X_{n-k}\into X$ is the inclusion then $H^i((j_k^!\mc S)_x)=0$ if $i<\vec p_1(Z)$ and $H^{\vec p_1(Z)}((j_k^!\mc S)_x)$ is $\vec p_2(Z)$-torsion free. 
\end{enumerate}
\end{proposition}

If $\vec p_2(Z)=\emptyset$ for all $Z$ then ${}^{\vec p} D^{\heartsuit}(X)$ is the standard perverse $t$-structure \cite[page 158]{BaIH}.

\begin{remark}
It is an easy exercise, though mildly messy to write down, to show that the order of the inductive gluing does not matter. In other words, given t-structures on each $X_{n-k}$, whether we start with $U_1=X_n$ and then successively glue on $X_{n-1}$, $X_{n-2}$, etc. or if we start with $X^0=X_0$ and successively glue on $X_1$, $X_2$, etc., we arrive at the same conditions stated in Proposition \ref{P: perverse t conditions} for a t-structure on $X$. Alternatively, as in \cite[Section 2.1]{BBD}, we could take the conditions of Proposition \ref{P: perverse t conditions} to be the definition of $({}^{\vec p} D^{\leq 0}(X),{}^{\vec p} D^{\geq 0}(X))$ and then observe analogously to \cite[Proposition 2.1.3]{BBD} that for any $k>0$ this $t$-structure is obtained by gluing those defined inductively on $X^{n-k}$ and $U_k=X-X^{n-k}$. 
\end{remark}

\subsection{ts-Deligne sheaves as perverse sheaves}
 
In this section $X$ is a stratified pseudomanifold in the usual sense. 
We show that our ts-Deligne sheaves of Section \ref{S: TSD} are ts-perverse. 
Unfortunately, this involves some shifting of the perversities on the singular strata, which is a well-known issue; see, e.g., \cite[page 170]{BaIH} or \cite[pages 60-61]{BBD}. So we need some notation for this.

\begin{definition}
If $\vec p$ is an extended ts-perversity (Definition \ref{D: extended perv})  on the stratified pseudomanifold $X$, let 
$\vec p^{\,+}:\{\text{strata of }X\}\to \Z\times \P(\primeset{P}(R))$ be the extended ts-perversity given by
\begin{equation*}
\vec p^{\,+}(Z)= 
\begin{cases}
(\vec p_1(Z),\vec p_2(Z)),&\text{$Z$ a regular stratum},\\
(\vec p_1(Z)+1,\vec p_2(Z)),&\text{$Z$ a singular stratum}.\\
\end{cases}
\end{equation*}

In particular, if $\vec p$ is a ts-perversity (Definition \ref{D: ts-perv}) and $\mc E$ is a ts-coefficient system  on $U_1$, then $\vec p_{\mc E}^{\,+}$ is given by 
\begin{equation*}
\vec p_{\mc E}^{\,+}(Z)= 
\begin{cases}
(0,\wp(Z,\mc E)),&\text{$Z$ a regular stratum},\\
(\vec p_1(Z)+1,\vec p_2(Z)),&\text{$Z$ a singular stratum}.\\
\end{cases}
\end{equation*}
\end{definition}

\begin{remark}
Our notation $\vec p^{\,+}$ should be distinguished from the notation $p^+$ with a different meaning in \cite[Section 3.3]{BBD}.
\end{remark}

We can now show that ts-Deligne sheaves are ts-perverse sheaves. In fact, analogously to the classical case, they can be realized as intermediate extensions \cite[Section 1.4]{BBD}.

\begin{proposition}\label{P: deligne is ie}
Let $R$ be a PID, $X$ an $n$-dimensional stratified pseudomanifold, and $\vec p$ a ts-perversity. 
Let $\mc E$ be a ts-coefficient system on $U_1=X-X^{n-1}$, let $u:U_1\into X$ be the inclusion, and let $u_{!*}$ be the intermediate extension functor ${}^{\vec p_{\mc E}^{\,+}} D^{\heartsuit}(U_1)\to {}^{\vec p_{\mc E}^{\,+}} D^{\heartsuit}(X)$. 
Then the ts-Deligne sheaf $\ms P_{X,\vec p,\mc E}$ is isomorphic to $u_{!*}(\mc E)$ in ${}^{\vec p_{\mc E}^{\,+}} D^{\heartsuit}(X)$.
\end{proposition}
\begin{proof}
Recall that on unions of strata we abuse notation by letting  $\vec p_{\mc E}^{\,+}$ denote the extended ts-perversity obtained by restricting the domain of the original $\vec p_{\mc E}^{\,+}$.
For simplicity, we will write $\ms P_{X,\vec p,\mc E}$ as just $\ms P$ and $\vec p^{\,+}_{\mc E}$ as $\vec p^+$. 

By Theorem \ref{T: axioms}, the sheaf complex $\ms P$ satisfies the axioms  TAx1$(X,\vec p, \mc E)$, and 
Axiom \ref{A: coeffs} says that $u^*\ms P=\ms P|_{U_1}\cong \mc E$. So $\ms P$ is an extension of $\mc E$. Furthermore, we have  $\mc E\in {}^{\vec p^{\,+}} D^{\heartsuit}(U_1)$ from the definitions.

Now for the composition $ij$ of inclusions of locally closed subspaces it holds that $(ij)_{!*}=i_{!*}j_{!*}$ \cite[Equation 2.1.7.1]{BBD}. So to verify the claim that $\ms P\cong u_{!*}(\mc E)$, we can proceed by induction. Let $u^k:U_1\to U_{k}$  be the inclusion, and suppose we have that $\ms P|_{U_k}\cong u^k_{!*}(\mc E)$ for some $k\geq 1$ the base case $k=1$ being trivial. Let $i^k:U_k\into U_{k+1}$, with $U_{k+1}-U_k=X_{n-k}$, as usual. We wish to show that $\ms P|_{U_{k+1}}\cong i^k_{!*}(\ms P|_{U_k})$, from which it will follow that 
$$\ms P|_{U_{k+1}}\cong i^k_{!*}(\ms P|_{U_k})\cong i^k_{!*}u^k_{!*}(\mc E)\cong u^{k+1}_{!*}\mc E,$$
completing the proof by induction. Note that $u^{k+1}_{!*}\mc E$ is an object  of  ${}^{\vec p^{\,+}}D^{\heartsuit}(U_{k+1})$ by definition. 

Let $g_k:X_{n-k}\into U_{k+1}$. 
From the properties of intermediate extensions, $i^k_{!*}(\ms P|_{U_k})$ is the unique (up to isomorphism) extension of $\ms P|_{U_k}$ such that $g_k^*i^k_{!*}(\ms P|_{U_k})\in {}^{\vec p^{\,+}}D^{\leq -1}(X_{n-k})$ and $g_k^!i^k_{!*}(\ms P|_{U_k})\in {}^{\vec p^{\,+}}D^{\geq 1}(X_{n-k})$ \cite[Corollary 1.4.24]{BBD}. So it suffices to verify that $g_k^*(\ms P|_{U_{k+1}})\in {}^{\vec p^{\,+}}D^{\leq -1}(X_{n-k})$ and $g_k^!(\ms P|_{U_{k+1}})\in {}^{\vec p^{\,+}}D^{\geq 1}(X_{n-k})$. 

First consider $g_k^*(\ms P|_{U_{k+1}})=\ms P|_{X_{n-k}}$. 
By Axiom TAx1'\ref{A: truncate}, if $x\in Z\subset  X_{n-k}$ for $Z$ a stratum, then $H^i(\ms P_x)=0$ for $i>\vec p_1(Z)+1=\vec p^{\,+}_1(Z)$ and $H^{\vec p_1(Z)+1}(\ms P_x)=H^{\vec p^{\,+}_1(Z)}(\ms P_x)$ is $\vec p_2(Z)$-torsion. So $g_k^*(\ms P|_{U_{k+1}})\in {}^{\vec p^{\,+}}D^{\leq -1}(X_{n-k})$ by Remark \ref{R: components}.

Continuing to assume $x\in Z\subset  X_{n-k}$, we next consider $g_k^!(\ms P|_{U_{k+1}})$. As $w_{k+1}:U_{k+1}\into X$ is an open inclusion, we have $$g_k^!(\ms P|_{U_{k+1}})=g_k^!w_{k+1}^*\ms P=g_k^!w_{k+1}^!\ms P=(w_{k+1}g_k)^!\ms P.$$ So if we let $j_k$ denote the inclusion $j_k:X_{n-k}\into X$ then $g_k^!(\ms P|_{U_{k+1}})=j_k^!\ms P$. 
Lemma \ref{A''} now implies that since $\ms P$ is $\mf X$-cc (Theorem \ref{T: cc}) and satisfies the axioms  TAx1'$(X,\vec p, \mc E)$, it also satisfies $H^i((j_k^!\ms P)_x)=0$ for $i\leq \vec p_1(Z)+1=\vec p^{\,+}(Z)$  and $H^{\vec p_1(Z)+2}((j_k^!\ms P)_x)=H^{\vec p_1^{\,+}(Z)+1}((j_k^!\ms P)_x)$ is $\vec p_2(Z)$-torsion free. So $g_k^!(\ms P|_{U_{k+1}})\in {}^{\vec p^{\,+}}D^{\geq 1}(X_{n-k})$ by Remark \ref{R: components}. 

This completes the induction step and hence the proof.
\end{proof}

By the general properties of intermediate extensions \cite[page 55]{BBD}, we have the following:

\begin{corollary}\label{C: int ext}
$\ms P_{X,\vec p,\mc E}\cong u_{!*}(\mc E)$  is the unique (up to isomorphism) extension of $\mc E$ in ${}^{\vec p_{\mc E}^{\,+}} D^{\heartsuit}(X)$ possessing no nontrivial subobject or quotient object supported on $X^{n-1}$. 
Furthermore, the simple objects of ${}^{\vec p_{\mc E}^{\,+}} D^{\heartsuit}(X)$ are either of the form $\ms P_{X,\vec p,\mc E}$, for $\mc E$ a simple object among ts-coefficient systems, or $j_! \mc S^*$ for $\mc S^*$ simple in ${}^{\vec p_{\mc E}^{\,+}} D^{\heartsuit}(X^{n-1})$.
\end{corollary}

\subsection{Duality}\label{S: perv duality}

In this section we consider duality, once again allowing unrestricted stratified pseudomanifolds unless noted otherwise.
We first need a definition of dual perversity adapted to the ts-perverse setting, cf.\ Definition \ref{D: dual perv}:

\begin{definition}
Given an extended ts-perversity $\vec p$, we define the \emph{perverse dual extended ts-perversity $\D\vec p$} so that for a stratum $Z$,  
$(\D\vec p)(Z)=(\codim(Z)-\vec p_1(Z),D\vec p_2(Z))$, where $D\vec p_2(Z)$ continues to represent $\primeset{P}(R)-\vec p_2(Z)$, the complement of $\vec p_2(Z)$ in the set of primes (up to unit) of $R$. 
\end{definition}

Since we must always dualize and shift, we simplify the notation in this section as follows:

\begin{definition}
On an $n$-dimensional unrestricted stratified pseudomanifold, we define the shifted dualizing functor by $\ms D\mc S=\mc D\mc S[-n]$.
\end{definition}

\begin{remark}\label{R: dual p}
If $\vec p$ is a ts-perversity and $\mc E$ is a ts-coefficient system then $\D(\vec p^{\,+}_{\mc E})=(D\vec p)_{\ms D\mc E}^{\,+}$. Indeed, on the first component of these extended perversities, if $Z$ is regular both sides evaluate to $0$, while if $Z$ is singular we have 
\begin{multline*}
\D(\vec p^{\,+}_{\mc E})_1(Z)=\codim(Z)-\vec p^{\,+}_{\mc E,1}(Z)=\codim(Z)-\vec p_{1}(Z)-1\\
=\codim(Z)-2-\vec p_1(Z)+1=(D\vec p)^+_{\ms D\mc E,1}(Z).
\end{multline*}
On the second components, the agreement on singular strata is obvious, while we saw in Proposition \ref{P: dual coeff} that $\ms D\mc E$ is a ts-coefficient system with respect to the complementary set of primes to that of $\mc E$ on each regular stratum.
\end{remark}

Next we show that $\ms D$ takes ts-perverse sheaves to ts-perverse sheaves with respect to the dual perversity, beginning with the manifold case. Recall Definition \ref{D: natural t}.

\begin{lemma}\label{L: manifold dual}
Let $M$ be an $n$-manifold. If $\mc S\in {}^\wp D^{\leq m}(M)$, then $\ms D \mc S\in {}^{D\wp} D^{\geq -m}(M)$. If $\mc S\in {}^\wp D^{\geq m}(M)$, then $\ms D \mc S\in {}^{D\wp} D^{\leq -m}(M)$.
\end{lemma}
\begin{proof}
By \cite[Corollary V.8.7]{Bo}, if $\mc S$ is $\mf X$-cc then so is $\ms D\mc S$. Let $f_x:x\into M$. For any $\mc S$ we have 
 \begin{equation}\label{E: dual point}
 \mc H^i((\ms D\mc E)_x)\cong  \Hom(H^{-i}(\mc S_x),R)\oplus \Ext(H^{-i+1}(\mc S_x),R)
 \end{equation}
by \eqref{E: dual stalk} and using that $f_x^!\cong f_x^*[-n]$ on a manifold \cite[Proposition V.3.7.b]{Bo}. 
If $\mc S\in {}^\wp D^{\leq m}(M)$, then for all $x\in M$, $H^i(\mc S_x)=0$ for $i>m+1$ and $H^{m+1}(\mc S_x)$ is $\wp$-torsion. It follows that $H^i((\ms D\mc S)_x)=0$ if $i< -m$ and $H^{-m}((\ms D\mc S)_x)$ is $D\wp$-torsion free, i.e.\ $\mc D\mc S\in {}^{D\wp} D^{\geq -m}(M)$. Similarly, if $\mc S\in  {}^{\wp} D^{\geq m}(M)$, then $H^i(\mc S_x)=0$ if $i<m$ and $H^{m}(\mc S_x)$ is $\wp$-torsion free, and it follows that 
$H^i((\ms D\mc S)_x)=0$ if $i> -m+1$ and $H^{-m+1}((\ms D\mc S)_x)$ is $D\wp$-torsion. So $\ms D\mc S\in {}^{D\wp} D^{\leq -m}(M).$
\end{proof}

\begin{lemma}\label{L: dual strata}
Let $X$ be an $n$-dimensional unrestricted stratified pseudomanifold and $\vec p$ an extended ts-perversity on $X$. Let $j_k:X_{n-k}\into X$. 
If $j_k^*\mc S\in {}^{\vec p} D^{\leq 0}(X_{n-k})$ then  $j_k^!\ms D\mc S\in{}^{\D\vec p} D^{\geq 0}(X_{n-k})$, and if  $j_k^!\mc S\in{}^{\vec p} D^{\geq 0}(X_{n-k})$ then $j_k^*\ms D\mc S\in {}^{\D\vec p} D^{\leq 0}(X_{n-k})$.
\end{lemma}

\begin{proof}
Again by \cite[Corollary V.8.7]{Bo}, if $\mc S$ is $\mf X$-cc then so is $\mc D\mc S$, and the arguments of Lemma \ref{L: gluing} show that $j_k^*$ and $j_k^!$ take $D^b_{\mf X}(X)$ to $D^b_{\mf X}(X_{n-k})$. It suffices to consider the codimension $k$ strata $Z\subset X_{n-k}$ independently. Abusing notation consider $j_k:Z\into X$. 
Suppose $j_k^*\mc S\in {}^{\vec p} D^{\leq 0}(Z)={}^{\vec p_2(Z)} D^{\leq \vec p_1(Z)}(Z)$. Then by Lemma \ref{L: manifold dual}, $\mc Dj_k^*\mc S[-(n-k)]\cong j_k^!\mc D\mc S[-n+k]$ is in ${}^{D\vec p_2(Z)} D^{\geq -\vec p_1(Z)}(Z)$. So $j_k^!\mc D\mc S[-n]=j_k^!\ms D\mc S\in {}^{D\vec p_2(Z)} D^{\geq k-\vec p_1(Z)}(Z)={}^{\D\vec p} D^{\geq 0}(Z)$. 

The second case is completely analogous.
\end{proof}

\begin{theorem}\label{T: perverse duality}
Let $X$ be an $n$-dimensional unrestricted stratified pseudomanifold and $\vec p$ an extended ts-perversity on $X$.
If $\mc S\in {}^{\vec p} D^{\leq 0}(X)$ then  $\ms D\mc S\in{}^{\D\vec p} D^{\geq 0}(X)$, and if  $\mc S\in{}^{\vec p} D^{\geq 0}(X)$ then $\ms D\mc S\in {}^{\D\vec p} D^{\leq 0}(X)$. 
\end{theorem}
\begin{proof}
Let $j_k:X_{n-k}\into X$. If $\mc S\in {}^{\vec p} D^{\leq 0}(X)$ then by the gluing construction $j_k^*\mc S\in {}^{\vec p} D^{\leq 0}(X_{n-k})$ for all $k$. So by Lemma \ref{L: dual strata}, $j_k^!\ms D\mc S\in {}^{\D\vec p} D^{\geq 0}(X_{n-k})$ for all $k$. Thus $\ms D\mc S\in {}^{\D\vec p} D^{\geq 0}(X)$, by the definition of gluing. Similarly, if  $\mc S\in {}^{\vec p} D^{\geq 0}(X)$ then $j_k^!\mc S\in {}^{\vec p} D^{\geq 0}(X_{n-k})$ for all $k$. So by Lemma \ref{L: dual strata}, $j_k^*\ms D\mc S\in {}^{\D\vec p} D^{\leq 0}(X_{n-k})$ for all $k$. Thus $\ms D\mc S\in {}^{\D\vec p} D^{\leq 0}(X)$.
\end{proof}

\begin{corollary}\label{C: perverse duality}
The functor $\ms D: D^b_{\mf X}(X)\to D^b_{\mf X}(X)$ restricts to an equivalence of categories ${}^{\vec p} D^{\heartsuit}(X)\to {}^{\D\vec p} D^{\heartsuit}(X)^{opp}$, i.e.\ ${}^{\vec p} D^{\heartsuit}(X)$ and  ${}^{\D\vec p} D^{\heartsuit}(X)$ are dual categories.
\end{corollary}
\begin{proof}
The preceding theorem shows that $\ms D$  takes ${}^{\vec p} D^{\heartsuit}(X)$ to ${}^{\D\vec p} D^{\heartsuit}(X)$. Applying $\ms D$ twice gives $\mc D\mc D\mc S$. By \cite[Theorem V.8.10]{Bo}, $\mc D\mc D$ is isomorphic to the identity (the statement in \cite{Bo} is about objects, but the proof shows that $BD_X:\id\to \mc D\mc D$ is a natural transformation).
\end{proof}

The next corollary could be taken as a consequence of Theorem \ref{T: perverse duality} and our ts-Deligne sheaf duality theorem (Theorem \ref{T: duality}). Rather, we will prove it directly and then observe that it provides an alternate proof of Theorem \ref{T: duality}.

\begin{corollary}\label{C: ext dual}
Let $X$ be an $n$-dimensional unrestricted stratified pseudomanifold. Let $\mc E$ be a ts-coefficient system on $U_1$, let $\vec p$ be a ts-perversity, and let $u:U_1\into X$ be the inclusion.
Let $u_{!*}\mc E$ be the intermediate extension of $\mc E$ in  ${}^{\vec p^{\,+}_{\mc E}}D^{\heartsuit}(X)$ and let $u_{!*}\ms D\mc E$ be the intermediate extension of $\ms D\mc E$ in  ${}^{\D(\vec p^{\,+}_{\mc E})}D^{\heartsuit}(X)={}^{(D\vec p)^+_{\ms D\mc E}} D^{\heartsuit}(X)$. 
Then $\ms Du_{!*}\mc E\cong u_{!*}\ms D\mc E$ in ${}^{\D(\vec p^{\,+}_{\mc E})}D^{\heartsuit}(X)$. 
\end{corollary}

\begin{proof}
We recall that $\D(\vec p^{\,+}_{\mc E})=(D\vec p)^+_{\ms D\mc E}$ by Remark \ref{R: dual p}, so that we do in fact have ${}^{\D(\vec p^{\,+}_{\mc E})}D^{\heartsuit}(X)={}^{(D\vec p)^+_{\ms D\mc E}} D^{\heartsuit}(X)$. Furthermore, $\ms Du_{!*}\mc E$ is in this category by Theorem \ref{T: perverse duality}.

As $u$ is an open inclusion $u^*=u^!$, so using \cite[Theorem V.10.17]{Bo} we have $u^*\ms Du_{!*}\mc E\cong \ms Du^*u_{!*}\mc E\cong \ms D\mc E$. Thus $\ms Du_{!*}\mc E$ is an extension of $\ms D\mc E$. If $\ms Du_{!*}\mc E$ has a nontrivial subobject or quotient object supported on $X^{n-1}$, then it follows from Corollary \ref{C: perverse duality} that the same must be true of $\ms D(\ms Du_{!*}\mc E)\cong u_{!*}\mc E$ in ${}^{\vec p^{\,+}_{\mc E}}D^{\heartsuit}(X)$ since equivalences of categories are exact functors and $\ms D\ms D\cong \id$ implies that the dual of a nontrivial sub- or quotient object is also nontrivial. It is also easy to see from \eqref{E: dual point} that the dual of an object supported on $X^{n-1}$ is supported on $X^{n-1}$. But by Corollary \ref{C: int ext}, 
$u_{!*}\mc E$ possesses no nontrivial subobject or quotient object supported on $X^{n-1}$, so the same is true of 
$\ms Du_{!*}\mc E$. The corollary now follows from the uniqueness statement of Corollary \ref{C: int ext}. 
\end{proof}

\begin{remark}\label{R: other dual}
Together, Corollary \ref{C: ext dual} and Proposition \ref{P: deligne is ie} give an alternative proof of our ts-Deligne sheaf duality theorem (Theorem \ref{T: duality}). The argument here is perhaps more conceptual, though of course the trade-off is that the proof in this section requires some arguably much more sophisticated machinery. 
\end{remark}

\subsection{Chain conditions}\label{S: chain cond}

It is well-known that the classical category of perverse sheaves is both Noetherian and Artinian when working with field coefficients on a complex variety \cite[Theorem 4.3.1]{BBD}. In other words, for any such perverse sheaf $\mc S$, every ascending or descending chain of subobjects eventually stabilizes. Unfortunately, this is not generally true for ts-perverse sheaves, as the following illustrative example demonstrates.

\begin{example}\label{E: point noeth}
Let $X=pt$ with $R=\Z$. In this case $\mf X$-cc sheaf complexes are just chain complexes whose cohomology modules are finitely generated abelian groups, and $D(X)$ is the corresponding derived category of such chain complexes. Consider the exact sequence of abelian groups 
\begin{equation}\label{E: SES p}
0\to \Z\xr{p^k} \Z\to \Z_{p^k}\to 0
\end{equation}
 for a prime $p\in \Z$. Such an exact sequence yields a distinguished triangle 
\begin{equation*}
\to \Z\xr{p^k} \Z\to \Z_{p^k}\xr{[1]}
\end{equation*}
in $D^b_{\mf X}(X)$ \cite[page 51]{BaIH}, treating each group as a complex that is nontrivial only in degree $0$. Applying the cohomology functor ${}^\wp H^0=\ttau^\wp_{\geq 0}\ttau^\wp_{\leq 0}$ will result in an exact sequence in the heart ${}^\wp D^{\heartsuit}(X)$ \cite[Proposition 7.1.12]{BaIH}. For an arbitrary sheaf complex $\mc S$ on $X$ (in this case just a complex of abelian groups), we have from our definitions that ${}^\wp H^0(\mc S)= \ttau^\wp_{\leq 0}\mc S/\ttau^\wp_{\leq -1}\mc S$.  Using the notation of Section  \ref{S: TTT}, we have $W^{\wp}\Z=0$, while $W^{\wp}\Z_{p^k}=0$ if $p\notin \wp$ and $W^{\wp}\Z_{p^k}=\Z_{p^k}$ if $p\in \wp$, as in the latter case for each $z\in \Z_{p^k}$ we have that $p^kz=0$ is a boundary. 

So if we first suppose that $p\notin \wp$, then ${}^\wp H^0(\Z)\cong\ttau^\wp_{\leq 0}\Z/\ttau^\wp_{\leq -1}\Z\cong \Z/0=\Z$. Similarly ${}^\wp H^0(\Z_{p^k})=\Z_{p^k}$, and all other cohomology groups are $0$. So we obtain an exact sequence of the form \eqref{E: SES p}   in ${}^{\wp} D^{\heartsuit}(X)$. In particular, if there are any primes not in $\wp$ then ${}^{\wp} D^{\heartsuit}(X)$  will not be Artinian as we can form the descending chain determined by the inclusions $\cdots \to \Z\xr{p}\Z\xr{p}\Z$.

Similarly,  if $\wp\neq \emptyset$ the category ${}^{\wp} D^{\heartsuit}(X)$ is not Noetherian: Let $p\in \wp$, and consider the chain of inclusions $\cdots \to \Z\xr{p}\Z\xr{p}\Z$ from our previous example, now in ${}^{\D\wp} D^{\heartsuit}(X)$. 
Due to Corollary \ref{C: perverse duality}, if we apply $\ms D$ we obtain an infinite chain of quotients in ${}^{\vec p} D^{\heartsuit}(X)$. It is an interesting exercise to show that the corresponding chain of inclusions has the form $$\Z_p[-1]\into \Z_{p^2}[-1]\into \Z_{p^3}[-1]\into \cdots \into\Z.$$

On the other hand, when $\wp=\emptyset$, then ${}^{\emptyset} D^{\heartsuit}(X)$ is Noetherian: In this case the $t$-structure is the standard $t$-structure and  ${}^{\emptyset} D^{\heartsuit}(X)$ is equivalent to the category of finitely generated $R$-modules \cite[Example 7.1.5]{BaIH}. It is thus Noetherian, as PIDs are Noetherian. By Corollary \ref{C: perverse duality}, the dual category to ${}^{\emptyset} D^{\heartsuit}(X)$ is ${}^{\primeset{P}(R)} D^{\heartsuit}(X)$, which is therefore Artinian \cite[page 370]{Pop73}.
\end{example}

Although our example shows that ts-perverse sheaves are neither Noetherian nor Artinian in general, we do have the following, generalizing our example when $X$ is a point.

\begin{theorem}\label{T: noether}
Suppose $X$ is an $n$-dimensional unrestricted stratified pseudomanifold with a finite\footnote{The finite strata requirement is necessary. For example, suppose $X$ is an infinite number of disjoint points and $\vec p(x)=(0,\emptyset)$ for each point. Let $\Z_x$ be the unique sheaf complex on $X$ with stalk $\Z$ in degree $0$ at the point $x$ and otherwise trivial. Then $\oplus_{x\in X}\Z_x$ is $\mf X-cc$ and an object of ${}^{\vec p} D^{\heartsuit}(X)$, but it is not Noetherian as if we order the points we have $\Z_{x_1}\subset \Z_{x_1}\oplus \Z_{x_2}\subset \Z_{x_1}\oplus \Z_{x_2}\oplus \Z_{x_3}\subset \cdots \oplus_{x\in X}\subset \Z_x$.}
 number of strata, and let $\vec p$ be a ts-perversity. 
\begin{enumerate}
\item If $\vec p_2(Z)=\emptyset$ for each stratum $Z$ then ${}^{\vec p} D^{\heartsuit}(X)$ is Noetherian.

\item If $\vec p_2(Z)=\primeset{P}(R)$ for each stratum $Z$ then ${}^{\vec p} D^{\heartsuit}(X)$ is Artinian.
\end{enumerate}
\end{theorem}

\begin{proof}
It suffices to prove the first statement, as the second then follows from Corollary \ref{C: perverse duality} and that the dual of a Noetherian category is Artinian \cite[page 370]{Pop73}. We provide an argument that we think is a bit different from the published proofs we could find in the field coefficient case, which simultaneously demonstrate the Noetherian and Artinian properties using an induction on length and properties of simple objects (cf.\ \cite[Theorem 4.3.1]{BBD}, \cite[Corollary 3.5.7]{Kiehl}, \cite[Proposition 19.10]{Bhatt15}). Since our categories won't be both Noetherian and Artinian in general, we don't have such a notion of length available.

First consider the case where $X=U_1$ is a connected manifold stratified trivially, and let $\mc S\in {}^{\vec p} D^{\heartsuit}(X)$. Since $\vec p_2(Z)=\emptyset$ for all strata, $\mc S$ is quasi-isomorphic to a local system of finitely-generated $R$-modules in degree 0 (Remark \ref{R: natural heart}). Thus the data is an $R$-module with extra structure (a $\pi_1$ action). As finitely-generated $R$-modules are Noetherian \cite[Proposition 10.1.4]{LANG}, so is $\mc S$.
If $X$ is a trivially-stratified manifold with any finite number of strata, we can proceed by induction, as a sheaf complex on a disjoint union of components is the direct sum of sheaf complexes supported on each component and $\mc A\cong \mc B\oplus \mc C$ implies there is an exact sequence $0\to \mc B\to \mc A\to \mc C\to 0$. This suffices as we recall that in any Abelian category an object is Noetherian if and only if all of its subobjects and quotient (image) objects are Noetherian \cite[Proposition 5.7.2]{Pop73}.

We now proceed by induction on dimension. The case of dimension $0$ is covered by the manifold case. So now suppose we have proven the statement for stratified pseudomanifolds of dimension $<n$, let $\dim(X)=n$, and suppose $\mc S\in {}^{\vec p} D^{\heartsuit}(X)$. Let $\mc A_1\subset \mc A_2\subset \cdots \subset \mc S$ be a sequence of subobjects of $\mc S$. Let $i:U_1\into X$. The functor ${}^{\vec p}i^*$ is exact \cite[page 157]{BaIH}, so we obtain a chain ${}^{\vec p}i^*\mc A_1\subset{}^{\vec p}i^* \mc A_2\subset \cdots \subset {}^{\vec p}i^*\mc S^*$ in $ {}^{\vec p} D^{\heartsuit}(U_1)$. By the manifold case, this sequence must stabilize, so we can assume by relabeling that the inclusions $\mc A_k\into \mc A_{k+1}$ all induce isomorphisms  ${}^{\vec p}i^*\mc A_k\cong {}^{\vec p}i^* \mc A_{k+1}$.

Now, let $\mc B_k=\mc A_k/\mc A_1$. Then we have a diagram with exact rows

\begin{diagram}[LaTeXeqno]\label{D: cok}
0&\rTo& \mc A_1&\rTo& \mc A_k&\rTo&\mc B_k&\rTo& 0\\
&&\dTo^=&&\dInto&&\dTo\\
0&\rTo& \mc A_1&\rTo& \mc A_{k+1}&\rTo&\mc B_{k+1}&\rTo& 0
\end{diagram}
and similarly with $\mc S$ and $\mc S/\mc A_1$ in place of  $\mc A_{k+1}$ and $\mc B_{k+1}$.
By the Snake Lemma, we thus have a chain of inclusions $\mc B_1\subset \mc B_2\subset \cdots\subset \mc S/\mc A_1$. If this chain stabilizes then the chain of $\mc A_k$ also stabilizes by applying the Five Lemma to Diagram \eqref{D: cok}. (Note that diagrammatic theorems such as the Snake Lemma and Five Lemma hold in any abelian category by embedding a small abelian subcategory containing the objects of the diagram into the category of abelian groups. See \cite[Section IV.1]{Mit65}, \cite{Mur06}.)

Applying the exact functor ${}^{\vec p}i^*$ to the first row of Diagram \eqref{D: cok}, our assumption that ${}^{\vec p}i^*\mc A_k\cong {}^{\vec p}i^* \mc A_{k+1}$ for all $k$ shows that ${}^{\vec p}i^*\mc B_k=0$. 
But for any  $\mc C\in {}^{\vec p} D^{\heartsuit}(X)$,  we have the exact sequence \cite[Lemma 1.4.19]{BBD}
$$0\to {}^{\vec p} j_*{}^{\vec p} j^!\mc C\to \mc C\to {}^{\vec p} i_*{}^{\vec p} i^*\mc C\to {}^{\vec p}j_*H^1j^!\mc C\to 0,$$
with $j:X^{n-1}\into X$.
So for each $\mc B_k$ we have $\mc B_k\cong {}^{\vec p} j_*{}^{\vec p} j^!\mc B_k$. Furthermore, ${}^{\vec p} j^!$ is a left exact functor \cite[page 157]{BaIH}, so we have a chain ${}^{\vec p} j^!\mc B_1\subset{}^{\vec p}  j^!\mc B_2\subset\cdots\subset{}^{\vec p}  j^!(\mc S/\mc A_1)$ in ${}^{\vec p} D^{\heartsuit}(X^{n-1})$. By our induction hypothesis, the chain of ${}^{\vec p} j^!\mc B_k$ must stabilize. Thus the chain of ${}^{\vec p} j_*{}^{\vec p} j^!\mc B_k\cong \mc B_k$ must stabilize, as desired.
\end{proof}

\section{Torsion-tipped truncation and manifold duality}\label{S: man}

In this section, we provide an interesting example by computing $\H^*(X;\ms P)$, where $X$ is a PL pseudomanifold with  just one singular point $v$ and $R=\Z$. We relate these groups to the homology groups of the $\bd$-manifold obtained by removing a distinguished neighborhood of $v$. We then use manifold techniques to verify (abstractly) some of the isomorphisms guaranteed by Corollary \ref{T: PD}, though we will see that not all of these isomorphisms seem easily obtainable from the manifold perspective. It would be interesting to have a proof that the isomorphisms of the Corollary are always induced by geometric intersection and linking pairings  as is the case for classical intersection homology with field coefficients (see \cite{GM1,GBF18, GBF39, GBF30}). We leave this question for future research.

Let $X$ be a compact  $\Z$-oriented $n$-dimensional PL stratified pseudomanifold with stratification $X=X^n\supset X^0=\{v\}$, where $v$ is a single point. Then $X$ has the form $X\cong M\cup_{\bd M}\bar c(\bd M)$,
where $M^n$ is a compact  $\Z$-oriented PL manifold with boundary.
 Let $U=X-v$. Let $\mc O$ be the constant orientation sheaf with $\Z$ coefficients on $U$, let $i:U\into X$ be the inclusion, and let $k\in \Z$. Let $\ms P=\ms P_{X,\vec p,\mc O}$ be the ts-Deligne sheaf with $\vec p_1(\{v\})=k$ and $\vec p_2(\{v\})=\wp$ for some $\wp\in\P(\primeset{P}(\Z))$. If $k<-1$, then $\ms P$ is the extension by $0$ of (an injective resolution of) $\mc O$. If $k\geq -1$, then $\ms P=\ttau^{\wp}_{\leq k}Ri_*\mc O$. If $\wp=\emptyset$, then $\ms P$ would be the classical Deligne sheaf  for the perversity $\bar p$ with $\bar p(\{v\})=k$ by Example \ref{E: GM}, and its hypercohomology would be the classical perversity $\bar p$ intersection homology. 
 
 To simplify notation, we let $\H^{i}(X;\ms P_{X,\vec p,\mc O})$ be denoted by $I^{\vec p}H_{n-i}(X)$, generalizing the standard intersection homology notation. 
We also let $\vec q=D\vec p$, so that $\vec q_1(\{v\})=n-k-2$ and $\vec q_2(\{v\})$ is the complement of $\wp$ in $\primeset{P}(\Z)$. Further, note that $\mc D\mc O[-n]\cong \mc O$, so $\mc D \ms P_{X,\vec p,\mc O}[-n]\cong \ms P_{X,\vec q,\mc O}$, with hypercohomology groups $I^{\vec q}H_{n-i}(X)$.

\paragraph{Group computations.} We begin by computing $I^{\vec p}H_{n-i}(X)$ as best as possible in terms of the homology groups of $M$. For comparison, it is worth recalling that if $\bar p$ is a perversity on $X$ with $\bar p(\{v\})=k$ then 
 the standard computation involving the cone formula and the Mayer-Vietoris sequence gives\footnote{See \cite[Example 6.3.15]{GBF35}. If $k>n-2$, this computation assumes that we use non-GM intersection homology in the sense of \cite[Chapter 6]{GBF35}; see also \cite{GBF23,GBF26}. }
\begin{equation*}
I^{\bar p}H_{n-i}(X)\cong
\begin{cases}
H_{n-i}(M), & i>k+1,\\
\im(H_{n-i}(M)\to H_{n-i}(M,\bd M)),& i=k+1,\\ 
H_{n-i}(M,\bd M), & i< k+1.
\end{cases}
\end{equation*}
As noted above, this will then also be the computation for $I^{\vec p}H_{n-i}(X)$ when $\vec p_2(\{v\})=\emptyset$. 

As $X$ is compact by assumption,  $I^{\vec p}H_{n-i}(X)=\H^i(X;\ms P)=\H^i_c(X;\ms P)$. Therefore, to study $I^{\vec p}H_{n-i}(X)$, we can use that the adjunction triangle yields a long exact sequence \cite[Remark 2.4.5.ii]{DI04}
\begin{equation*}
\to \H^i_c(U;\ms P)\to \H^i_c(X;\ms P)\to \H^i_c(v;\ms P)\to.
\end{equation*}
We know  the restriction of  $\ms P$ to $U$ is quasi-isomorphic to $\mc O$, so $$\H^i_c(U;\ms P)\cong \H^i_c(U;\mc O)\cong H_{n-i}^c(U)\cong H_{n-i}(M).$$ Furthermore, 
applying Lemma \ref{L: sheaf ttau},
\begin{align*}
\H^i_c(v;\ms P)&\cong \H^i(v;\ms P)\cong 
\begin{cases}
0, & i>k+1,\\
T^{\wp}H^{k+1}((Ri_*\mc O)_v), & i=k+1,\\
H^i((Ri_*\mc O)_v),& i\leq k.
\end{cases}
\end{align*}

But,letting $H^\infty_*$ denotes Borel-Moore homology,

$$H^i((Ri_*\mc O)_v)\cong \dlim_{v\in U} \H^i(U; Ri_*\mc O)
\cong \dlim_{v\in U} \H^i(U-v; \mc O)
\cong \dlim_{v\in U} H_{n-i}^{\infty}(U-v; \mc O).$$ 
So restricting to a cofinal sequence of conical neighborhoods, this becomes simply 
$H_{n-i}^{\infty}(\bd M\times (0,1); \mc O)\cong H_{n-i-1}(\bd M)$. Similarly, $T^{\wp}H^{k+1}((Ri_*\mc O)_v)\cong T^{\wp}H_{n-k-2}(\bd M)$. 

So our exact sequences look like 
\begin{equation*}
\to H_{n-i}(M)\to I^{\vec p}H_{n-i}(X)\to H_{n-i-1}(\bd M)\to
\end{equation*}
for $i\leq k$, like
\begin{equation*}
\to H_{n-i}(M)\to I^{\vec p}H_{n-i}(X)\to 0\to
\end{equation*}
for $i>k+1$, and at the transition, we have
\begin{multline} \label{E: seq}
\to H_{n-k-1}(\bd M) \to H_{n-k-1}(M)\to I^{\vec p}H_{n-k-1}(X)  \\
\to T^{\wp}H_{n-k-2}(\bd M)\to H_{n-k-2}(M)\to I^{\vec p}H_{n-k-2}(X)\to 0.
\end{multline}

It is therefore immediate that $I^{\vec p}H_j(X)\cong H_j(M)$ for $j\leq n-k-3$.
Furthermore, the canonical morphism $\ms P=\ttau^\wp_{\leq k}Ri_*\mc O\to Ri_*\mc O$  induces a map between the corresponding long exact adjunction sequences. The  sequence for  $Ri_*\mc O$ is simply the sheaf-theoretic long exact (compactly supported) cohomology sequence of the pair $(M,\bd M)$, and so it follows from the five lemma that  $I^{\vec p}H_j(X)\cong H_j(M,\bd M)$ for $j\geq n-k$.  It also follows from this that all maps in the sequence for $\ms P$ are the evident ones.
For $i=n-k-2, n-k-1$, we see that $I^{\vec p}H_{n-k-2}(X)\cong \cok(T^{\wp}H_{n-k-2}(\bd M)\to H_{n-k-2}(M))$. The module  $I^{\vec p}H_{n-k-1}(X)$ is a bit more complicated, but we can nonetheless compute it using the following lemma.

\begin{lemma}
Let $\bd_*: H_{n-k-1}(M,\bd M)\to H_{n-k-2}(\bd M)$ be the boundary map of the exact sequence, and let $\q^{\wp}$ be the quotient $\q^{\wp}: H_{n-k-2}(\bd M)\to H_{n-k-2}(\bd M)/T^{\wp}H_{n-k-2}(\bd M)$.
Then $I^{\vec p}H_{n-k-1}(X)\cong \ker(\q^{\wp}\bd_*)$.  
\end{lemma}
\begin{proof}
Consider the following diagram of exact sequences induced by the inclusion $\ms P\into Ri_*\mc O$, letting  $h$ denote the inclusion of the $\wp$-torsion subgroup of $H_{n-k-2}(\bd M)$:
\begin{diagram}[LaTeXeqno]\label{D: compare}
H_{n-k-1}(\bd M)&\rTo& H_{n-k-1}(M)&\rTo^{\td j} & I^{\vec p}H_{n-k-1}(X)&\rTo^f & T^{\wp}H_{n-k-2}(\bd M)&\rTo & H_{n-k-2}(M)\\
\dTo^=&& \dTo^=&& \dTo^g&& \dTo^h&& \dTo^=&&\\
H_{n-k-1}(\bd M)&\rTo^i &H_{n-k-1}(M)&\rTo^j & H_{n-k-1}(M,\bd M)&\rTo^{\bd_*} & H_{n-k-2}(\bd M)&\rTo & H_{n-k-2}(M)
\end{diagram}

From the diagram, if $x\in I^{\vec p}H_{n-k-1}(X)$, then $g(x)\in H_{n-k-1}(M,\bd M)$ maps to a $\wp$-torsion element in $H_{n-k-2}(\bd M)$ under the boundary map. Thus $I^{\vec p}H_{n-k-1}(X)$ must map into $\ker(\q^{\wp}\bd_*)$. 

We now proceed with diagram chases akin to those in the proof of the five lemma.

To see that $g$ maps onto $\ker(\q^{\wp}\bd_*)$, suppose  $u\in \ker(\q^{\wp}\bd_*)$. Then $\bd_*(u)\in T^{\wp}H_{n-k-2}(\bd M)$, so $\bd_*(u)$ is in the image of $h$. Since the image of $\bd_*(u)$ in $H_{n-k-2}(M)$ must be $0$ (from the long exact sequence on the bottom), it follows that $\bd_*(u)$, as an element of $T^{\wp}H_{n-k-2}(\bd M)$ must be in the image of $f$. Let $x\in I^{\vec p}H_{n-k-1}(X)$ be such that $hf(x)=\bd_* u\in H_{n-k-2}(\bd M)$. Then $\bd_*(g(x))=\bd_*(u)$ from the diagram, i.e. $\bd_*(g(x)-u)=0$, so there is a $z\in H_{n-k-1}(M)$ such that $j(z)=g(x)-u$. But $j(z)=g\td j(z)$, so $g\td j(z)=g(x)-u$, whence $u=g(x)-g\td j(z)=g(x-\td j(z))$. Therefore $u$ is in the image of $g$ and so $g$ maps onto $\ker(\q^{\wp}\bd_*)$.

 For injectivity,  suppose $x\in  I^{\vec p}H_{n-k-1}(X)$ and $g(x)=0$. Then $\bd_*g(x)=hf(x)=0$, but $h$ is injective, so $f(x)=0$ and $x=\td j(y)$ for some $y\in H_{n-k-1}(M)$. This implies that $j(y)=g\td j(y)=g(x)=0$, so $y=i(z)$ for some $z\in H_{n-k-1}(\bd M)$. But then $x=\td j(y)=\td ji(z)=0$, from the short exact sequence.
\end{proof}

So, altogether, we see that if $\vec p_1(\{v\})=k$ then

\begin{equation}\label{E: point sing}
I^{\vec p}H_i(X)\cong
\begin{cases}
H_i(M,\bd M),&i\geq n-k,\\
\ker(H_{i}(M,\bd M)\xr{\q^{\wp}\bd_*} H_{i-1}(\bd M)/T^{\wp}H_{i-1}(\bd M)),&i=n-k-1,\\
\cok(T^{\wp}H_{i}(\bd M)\to H_{i}(M)),&i=n-k-2,\\
H_i(M), &i\leq n-k-3.\\
\end{cases}
\end{equation}
For reference below, if we replace $\vec p$ with its dual $\vec q$ we see that similarly
\begin{equation*}
I^{\vec q}H_j(X)\cong
\begin{cases}
H_j(M,\bd M),&j\geq k+2,\\
\ker(H_{k+1}(M,\bd M)\xr{\q^{D\wp}\bd_*} H_{k+1}(\bd M)/T^{D\wp}H_{k+1}(\bd M)),&j=k+1,\\
\cok(T^{D\wp}H_{j}(\bd M)\to H_{j}(M)),&j=k,\\
H_j(M), &j\leq k-1\\
\end{cases}
\end{equation*}
In particular,  $I^{\vec p}H_i(X)\cong I^{\vec p_1}H_i(X)$ for $i\neq n-k-2, n-k-1$ and  $I^{\vec q}H_i(X)\cong I^{\vec q_1}H_i(X)$ for $i\neq k, k+1$.

\paragraph{Duality isomorphisms.} 
Corollary \ref{T: PD} implies that for all $i$ there must be isomorphisms
\begin{equation}\label{E: point duality}
FI^{\vec p}H_i(X)\cong \Hom(FI^{\vec q}H_{n-i}(X),\Z) \qquad TI^{\vec p}H_i(X)\cong \Hom(TI^{\vec q}H_{n-i-1}(X),\Q/\Z),
\end{equation}
where given an abelian group $G$ we again let  $TG$ denote the torsion subgroup of $G$ and  $FG=G/TG$.
We would like to see how these isomorphisms \eqref{E: point duality} relate to known isomorphisms from Lefschetz duality. For such a simple pseudomanifold, many of the isomorphisms of \eqref{E: point duality} correspond to the known duality isomorphisms of ordinary intersection homology, which themselves can be described in terms of the intersection and torsion linking pairings on the manifold $M$. However, even for classical intersection homology the direct relation between the sheaf-theoretic and PL intersection pairings turns out to be a difficult result; see \cite{GBF30}. So we will not attempt to prove here that all the pairings \eqref{E: point duality} can be obtained by PL intersection and linking, though the author hopes to demonstrate this in the future. 

Rather, what we will look at here is the extent to which the isomorphisms of \eqref{E: point duality} can be deduced \emph{abstractly} from Lefschetz duality, meaning that we will look at when Lefschetz duality provides \emph{some} isomorphisms as in \eqref{E: point duality} but without showing that it provides the \emph{same} isomorphisms. When $i=n-k-1,n-k-2$, it is not so obvious that these isomorphisms come from classical manifold duality.  For the simple cases where $\wp$ is $\emptyset$ or $P(\Z)$ we will be able to verify these abstract isomorphisms using intersection and linking pairings now that we know to look for them, though we will see that even this requires some effort. We will not provide such a verification for more general $\wp$, as we will see that these isomorphisms are much less clear from the pure manifold perspective. Rather, we consider the isomorphisms obtained by combining \eqref{E: point duality} and \eqref{E: point sing} as an  application of our ts-Deligne-sheaf machinery to detect facts about the homology of manifolds not easily obtained by direct means. 

We begin with the following easy observations:

\begin{enumerate}
\item We have seen that $I^{\vec p}H_i(X)\cong H_i(M)$ for $i\leq n-k-3$, while $I^{\vec q}H_{i}(X)\cong
H_i(M,\bd M)$ for $j\geq  k+2$. So for  $i\leq n-k-3$, there exist isomorphisms of the form \eqref{E: point duality} by classical Lefschetz duality. 

\item Similarly, we have $I^{\vec p}H_i(X)\cong H_i(M,\bd M)$ for $i\geq n-k$, while $I^{\vec q}H_i(X)\cong H_i(M)$ for $i\leq  k-1$. So, again, there exist isomorphisms of the form \eqref{E: point duality} by classical Lefschetz duality when $i\geq n-k+1$ and also for the classical Lefschetz torsion pairing when $i=n-k$. 

\item When $i=n-k$, the torsion-free part of \eqref{E: point duality} also follows abstractly  from Lefschetz duality, since 
$FI^{\vec q}H_i(X)\cong F(\cok(T^{D\wp}H_{i}(\bd M)\to H_{i}(M)))
\cong FH_i(M).$

\item We have  seen that $I^{\vec p}H_{n-k-2}(X)\cong \cok(T^{\wp} H_{n-k-2}(\bd M)\to H_{n-k-2}(M))$, and so $FI^{\vec p}H_{n-k-2}(X)\cong FH_{n-k-2}(M)$. Once again, $I^{\vec q}H_{k+2}(X)\cong H_{k+2}(M,\bd M)$, so there is an isomorphism as in \eqref{E: point duality} by Lefschetz duality.

\end{enumerate}

By contrast, the remaining isomorphisms 
\begin{align*}
FI^{\vec p}H_{n-k-1}(X)&\cong \Hom(FI^{\vec q}H_{k+1}(X),\Z)\\
 TI^{\vec p}H_{n-k-1}(X)&\cong \Hom(TI^{\vec q}H_{k}(X),\Q/\Z)\\
 TI^{\vec p}H_{n-k-2}(X)&\cong \Hom(TI^{\vec q}H_{k+1}(X),\Q/\Z)
\end{align*} 
are not evident from the classical manifold point of view, though, by  Corollary \ref{T: PD}, such isomorphisms must exist.
We will provide  explicit such isomorphisms via intersection and linking  forms on $M$ in the special case where $\wp=\primeset{P}(\Z)$, the set of all primes, and $D\wp=\emptyset$. The same arguments would also handle the case with $\wp$ and $D\wp$ reversed. The more general situation seems to be quite a bit more delicate, and we will not take it up here.

\begin{lemma}
If $M$ is a compact PL manifold with non-empty boundary and  $\vec p_2(\{v\})=\primeset{P}(\Z)$, the intersection pairing on $M$ induces a nonsingular pairing between $FI^{\vec p}H_{n-k-1}(X)\subset FH_{n-k-1}(M,\bd M)$ and $FI^{\vec q}H_{k+1}(X)\subset FH_{k+1}(M,\bd M)$.
\end{lemma}\begin{proof} 
As indicated in the statement of the lemma, we identify $FI^{\vec p}H_{n-k-1}(X)$ with
$F\ker(\q^{\wp}\bd_*)\subset H_{n-k-1}(M,\bd M)$ and $FI^{\vec q}H_{k+1}(X)$ with $F\ker(\q^{D\wp}\bd_*)\subset H_{k+1}(M,\bd M)$.
As  $D\wp=\emptyset$, the latter group is really just $F\ker(\bd_*)\subset H_{k+1}(M,\bd M)$.

We define $\phi: FI^{\vec q}H_{k+1}(X)\to \Hom(FI^{\vec p}H_{n-k-1}(X),\Z)$ via intersection pairings.
 Suppose   $\xi\in FI^{\vec q}H_{k+1}(X)$. Then $\bd_*\xi=0$, and $\xi=j(x)$ for some $x\in H_{k+1}(M)$ by the long exact sequence in \eqref{D: compare}. Define the homomorphism $\phi(\xi)$ so that if $y\in FI^{\vec p}H_{n-k-1}(X)$ then $(\phi(\xi))(y)=x\pf y$, where $\pf$ denotes the Lefschetz duality intersection pairing on $M$.   We first check this is well-defined. 

The intersection pairing is trivial on torsion elements, so $\phi$ is well defined on the torsion free quotients. Next, we show that $\phi$ is independent of the choice of $x$. For this, suppose  $z\in \ker (H_{k+1}(M)\to H_{k+1}(M,\bd M))$. We will show that 
 $z\pf y=0$. So if $x'$ is another preimage of $\xi$ in $H_{k+1}(M)$, then $x-x'\in \ker (H_{k+1}(M)\to H_{k+1}(M,\bd M))$, so $(x-x')\pf y=0$ and $x\pf y=x'\pf y$. 
It will follow that $\phi$ is independent of the choice of $x$.  
So let $z\in \ker (H_{k+1}(M)\to H_{k+1}(M,\bd M))$. Then $z$  is represented by a chain in $\bd M$.  Now if $y\in \ker (\q^{\wp}\bd_*)$, then for some $m\in S(\wp)$,  we have $m\bd_* y=0\in H_{n-k-2}(\bd M)$, and this implies $m\bd_* y$, which is represented by $m\bd y$, is itself a boundary in $\bd M$, say\footnote{We will have occasion to abuse notation by sometimes letting the same symbol refer to both a chain and the homology class it represents.} $m\bd y=\bd Y$ for some $Y\in C_{n-k-1}(\bd M)$. So $my-Y$ is a cycle in $M$ that also represents $my$ in $H_{k+1}(M,\bd M)$. But then $my-Y$ is homologous to a cycle $u$ in the interior of $M$ by pushing in along a collar of the boundary. In particular, $u$ and $z$ can be represented by disjoint cycles in $M$. So, in $M$, the intersection number of $z$ and $u$ is $0$. But the intersection number between $z$ and $u$ represents $z\pf my$ as $my=u \in H_{n-k-1}(M,\bd M)$. So $z\pf my=m(z\pf y)=0$, and $z\pf y$ must be $0$.  Thus $\phi$ is independent of the choice of $x$.

We also observe that $\phi(x)(y_1+y_2)=\phi(x)(y_1)+\phi(x)(y_2)$ by the basic properties of intersection products. To show that $\phi$ is a homomorphism, we note that if  $\xi_1,\xi_2\in FI^{\vec q}H_{k+1}(X)$ and $j(x_1)=\xi_1$,  $j(x_2)=\xi_2$, then $j(x_1+x_2)=\xi_1+\xi_2$, and so 
\begin{equation*}
\phi(\xi_1+\xi_2)(y)= (\xi_1+\xi_2)\pf y
=\xi_1\pf y+\xi_2\pf y
=\phi(\xi_1)(y)+\phi(\xi_2)(y).
\end{equation*}
Altogether, we have now shown that $\phi$ is a well-defined homomorphism. 

Next we show that $\phi$ is injective. Recall that, by Lefschetz duality,  $FH_{k+1}(M)\cong \Hom(FH_{n-k-1}(M,\bd M),\Z)$ and 
$FH_{k+1}(M,\bd M)\cong \Hom(FH_{n-k-1}(M),\Z)$ via the intersection pairing. 
Let $\xi\in FI^{\vec q}H_{k+1}(X)\cong 
F\ker(\bd_*)$ with $\xi\neq 0$. We will show that $\phi(\xi)\neq 0$, which implies injectivity. The class $\xi$ is represented by a cycle  $x$ in $M$, which also represents an element of $FH_{k+1}(M)$. As $0\neq \xi \in FH_{k+1}(M,\bd M)$, by Lefschetz duality, there must be a $y\in FH_{n-k-1}(M)$ such that $x\pf y\neq 0$. Furthermore, the intersection number continues to be the same if we think of a chain representing $y$ as instead representing an element of $FH_{n-k-1}(M,\bd M)$, while $x$ can be represented by  an element of $H_{k+1}(M)$. Therefore, the class of the chain representing $y$ must be non-zero in $FH_{n-k-1}(M,\bd M)$, and, since it's in the image of $FH_{n-k-1}(M)$, it must be in $\ker(\bd_*)$ and hence in $F\ker(\q^{\wp}\bd_*)$. Therefore, given a non-zero $\xi\in FI^{\vec q}H_{k+1}(X)$, with $x$ a preimage of $\xi$ in $H_{k+1}(M)$, we have found a $y\in  FI^{\vec p}H_{n-k-1}(X)$ such that $x\pf y\neq 0$. It follows that $\phi(\xi)\neq 0$, and thus $\phi$ is injective.

For surjectivity, $\q^{\wp}\bd_*$ has free image (as $\wp=\primeset{P}(\Z)$), so  $\ker (\q^{\wp}\bd_*)=FI^{\vec p}H_{n-k-1}(X)$ is a direct summand of $FH_{n-k-1}(M,\bd M)$. Let $y$ be a generator of $\ker (\q^{\wp}\bd_*)$, and let $\{y'_j\}$ be a collection of elements of $FH_{n-k-1}(M,\bd M)$ that together with $y$ form a basis.  Let $\{y''_\ell\}$ be a collection of elements of $\ker (\q^{\wp}\bd_*)$ that together with $y$ form a basis. As $\ker (\q^{\wp}\bd_*)\subset FH_{n-k-1}(M,\bd M)$, ever $y''_\ell$ must be a linear combination of the $\{y'_j\}$. 
Now, let $x\in FH_{k+1}(M)$ be the Lefschetz dual of $y$ in the pairing between $FH_{n-k-1}(M,\bd M)$ and $FH_{k+1}(M)$. In other words, let $x$ be the unique element with $x\pf y=1$, while $x\pf y_j'=0$ for each of the $y_j'$. Let $\xi$ be the image of $x$ in $FH_{k+1}(M, \bd M)$; then $\xi\in F\ker(\bd_*)=FI^{\vec q}H_{k+1}(X)$. We must have $\phi(\xi)(y)=1$, while all $\phi(\xi)(y''_\ell)=0$.  So $\xi$ is a dual to $y$ in the pairing between  $FI^{\vec p}H_{n-k-1}(X)$ and $FI^{\vec q}H_{k+1}(X)$. 
Since $y$ was an arbitrary generator of $F\ker(\q^{\wp}\bd_*)$, we can construct a dual basis in $FI^{\vec q}H_{k+1}(X)$  to our basis of $FI^{\vec p}H_{n-k-1}(X)$, and  so $\phi$ is surjective.
\end{proof}

\begin{lemma}
If $M$ is a compact PL manifold with non-empty boundary and  $\vec p_2(\{v\})=\primeset{P}(\Z)$, the linking pairing on $M$ induces a nonsingular pairing 
between $TI^{\vec p}H_{n-k-1}(X)$ and $TI^{\vec q}H_{k}(X)$ and a nonsingular pairing between $TI^{\vec p}H_{n-k-2}(X)$ and $TI^{\vec q}H_{k+1}(X)$. 
\end{lemma}
\begin{proof}
Given that $\vec p_2(\{v\})=\primeset{P}(\Z)$, the  pairing involving $TI^{\vec p}H_{n-k-1}(X)$ actually reduces to the standard Lefschetz torsion linking pairing. To see this,
 we first have from our computations that $TI^{\vec p}H_{n-k-1}(X)\cong T\ker(H_{n-k-1}(M,\bd M)\xr{\q^{\wp}\bd_*} H_{n-k-2}(\bd M)/T^{\wp}H_{n-k-2}(\bd M))$. But this is precisely
 $TH_{n-k-1}(M,\bd M)$, itself, as any torsion element of $H_{n-k-1}(M,\bd M)$ that is not in $\ker\bd_*$ has its image in $TH_{n-k-2}(M)$, and so dies under\footnote{Here is one place we use our assumption $\vec p_2(\{v\})=P(\Z)$. We can also see here one reason that a general choice of $\wp$ would make things much more complicated, as in this case $TI^{\vec p}H_{n-k-1}(X)$ would have to contain all of the $\wp$-torsion of $TH_{n-k-1}(M,\bd M)$ but perhaps also some $D\wp$-torsion elements that happen to be in $\ker \bd_*$ though this need not be all the $D\wp$-torsion of $TH_{n-k-1}(M,\bd M)$, nor even a direct summand of the $D\wp$-torsion subgroup.} $\mf q^{\wp}=\mf q^{\primeset{P}(\Z)}$. 
On the other hand,  $I^{\vec q}H_k(M)\cong \cok(T^{D\wp}H_{k}(\bd M)\to H_{k}(M))=H_{k}(M)$, as $D\wp$ is empty and thus $T^{D\wp}H_{k}(\bd M)=0$. So the isomorphism $TH_{n-k-1}(M,\bd M)\cong \Hom(TH_k(M),\Q/\Z)$ of the classical linking pairing becomes $TI^{\vec p}H_{n-k-1}(X) \cong \Hom(TI^{\vec q}H_k(M),\Q/\Z).$ 

Now, we consider $TI^{\vec p}H_{n-k-2}(X)$. By \eqref{E: point sing}, we have $I^{\vec p}H_{n-k-2}(X)\cong \cok(TH_{n-k-2}(\bd M)\to H_{n-k-2}(M)).$
So if we let $U=\im(TH_{n-k-2}(\bd M)\to TH_{n-k-2}(M))$; then  $TI^{\vec p}H_{n-k-2}(X)\cong TH_{n-k-2}(M)/U$. 
Meanwhile
\begin{align*}
I^{\vec q}H_{k+1}(X)&\cong \ker(H_{k+1}(M,\bd M)\xr{\q^{D\wp}\bd_*} H_{k+1}(\bd M)/T^{D\wp}H_{k+1}(\bd M))\\
&\cong \ker(H_{k+1}(M,\bd M)\xr{\bd_*} H_{k+1}(\bd M))\cong \im(H_{k+1}(M)\to H_{k+1}(M,\bd M)),
\end{align*}
since $D\wp=\emptyset$. For brevity, let  $W=\im(H_{k+1}(M)\to H_{k+1}(M,\bd M))\cong I^{\vec q}H_{k+1}(X)$, and let $\odot: TH_{n-k-2}(M)\otimes TH_{k+1}(M,\bd M)\to \Q/\Z$ denote the linking pairing operation\footnote{Recall that the linking number can be described geometrically as follows: if $x, y$ are cycles in general position with $mx=\bd z$ and $ny=\bd u$, $m,n\neq 0$, then $x\odot y=\frac{z\pf y}{m}=\frac{x\pf u}{n}\in \Q/\Z$, where now $\pf$ denotes the intersection number on chains in general position. A derivation of this formula in the dual cohomological setting can be found in \cite[Section 8.4.3]{GBF35}.}. Define  $f:TI^{\vec p}H_{n-k-2}(X)\to \Hom(TW,\Q/\Z)$ by $f(x)(y)=x\odot y$. We must first show that this is well defined by showing that $x\odot y=0$ if $x\in U$. But in this case $x$ is represented by a cycle in $\bd M$ and if $mx=0\in TH_{n-k-2}(\bd M)$, $m\neq 0$, then $mx=\bd z$ for some chain $z$ in $\bd M$. By definition, $y$ is represented by a cycle in $M$, which we can assume is supported in the interior of $M$. Thus $z\pf y=0$, so $x\odot y=0$. 

Consider the inclusion $TW\into  TH_{k+1}(M,\bd M)$. By classical manifold linking duality, the linking pairing induces an isomorphism $TH_{n-k-2}(M)\to \Hom(TH_{k+1}(M,\bd M),\Q/\Z)$. Since $TW$ is a subgroup of $TH_{k+1}(M,\bd M)$ and $\Q/\Z$ is an injective group, we have a surjection $\Hom(TH_{k+1}(M,\bd M),\Q/\Z)\to \Hom(TW,\Q/\Z)$ induced by restriction. The composition $g: TH_{n-k-2}(M)\to \Hom(TW,\Q/\Z)$ induces $f$, which we therefore see is onto.

Next, since we already know $U\subset \ker g$, to show that $f$ is injective, it now suffices to show $\ker g\subset U$.  By counting,
\begin{equation*}
|TH_{n-k-2}(M)|=|\ker g| \cdot |\im g|=|\ker g|\cdot |\Hom(TW,\Q/\Z)|=|\ker g| \cdot |TW|.
\end{equation*}

Consider the linking duality isomorphism $TH_{k+1}(M,\bd M)\to \Hom(TH_{n-k-2}(M),\Q/\Z)$. Since $U\subset TH_{n-k-2}(M)$ and $\Q/\Z$ is an injective, the restriction map $\Hom(TH_{n-k-2}(M),\Q/\Z)\to \Hom(U,\Q/\Z)$ is surjective, and thus we have a composite surjection $h: TH_{k+1}(M,\bd M)\to \Hom(U,\Q/\Z)$. So $$|TH_{k+1}(M,\bd M)|=|\ker h|\cdot|\im h|=|\ker h|\cdot|\Hom(U,\Q/\Z)|=|\ker h|\cdot|U|.$$

We have already seen that $U$ and $TW$ are orthogonal under the linking pairing, thus $h$ induces a surjective  homomorphism $TH_{k+1}(M,\bd M)/TW\onto \Hom(U,\Q/\Z)$. In particular, $TW\subset \ker h$. 
We will see that also $\ker(h)\subset TW$, so $\ker (h)=TW$. Therefore, 
\begin{multline*}
|\ker(g)|=|TH_{n-k-2}(M)|\div|TW|
=|TH_{k+1}(M,\bd M)|\div|TW|\\
=|\ker (h)|\cdot|U|\div|TW|
=|TW|\cdot|U|\div|TW|
=|U|,
\end{multline*}
which implies $\ker g=U$.

To prove the claim that $\ker h\subset TW$, 
suppose $x\in TH_{k+1}(M,\bd M)$ and $x\notin W$. Then $\bd_* x\neq 0\in TH_k(\bd M)$. However, since $x$ is a torsion element, there exists a $z\in C_{k+2}(M)$ such that $\bd z=mx+z'$, where $m\neq 0$ and $z'$ is a chain in $\bd M$. Then $m\bd x=-\bd z'\in C_k(\bd M)$. Now since $TH_k(\bd M)\cong \Hom(TH_{n-k-2}(\bd M),\Q/\Z)$ by the linking pairing $\odot_{\bd M}$ in $\bd M$, there is a $y\in TH_{n-k-2}(\bd M)$ such that $\bd x\odot_{\bd M}y=\frac{-1}{m}z'\pf_{\bd M} y\neq 0$ (see e.g. \cite[Appendix]{GBF27}). But  $z'\pf_{\bd M} y=\pm z\pf_M y$, where the subscript indicates the space in which we are computing the intersection number, after moving chains into general position (which does not alter homology classes). Therefore $\bd x\odot_{\bd M}y= \pm\frac{1}{m} z\pf_M y$.  But now thinking of $y$ as representing an element of $U$ and of $z$ as a chain rel $\bd M$, in which case $\bd z=mx$, we have 
$\frac{1}{m} z\pf_M y =x\odot_M y$. As this linking number is not $0$, we have shown that if $x\notin TW$, then $h(x)\neq 0$. Thus $\ker h\subset TW$.  
\end{proof}

\bibliographystyle{amsplain}

\begin{thebibliography}{10}

\bibitem{BaIH}
M.~Banagl, \emph{Topological invariants of stratified spaces}, Springer
  Monographs in Mathematics, Springer, Berlin, 2007.

\bibitem{BBD}
A.~A. Be{\u\i}linson, J.~Bernstein, and P.~Deligne, \emph{Faisceaux pervers},
  Analysis and topology on singular spaces, {I} ({L}uminy, 1981), Ast\'erisque,
  vol. 100, Soc. Math. France, Paris, 1982, pp.~5--171.

\bibitem{Bhatt15}
Bhargav Bhatt, \emph{Perverse sheaves}, Fall 2015 course notes;
  \url{http://www-personal.umich.edu/~takumim/MATH731.pdf}.

\bibitem{Bo}
A.~Borel et~al., \emph{Intersection cohomology}, Progress in Mathematics,
  vol.~50, Birkh\"auser Boston, Inc., Boston, MA, 1984, Notes on the seminar
  held at the University of Bern, Bern, 1983, Swiss Seminars.

\bibitem{Br}
Glen~E. Bredon, \emph{Sheaf theory}, second ed., Graduate Texts in Mathematics,
  vol. 170, Springer-Verlag, New York, 1997.

\bibitem{CS91}
Sylvain~E. Cappell and Julius~L. Shaneson, \emph{Singular spaces,
  characteristic classes, and intersection homology}, Ann. of Math. (2)
  \textbf{134} (1991), no.~2, 325--374.

\bibitem{DI04}
Alexandru Dimca, \emph{Sheaves in topology}, Universitext, Springer-Verlag,
  Berlin, 2004.

\bibitem{GBF35}
Greg Friedman, \emph{Singular intersection homology}, book in preparation;
  current draft available at \url{http://faculty.tcu.edu/gfriedman/IHbook.pdf}.

\bibitem{GBF18}
\bysame, \emph{On the chain-level intersection pairing for {PL}
  pseudomanifolds}, Homology Homotopy Appl. \textbf{11} (2009), no.~1,
  261--314.

\bibitem{GBF23}
\bysame, \emph{Intersection homology with general perversities}, Geom. Dedicata
  \textbf{148} (2010), 103--135.

\bibitem{GBF26}
\bysame, \emph{An introduction to intersection homology with general perversity
  functions}, Topology of stratified spaces, Math. Sci. Res. Inst. Publ.,
  vol.~58, Cambridge Univ. Press, Cambridge, 2011, pp.~177--222.

\bibitem{GBF39}
\bysame, \emph{The chain-level intersection pairing for {PL} pseudomanifolds
  revisited}, J. Singul. \textbf{17} (2018), 330--367.

\bibitem{GBF27}
Greg Friedman and Eug\'enie Hunsicker, \emph{Additivity and non-additivity for
  perverse signatures}, J. Reine Angew. Math. \textbf{676} (2013), 51--95.

\bibitem{GBF30}
Greg Friedman and James~E. McClure, \emph{Intersection homology duality and
  pairings: singular, {PL}, and sheaf-theoretic}, preprint.

\bibitem{GM1}
Mark Goresky and Robert MacPherson, \emph{Intersection homology theory},
  Topology \textbf{19} (1980), no.~2, 135--162.

\bibitem{GM2}
\bysame, \emph{Intersection homology. {II}}, Invent. Math. \textbf{72} (1983),
  no.~1, 77--129.

\bibitem{GS83}
Mark Goresky and Paul Siegel, \emph{Linking pairings on singular spaces},
  Comment. Math. Helv. \textbf{58} (1983), no.~1, 96--110.

\bibitem{HS91}
Nathan Habegger and Leslie Saper, \emph{Intersection cohomology of cs-spaces
  and {Z}eeman's filtration}, Invent. Math. \textbf{105} (1991), no.~2,
  247--272.

\bibitem{Ju09}
Daniel Juteau, \emph{Decomposition numbers for perverse sheaves}, Ann. Inst.
  Fourier (Grenoble) \textbf{59} (2009), no.~3, 1177--1229.

\bibitem{KS}
Masaki Kashiwara and Pierre Schapira, \emph{Sheaves on manifolds},
  Springer-Verlag, Berlin, 1994.

\bibitem{Kiehl}
Reinhardt Kiehl and Rainer Weissauer, \emph{Weil conjectures, perverse sheaves
  and {$l$}'adic {F}ourier transform}, Springer-Verlag, Berlin, 2001.

\bibitem{LANG}
Serge Lang, \emph{Algebra}, third ed., Graduate Texts in Mathematics, vol. 211,
  Springer-Verlag, New York, 2002.

\bibitem{Mit65}
Barry Mitchell, \emph{Theory of categories}, Pure and Applied Mathematics, Vol.
  XVII, Academic Press, New York-London, 1965.

\bibitem{MK2}
James~R. Munkres, \emph{Topology: a first course}, Prentice-Hall, Inc.,
  Englewood Cliffs, N.J., 1975.

\bibitem{Mur06}
Daniel Murfet, \emph{Diagram chasing in abelian categories},
  \url{http://therisingsea.org/notes/DiagramChasingInAbelianCategories.pdf}.

\bibitem{Pa90}
William~L. Pardon, \emph{Intersection homology {P}oincar\'e spaces and the
  characteristic variety theorem}, Comment. Math. Helv. \textbf{65} (1990),
  no.~2, 198--233.

\bibitem{Pop73}
N.~Popescu, \emph{Abelian categories with applications to rings and modules},
  Academic Press, London-New York, 1973, London Mathematical Society
  Monographs, No. 3.

\bibitem{Sch00}
Jean-Pierre Schneiders, \emph{Introduction to characteristic classes and index
  theory}, Textos de matematica, vol.~13, Faculdade de Ci\^{e}ncias da
  Universidade de Lisboa, 2000.

\end{thebibliography}

\providecommand{\bysame}{\leavevmode\hbox to3em{\hrulefill}\thinspace}
\providecommand{\MR}{\relax\ifhmode\unskip\space\fi MR }
\providecommand{\MRhref}[2]{%
  \href{http://www.ams.org/mathscinet-getitem?mr=#1}{#2}
}
\providecommand{\href}[2]{#2}

Some diagrams in this paper were typeset using the \TeX\, commutative
diagrams package by Paul Taylor.

\end{document}